\documentclass[a4paper,UKenglish,cleveref, autoref, thm-restate,anonymous,appendix]{article}
\newif\ifflag
\usepackage[affil-it]{authblk}

\flagtrue

\pdfoutput=1

\usepackage{graphicx}
\usepackage{amsmath,amsthm,amssymb,mathtools,amsfonts}
\usepackage{thmtools,thm-restate}
\usepackage{enumitem}
\usepackage{mathptmx}
\newtheorem{lemma}{Lemma}
\newtheorem{corollary}{Corollary}
\newtheorem{proposition}{Proposition}
\newtheorem{theorem}{Theorem}
\newtheorem{remark}{Remark}
\newtheorem{definition}{Definition}
\newtheorem{assumption}{Assumption}
\usepackage[hidelinks,bookmarksopen,bookmarksdepth=2]{hyperref}

\newcommand{\PP}{\Pr}
\newcommand{\EE}{\Exp}

\newcommand{\Dh}{\hat{D}}

\newcommand{\eps}{\varepsilon}
\newcommand{\sign}{\operatorname{sign}}
\newcommand{\Law}{\mathrm{Law}}
\newcommand{\Support}{\mathrm{Support}}

\newcommand{\R}{\mathbb{R}}
\newcommand{\Z}{\mathbb{Z}}

\newcommand{\us}{u^*}
\newcommand{\ust}{u^{*\T}}
\DeclareMathOperator{\Var}{Var}
\DeclareMathOperator{\Tr}{Tr}
\DeclareMathOperator{\val}{val}
\DeclareMathOperator{\Prob}{\Pr}
\DeclareMathOperator{\poly}{\mathrm{poly}}
\DeclareMathOperator{\Exp}{\mathbb{E}}
\DeclareMathOperator{\Cov}{Cov}
\newcommand{\E}{\Exp}
\newcommand{\T}{\mathsf{T}}

\usepackage{cleveref}[capitalise]
\crefname{lemma}{Lemma}{Lemmas}
\crefname{claim}{Claim}{Claims}
\crefname{theorem}{Theorem}{Theorems}
\crefname{section}{Section}{Sections}
\crefname{proposition}{Proposition}{Propositions}
\renewcommand{\phi}{\varphi}

\newcommand{\inner}[1]{\langle{#1}\rangle}

\newcommand{\floor}[1]{\lfloor{#1}\rfloor}
\newcommand{\ceil}[1]{\lceil{#1}\rceil}

\newif\ifnotes\notestrue
\ifnotes
 \usepackage{color}
 \definecolor{mygrey}{gray}{0.50}
 \newcommand{\notenamem}[2]{{\textcolor{mygrey}{\footnotesize{\bf (#1:} {#2}{\bf ) }}}}

 \else

 \newcommand{\notenamem}[2]{{}}

 \fi

\title{Integrality Gaps for Random Integer Programs via Discrepancy} \author[1]{Sander Borst\thanks{This project has received funding from the European Research Council (ERC) under the European Union's Horizon 2020 research and innovation programme (grant agreement QIP--805241).}}
\author[1]{Daniel Dadush\protect\footnotemark[1]}
\author[2]{
 	Dan Mikulincer\thanks{Dan Mikulincer was partially supported by a European Research Council grant no. 803084.}}
\affil[1]{Centrum Wiskunde \& Informatica (CWI), Amsterdam,~The~Netherlands\\
 	 \protect\\ \texttt{\{sander.borst,dadush\}@cwi.nl}}
 \affil[2]{MIT, Cambridge, USA\protect\\
	\texttt{danmiku@mit.edu}}
\usepackage[
backend=bibtex8,
natbib=true,
style=numeric,
isbn=false,
maxnames=99,
minnames=3,
firstinits=true, terseinits=true, sorting=none]{biblatex}
\bibliography{bibliography}

\begin{document}

\maketitle
\begin{abstract}
In this work, we prove new bounds on the additive gap between the value of a
random integer program $\max c^\T x,\ Ax \leq b,\ x \in \{0,1\}^n$ with $m$
constraints and that of its linear programming relaxation for a wide range of
distributions on $(A,b,c)$. Our investigation is motivated by the work of
Dey, Dubey, and Molinaro (SODA '21), who gave a framework for relating the
size of Branch-and-Bound (B\&B) trees to additive integrality gaps.     

Dyer and Frieze (MOR '89) and Borst et al.\ (Mathematical Programming '22), respectively, showed
that for certain random packing and Gaussian IPs, where the entries of $A,c$
are independently distributed according to either the uniform distribution on
$[0,1]$ or the Gaussian distribution $\mathcal{N}(0,1)$, the integrality gap
is bounded by $O_m(\log^2 n / n)$ with probability at least
$1-1/n-e^{-\Omega_m(1)}$. In this paper, we generalize these
results to the cases where the entries of $A$ are uniformly distributed
on an integer interval (e.g., entries in $\{-1,0,1\}$), and where the columns
of $A$ are distributed according to an isotropic logconcave
distribution. Second, we substantially improve the success probability to
$1-1/\poly(n)$, compared to constant probability in prior works
(depending on $m$). Leveraging the connection to Branch-and-Bound, our gap
results imply that for these IPs B\&B trees have size $n^{\poly(m)}$ with high probability (i.e., polynomial for fixed $m$), which significantly extends
the class of IPs for which B\&B is known to be polynomial.        

Our main technical contribution and the key to achieving the above results is
a new linear discrepancy theorem for random matrices. Our theorem gives
general conditions under which a target vector is equal to or very close to a
$\{0,1\}$ combination of the columns of a random matrix $A$. Compared to
prior results, our theorem handles a much wider range of distributions on
$A$, both continuous and discrete, and achieves success probability
exponentially close to $1$, as opposed to the constant probability shown in
earlier results. Our proof uses a Fourier analytic approach, building on the
work of Hoberg and Rothvoss (SODA '19) and Franks and Saks (RSA '20) who
studied the discrepancy of random set systems and matrices respectively.
\end{abstract}

\section{Introduction}

Consider an integer program (IP) in $n$ variables and a fixed number $m$ of
constraints of the form:
\begin{align*}
	\val_\mathsf{IP}(A,b,c) := \max_x         & \quad c^\T x \\
        \text{s.t. } & Ax\leq b,\ x \in \{0,1\}^n.             \tag{Primal IP}\label{primal-ip}
\end{align*}
An important factor controlling our ability to solve IPs is the tightness of
the linear programming (LP) relaxation. A natural way to measure tightness is
the size of the gap 
\[
\mathsf{IPGAP}(A,b,c):=\val_\mathsf{LP}(A,b,c)-\val_\mathsf{IP}(A,b,c),
\]
where $\val_{\mathsf{LP}}(A,b,c)$ relaxes $x \in \{0,1\}^n$ to $x \in
[0,1]^n$. 

In practice, automated methods for tightening LP relaxations such as
presolving and cutting planes are crucial to the performance of modern IP
solvers~\cite{bixby2007progress,achterberg2013mixed}. Presolving refers to
simple inference rules applied to constraints in sequence, which among other
things are used to find implied variable fixings, tighten variable
bounds and strengthen the coefficients of
inequalities~\cite{savelsbergh1994preprocessing}. Cutting planes refers to
additional valid linear inequalities for the IP, which are often generated
from the optimal simplex tableau~\cite{conforti2010polyhedral}. The
effectiveness of cutting planes is very often measured in terms of the
fraction of the integrality gap they close, which helps justify the
integrality gap as a key metric in
practice~\cite{balas1996mixed,fischetti2013approximating,fischetti2016improving}.    

\subsection{Complexity of Branch-and-Bound} 

Recently, Dey, Dubey, and Molinaro~\cite{dey_branch-and-bound_2021} provided
theoretical evidence for the importance of the gap in IP solving, in the context
of solving \emph{random packing} IPs via Branch-and-Bound (B\&B). We state a
generalization of their result, from~\cite{BDHT22}, to the case of random
logconcave IPs: 

\begin{theorem}[\cite{dey_branch-and-bound_2021,BDHT22}]
\label{thm:meta-logcon} 
Let $n \geq \Omega(m)$, $b \in \R^m$, and $\begin{pmatrix} c \\ A
\end{pmatrix} \in \R^{(m+1) \times n}$ be a matrix whose columns are
independent logconcave random vectors with identity covariance. Then, for $G
\geq 0$, with probability at least $1-\Pr_{A,c}[{\rm IPGAP}(A,b,c) \geq
G]-1/\poly(n)$, the best bound first Branch-and-Bound algorithm applied
to~\eqref{primal-ip} produces a tree of size at most $n^{O(m)}
e^{2\sqrt{2nG}}$.
\end{theorem}

We recall that B\&B proceeds by building a binary search tree whose leaves
represent a partition of the solution space. Each node of the tree is
associated with a subinstance, corresponding to a partial fixing of the
variables to $0$ or $1$, and is labelled with the optimal solution to the LP
relaxation consistent with this fixing. The tree is initialized with a single
node corresponding to the empty fixing which is labelled with the optimal
solution to $\val_\mathsf{LP}(A,b,c)$. B\&B keeps track of the best integer
solution found in the tree, picked from the nodes whose LP solutions are
integral (or possibly found via heuristics), and removes from consideration
all leaf nodes whose LP values are at most the value of this best found
solution (a procedure known as node pruning). At every iteration, B\&B chooses an unpruned
leaf node in the current tree to process using a \emph{node selection rule}.
To process the node, B\&B examines the optimal LP solution $x$ at the node and
chooses a fractional coordinate $x_i \in (0,1)$ using a \emph{branching
rule}. B\&B then creates two children, each one corresponding to fixing $x_i$ to $0$
or $1$, computes the corresponding LP solutions and checks whether these
nodes can be pruned. Once there are no more nodes left to process, the best
found integer solution is the optimal solution to the IP. For more background on B\&B, the interested reader may consult the following standard textbooks~\cite{wolsey1999integer,conforti2014integer}.   

As B\&B is still the backbone of all modern IP solvers such
as CPLEX, Gurobi, SCIP and Mosek, it is natural to try and understand under
which circumstances the complexity of B\&B can be meaningfully upper bounded.
\Cref{thm:meta-logcon} above gives probabilistic upper bounds on the size of
B\&B trees for random logconcave IPs in terms of the integrality gap, where
the branching rule can be arbitrary as long as 'best bound first' node
selection rule is used. Best bound first corresponds to always processing
leaf nodes with the largest LP value first. This rule ensures that we never
process a node whose value is worse than that of the optimal IP solution,
which is crucial for relating the size of the tree to the integrality gap.    

While \Cref{thm:meta-logcon} is relatively general, it is of course only
meaningful for distributions over which strong upper bounds on the
integrality gap are known. Note that the class of logconcave distributions is
very rich and includes, among others, the uniform distribution on any convex
body, and unbounded distributions such the Gaussian and Laplace
distributions. In terms of applications, due to the $n^{O(m)}$ term in the
complexity estimate, they are mainly restricted to the case where the number
of constraints $m$ is constant, which includes fundamental problems such as
knapsack. We discuss these applications together with our contributions in
the next section.   

\subsection{Integrality Gap Bounds and Their Applications} 

For models directly captured by \Cref{thm:meta-logcon}, strong integrality
have been proven for random packing~\cite{dyer_probabilistic_1989,
goldberg_finding_1984,lueker_average_1982} and Gaussian IPs~\cite{BDHT22},
where the entries of $(A,c)$ are either independent uniform $[0,1]$ or
$\mathcal{N}(0,1)$. For random packing IPs, when $b \in (n/4,n/2)^m$, Dyer and
Frieze~\cite{dyer_probabilistic_1989} proved that $\mathsf{IPGAP}(A,b,c) \leq
2^{O(m)} \log^2(n)/n$ with probability at least $1-1/\poly(n)-2^{-\poly(m)}$.
For Gaussian IPs, when $\|b^-\|_2 \leq n/10$ ($b^-$ is the negative part of $b$), Borst et al~\cite{BDHT22}
proved that $\mathsf{IPGAP}(A,b,c) \leq \poly(m) \log^2(n)/n$ with
probability at least $1-1/\poly(n)-2^{-\poly(m)}$. For these models,
\Cref{thm:meta-logcon} implies that B\&B is polynomial for fixed $m$ with
good probability. More precisely, letting $G = O_m(\log^2 n/n)$ be the gap
for these models, the complexity of B\&B grows as $e^{\sqrt{n m G}} =
e^{O_m(\log n)} = n^{O_m(1)}$. 

We remark that the works~\cite{goldberg_finding_1984,dyer_probabilistic_1989}
also introduced a specialized algorithm for solving random packing IPs, with
the same complexity\footnote{\cite{goldberg_finding_1984} examined the case
$m=1$ while~\cite{dyer_probabilistic_1989} handled general $m$.}. At a high
level, this algorithm simply examines all possible ways to round the optimal
LP solution without violating the target gap bound. It is worthwhile to note
that \cite{dey_branch-and-bound_2021} (implicitly) bounds the complexity of
B\&B by showing that its work never exceed the work of the algorithm in~\cite{goldberg_finding_1984,dyer_probabilistic_1989}.

The above results yield some of the very few cases where meaningful estimates
are known on the size of B\&B trees (additional examples will be
discussed in the related works section). Motivated by the scarcity of results
on B\&B, the main focus of this work is to extend the range of distributions
over $(A,b,c)$ for which strong gap upper bounds hold.

Our first result is a substantial generalization of the Gaussian gap bound of
Borst et al~\cite{BDHT22}. The result applies to a wide range of constraint
matrices whose columns are independent mean-zero random vectors, which we
refer to as \emph{Centered IPs}. 

\begin{theorem}[Gap Bound for Centered IPs]\label{thm:ipgap_centered}
For $m \geq 1$, $n \geq \poly(m)$, $b \in \R^m$ with $\|b^-\|_2 \leq O(n)$,
if $c$ has i.i.d.~$\mathcal{N}(0,1)$ entries and the columns of $A$ are
independent isotropic, logconcave random vectors whose
support is contained in a ball of radius $O(\sqrt{\log n} + \sqrt{m})$, then
\begin{align*}
    \Prob\left(\mathsf{IPGAP}(A,b,c) \geq  \frac{\poly(m)(\log n)^2}{n}\right)\leq n^{-\poly(m)}.
\end{align*}
Furthermore, the same result holds if the entries of $A$ are distributed
independently and uniformly in $\{0,\pm 1, \dots, \pm k\}$ and $b \in \Z^m$
with $\|b^-\|_2 \leq O(k n)$, for any $1 \leq k \leq n$. 
\end{theorem}
We recall that a random vector $X \in \R^m$ is
isotropic if $\E[X]=0$ (mean zero) and $\E[(X-\E[X])(X-\E[X])^\T]=\mathrm{I}_m$ (identity
covariance).  Since one can apply an affine transformation to any full dimensional
random variable to make it isotropic, this condition may be regarded as a useful normalization. When $A$ has iid $\mathcal{N}(0,1)$ entries, we
remark that the columns of $A$ have norm bounded by $O(\sqrt{m}+\sqrt{\log
n})$ with probability $1-1/\poly(n)$ by standard tail bounds. Thus, the above
theorem yields a strict generalization of~\cite{BDHT22}.

Our second result is a discrete variant of the random packing IP gap bound
of~\cite{dyer_probabilistic_1989}.   
\begin{theorem}[Discrete Packing IPs]\label{thm:ipgap_packing}
For $m \geq 1$, $k \geq 3$, $\beta \in (0,1/4)$, $n \geq \poly(m,k)\exp(\Omega(1/\beta))$, $b \in ((kn \beta,kn(1/2-\beta)) \cap \Z)^m$, if $c$ has i.i.d.
exponential entries and the entries of $A$ are independent and uniform in $\{1,\dots,k\}$,
then
\begin{align*}
\Prob\left(\mathsf{IPGAP}(A,b,c) \geq  \frac{\exp(O(1/\beta)) \poly(m)(\log n)^2}{n}\right) \leq n^{-\poly(m)}.
\end{align*}
\end{theorem}
Compared to~\cite{dyer_probabilistic_1989} we require exponentially
distributed objective coefficients, i.e., with density $e^{-x}$, $x \geq 0$,
instead of uniform $[0,1]$. The additional smoothness of the distribution
makes the arguments much cleaner while preserving the qualitative nature of
the results. We remark that extending the above gap bound to the setting of
non-negative logconcave random vectors is also possible, though we we do not
pursue this here.  

As a consequence of these gap bounds, we derive the following generalized
complexity bounds for B\&B.
\begin{corollary}\label{cor:branch-and-bound}
With probability $1-1/\poly(n)$, the best-bound first Branch-and-Bound
algorithm applied to~\eqref{primal-ip} produces a tree of size at most
$n^{\poly(m)}$ in the Centered IP model of Theorem~\ref{thm:ipgap_centered}
and of size $n^{\exp(O(1/\beta)) \poly(m)}$ in the Discrete Packing IP model from Theorem~\ref{thm:ipgap_packing}.
\end{corollary} 
In the discrete case, we note that the above result does not follow directly
from \Cref{thm:meta-logcon}, since it requires the constraint matrix to have
a logconcave distribution (a suitable adaptation of the proof still applies
however, see \Cref{sec:bnb}).

The above results extend our understanding of gap bounds in two significant
ways. Firstly, they show that strong gap bounds are obtainable when the
entries of the constraint matrix take only small integer values. This is well
motivated by the fact that the constraints for practical IPs are often
combinatorial and thus are expressed using small integer coefficients.
Secondly, in the case of centered IPs, we show that non-trivial correlations
between the variables in a column of the constraint (induced by logconcavity)
can be handled. Thus, we establish a limited form of universality for these
gap bounds. 

We would like to stress that neither of these extensions is a priori obvious.
To highlight the subtleties in the discrete setting, we note that even for
$m=1$, the gap bound does not hold when the right hand side $b$ is
non-integral. As a simple example, it is easy to see that for $n$ odd, the
integrality gap of the knapsack program $\max \sum_{i=1}^n x_i$ subject to
$\sum_{i=1}^n x_i \leq n/2,\ x \in \{0,1\}^n$ is $1/2$. It is not hard to check
that a constant gap is preserved with overwhelming probability even if the
profit vector $c$ is exponentially distributed and weights $a_1,\dots,a_n$
are drawn uniformly from $\{1,\dots,k\}$, for $k$ fixed. 

To circumvent the above issue, a crucial step in our proof in the discrete
setting is to show that we can round the optimal LP solution $x^*$ to an
integer solution $x'$ that is tight on the same set of constraints. That is,
$(Ax^*)_i = b_i \Rightarrow (Ax')_i = b_i, \forall i \in [m]$ (which is only
possible if $b$ is integral). Indeed, one of our main technical
contributions, described in the next section, is to give general conditions
under which such exact roundings are possible, using Fourier analytic
techniques from discrepancy theory. We remark that Dyer and
Frieze~\cite{dyer_gap_1992} believed that a discrete extension of their gap
bounds should be possible provided the range of the discrete distribution (i.e., $k$ above)
grows with $n$, which we confirm here without this assumption\footnote{We note that our gap bounds seemingly require that $n$ be somewhat larger than the range $k$. This is an artificial restriction. For (very) large $k$, one must treat the discrete distribution as if it were continuous, which requires a slight adaptation of the proofs.}.

Beyond generalizing previous work, our results also improve the probabilistic guarantees.
Compared to prior works, our gap bounds hold with high probability
$1-1/\poly(n)$ instead of $1-1/\poly(n)-2^{-\poly(m)}$ (which is constant for
fixed $m$). Achieving a high probability bound via prior estimates increases
the gap by a $O(\log n)$ factor, which makes the corresponding B\&B
complexity quasi-polynomial instead of polynomial.

Is it natural to ask whether there are natural limitations to extending these
gap bounds to much wider classes of IPs. For logconcave random IPs, one may observe that any worst-case instance can be encoded as the mean
$\E[(A,b,c)]$ of the random instance. In this way, one can view the
distribution $(A,b,c)$ as a smoothed version of worst-case IP instance
$\E[(A,b,c)]$. By appropriately scaling up the means, or equivalently scaling
down the variance of the entries of $(A,b,c)$, the the instance $(A,b,c)$
converges to the worst-case instance $\E[(A,b,c)]$. It is not hard to check
that for many hard combinatorial problems such as SET COVER or MAX 3-SAT,
adding small random perturbations (of appropriate sign) to the instance data
does not change the optimal solution, and thus we cannot expect strong
integrality gap results in these settings.   

We note that smoothed analysis of integer programming has indeed been studied
by various works \cite{beier_core_2004,Rglin2007}. In particular, R\"oglin
and V\"ocking~\cite{Rglin2007} showed that under mild conditions, the class
of IPs that are easy to solve on average over suitable random
perturbations of the instance data (i.e., that can be solved in polynomial
time with high probability over the perturbations) correspond exactly to the
class of IPs solvable in pseudopolynomial time.

\paragraph{\bf Gap Bounds in Related Settings.}

Strong gap bounds have also been proven for random instances of combinatorial
optimization problems, though this has not translated into good upper bounds
on the size of B\&B trees. Dyer and Frieze~\cite{dyer_gap_1992}, proved an
$O_m(\log^2n/n)$ gap bound for random instances of a generalized
assignment problem on a bipartite graph with $m$ vertices on the small side
and $n$ vertices on the big side. They further gave an $n^{O_m(1)}$ algorithm
to solve this problem with high probability. Frieze and
Sorkin~\cite{FriezeSorking07} showed that the cycle cover relaxation for the
asymmetric TSP has an expected $O(\log^2 n / n)$ additive integrality gap,
where the edge weights are chosen uniformly from $[0,1]$ for a complete
digraph on $n$ vertices and gave an $2^{\tilde{O}(\sqrt{n})}$ algorithm which
solves these instances with high probability. Recently,
Frieze~\cite{Frieze20} showed that any B\&B tree using the cycle cover
relaxation to prune nodes\footnote{The cycle cover relaxation is, in fact,
integral, but its solutions are not feasible (i.e., they may be a union of
many disjoint directed cycles). B\&B must therefore branch on variables that
are integral in the current relaxation.} will have expected size
$\Omega(2^{n^{\alpha}})$, for some $0 < \alpha < 1/20$. It is an interesting
open question to understand if B\&B can recover the same running times as the
specialized algorithms above for these problems.

\subsection{Techniques}

We now explain the high level strategy behind the obtaining integrality gap
bounds of~\cite{dyer_probabilistic_1989,BDHT22}. We will then extract a
natural discrepancy theoretic problem which will be the key to our
generalizations. 

\paragraph{\bf Proving Gap Bounds.} Given integer program indexed by
$(A,b,c)$, we recall that the goal is to bound the gap between the value of the integer
program and its linear programming relaxation. The strategy to bound the
integrality gap is to suitably round an optimal basic primal solution $x^*
\in [0,1]^n$ (for possibly a slightly perturbed right-hand side $b$) to a
nearly optimal solution $x'' \in \{0,1\}^n$ to the IP. The rounding procedure
will make use of the dual optimal solution $\us \in \R^m_+$ (see
\Cref{sec:lp-prelim} for the dual program). 

To round $x^*$, we first round down its fractional components to $0$ to get
$x'$ (there are at most $m$ such components). From here, one selects a small subset of variables $T \subseteq \{i:
x^*_i = 0\}$, $|T| = O_m(\log n)$, with tiny reduced costs (which are
``cheap'' to flip to $1$), namely $|c_i - \ust a_i| = O_m(\log n/n)$, for $i \in T$.
The subset $T$ is carefully chosen such that $x'' := x'+1_T \in \{0,1\}^n$,
obtained by flipping the variables in $T$ from $0$ to $1$, is the desired
feasible and nearly optimal IP solution. For feasibility, we need that
$A(x'+1_T) \leq b$. For near-optimality, we additionally need $x'+1_T$ to be nearly (or even exactly) tight on the
constraints that $x^*$ is tight on. That is, we need that $(Ax^*)_i = b_i
\Rightarrow (A(x'+1_T))_i \approx b_i,\ \forall i \in [m]$. Simplifying
slightly, a sufficient condition is $A 1_T \approx A(x^*-x')$,
thus we may think of $T$ as a ``slack repairing'' set.   

A crucial property, in both the packing and Gaussian setting, is that $x^*$ has
at least $\Omega(n)$ zero entries and that the corresponding columns of $A$ are
independent subject to having negative reduced cost (see
\Cref{lem:conditionaldist} and \Cref{lem:solution-props}). Furthermore, there
are at least $\Omega_m(\log n)$ columns with suitably small reduced costs, and so there is large set of columns from which to select $T$.    

To generalize the above strategy to other distributions, the crucial
difficulty is understanding under what conditions the ``slack repairing''
set $T$ above exists.   

\paragraph*{\bf The Discrepancy Problem.} Stated abstractly, the existence of
$T$ can be phrased as a natural discrepancy theoretic problem. Let $t \in
\R^m$ be a target vector (e.g., the difference of slack vectors) and let $A
\in \R^{m \times \bar{n}}$ be a ``nice'' random matrix with independent
columns. When can we ensure with high probability that $t$ is equal to or
very close to a $\{0,1\}$ combination of the columns of $A$? 

A general answer to this question was given in~\cite[Lemma 1]{BDHT22},
improving upon~\cite[Lemma 3.4]{dyer_probabilistic_1989}. However, it
enforced very strict conditions on the entries of $A$. Specifically, the
entries of $A$ needed to be independent, mean zero with unit variance,
absolutely continuous random variables of bounded density which ``converge
quickly enough'' to a Gaussian when averaged. Furthermore, for the targets
$t$ in the ``range'' of $A$, the probability of successfully hitting $t$ was
only $\Theta(1)$. 

As our main technical contribution, which yields the key ingredient for
extending the IP gap bounds, we give a much more general and powerful
discrepancy theorem. We state it below, restricted to special cases relevant
for our applications (see \Cref{thm:main} for the general result). 

\begin{theorem}[Linear Discrepancy Theorem for Random Matrices]
\label{thm:disc-concrete}
\phantom{ }
Let $A = (a_1,\dots,a_{\bar{n}}) \in \R^{m \times \bar{n}}$, $\bar{n}
\geq \poly(m)$, be a random matrix with independent columns with same mean
$\mu \in \R^m$. Let $p \in (0,1)$ satisfy $\frac{\poly(m)\log \bar{n}}{\bar{n}} \leq p \leq \frac{1}{\poly(m)}$. Then, the following holds:
\begin{itemize}
\item Discrete case: suppose that the entries $A_{ij}$ are i.i.d. uniform on an integer interval $\{l,l+1,\dots,l+k\}$ for $l,k \in \Z$, $\bar{n} \geq k \geq \max\{2,|l|\}$. Then, with probability at least
$1-e^{-\Omega(p \bar{n})}$, for every $t \in \Z^m$ satisfying
$$\|t-p\bar{n}\mu\|_2 \leq O(\frac{k\sqrt{p\bar{n}}}{m \log m})$$ there
exists $x \in \{0,1\}^{\bar{n}}$ with $\|x\|_1 = \Theta(p\bar{n})$ such that
$Ax = t$.
\vspace{1em}  
  
\item Continuous case: suppose that $a_1,\dots,a_{\bar{n}}$ are
logconcave with identity covariance and that $\|\mu\| \leq \poly(m)$. Then,
with probability at least $1-e^{-\Omega(p\bar{n})}$, for every $t \in \R^m$
satisfying, $$\|t-p\bar{n}\mu\|_2 \leq O(\frac{\sqrt{p\bar{n}}}{m \log m}),$$
there exists $x \in \{0,1\}^{\bar{n}}$, with $\|x\|_1 = \Theta(p\bar{n})$
such that $\|t-Ax\|_2 \leq e^{-\Omega(p\bar{n}/m)}$.    
\end{itemize}
\end{theorem}

When applying the above to find the slack repairing set $T$, $\bar{n}$ will
roughly be $O(\poly(m) \log n)$ and $p = 1/\poly(m)$. The corresponding
success probability will now be $1-n^{-\poly(m)}$ as opposed to $\Theta(1)$,
which is what allows us to much better tail bounds for the integrality gap.

\paragraph*{\bf Gap bounds via Discrepancy.} 

Given the discrepancy theorem, the proof of the gap bounds mirrors the proofs
in~\cite{dyer_probabilistic_1989,BDHT22} as described in the last section,
though with non-trivial technical adaptations as well as some
simplifications. 

As in prior work, the relevant properties of the optimal primal $x^*$ and
dual solution $\us$ must be established in these generalized settings (see
\Cref{lem:solution-props} and \Cref{lem:solution-props-packing}). In
particular, one must show that the norm of $\us$ is suitably bounded, and
that there are sufficiently many columns of $A$ with small reduced costs
indexed by the zeros of $x^*$. These properties can be derived using standard
tools for concentration and anti-concentration of logconcave and uniform
discrete random variables.

Providing ``nice enough'' columns for our discrepancy theorem can however be
challenging. For centered IP where $A$ has independent logconcave
columns, conditioning on the size of the reduced costs can significantly
perturb the mean of each column (possibly in different directions when the
columns are not i.i.d.). We overcome this problem using sophisticated forms
of rejection sampling, which allows to ``virtually recenter'' each logconcave
column (see \Cref{lem:sample-mean-zero}). Interestingly, the rejection
sampling procedure does not even preserve logconcavity, however the
properties required for our more general discrepancy theorem persist (see
\Cref{thm:main}).    
  
In terms of simplifications, compared to earlier proofs we no longer
require repeated trials on disjoint subsets of columns of $A$ to find a
suitable slack repairing set $T$. In particular, due to our new discrepancy
theorem, using all the small reduced cost columns together both
exponentially decreases the probability of failure and increases the size of
the targets one can hit. Furthermore, since the discrepancy theorem directly
applies to columns with non-zero means, one can avoid ad-hoc reductions to
the mean-zero case as in the proof of~\cite{dyer_probabilistic_1989} for the
packing case. 

\paragraph{\bf Relations to Discrepancy Theory.}

We first explain the connection to linear discrepancy. As defined by
Lov{\'a}sz, Spencer, Vesztergombi~\cite{lovasz1986discrepancy}, the linear
discrepancy of a matrix $A \in \R^{m \times n}$ is ${\rm lindisc}(A) :=
\max_{\lambda \in [0,1]^n} \min_{x \in \{0,1\}^n} \|A(x-\lambda)\|_\infty$.
That is, it is the maximum ``rounding error'' one must incur to round a
$[0,1]$ combination of the columns to a $\{0,1\}$ combination (the specific
choice of $\ell_\infty$ vs $\ell_2$ norm is not important in our context).
The discrepancy ${\rm disc}(A)$ of $A$ is linear discrepancy restricted to
$\lambda = {\bf 1}_n/2$ (the all $1/2$ vector). It is more common to
expressed discrepancy by $\min_{x \in \{-1,1\}^n} \|Ax\|_\infty = 2 {\rm
disc}(A)$, in which case $x \in \{-1,1\}^n$ is interpreted as a 2-coloring of
the columns of $A$. While linear discrepancy is always larger than
discrepancy,~\cite{lovasz1986discrepancy} showed the maximum discrepancy of
any subset of the columns of $A$, known as hereditary discrepancy, upper
bounds ${\rm lindisc}(A)$ up to a factor of $2$.
 
Bounds on the discrepancy of various matrix classes, often induced by the incidence matrix of set
system, have found many applications in computational geometry and
complexity (see~\cite{matousek1999geometric,chazelle2000discrepancy}). Over
the last decade or so, efficient algorithms for producing low-discrepancy
colorings have been
developed~\cite{bansal2010constructive,lovett2015constructive,rothvoss2017constructive,eldan2014efficient,bansal2018gram}
and have found many applications in the context of approximation
algorithms~\cite{lau2011iterative,HR17,bansal2019generalization,bansal2022flowtime}.      

With the above perspective, \Cref{thm:disc-concrete} can be interpreted as
bounding the linear discrepancy of the random matrix $A$ for combinations
$\lambda \in [0,1]$ which are very close to $p{\bf1}_n$, where $p \in (0,1)$
is as above. As the columns of $A$ are independent with mean $\mu$, one can
expect that $A(p{\bf1}_n) \approx pn\mu$. Perhaps slightly less clear is that
every $t\in \R^m$, which is close to $pn\mu$, will in fact be exactly
expressible as $t = A \lambda$, where $\lambda \approx p {\bf1}_n$ with high
probability (this requires an analysis of the singular values of $A$).
\Cref{thm:disc-concrete} now shows that $\lambda$ can be replaced by $x \in
\{0,1\}^n$ with $\|x\|_1 = \Theta(pn)$ incurring either no error in the
discrete case (assuming $t \in \Z^n$) or exponentially small error in the
continuous case, thereby bounding the linear discrepancy of the combination
$\lambda$. At least in the continuous case, we note that one can in fact
adapt the proof of \Cref{thm:disc-concrete} to directly bound the linear
discrepancy of any ``reasonable'' combination $\lambda$ (i.e., without the
detour through the mean $\mu$). This would be somewhat less useful for the
integrality gap application we consider here, since it would be give
significantly less precise guarantees on the targets we can expect to hit.
Interestingly, while the application pursued here is algorithmic, i.e,
bounding the complexity of B\&B, we are not aware of any efficient algorithm
to compute the rounding $x$ guaranteed by \Cref{thm:disc-concrete}.    

Crucial to the linear discrepancy bounds we achieve in
\Cref{thm:disc-concrete}, i.e., either exponential small or zero, is that the
matrix $A$ has many more columns that the dimension $m$. In particular,
reducing to bounds on hereditary discrepancy become useless in this setting.
The study of discrepancy in the ``over complete'' setting has been become
very active somewhat more recently, with works focusing on the discrepancy of
large random matrices and set
systems~\cite{costello2009balancing,kuperberg2017probabilistic,ezra2019beck,hoberg2019fourier,FS20,potukuchi2018discrepancy,bansal2020discrepancy}.
Many of these works were motivated by the Beck-Fiala conjecture, which posits
that discrepancy of any $\{0,1\}$ matrix $A$ in which every column has at
most $t$ ones is bounded by $O(\sqrt{t})$ (here $A$ can be interpreted as the
incidence matrix of a set system). Variants of this conjecture for random set
systems and matrices where established
in~\cite{costello2009balancing,ezra2019beck,hoberg2019fourier,FS20,potukuchi2018discrepancy,macrury2021phase,bansal2020discrepancy},
where in the case $n \gg m$, it was shown that discrepancy quickly drops to
$1$~\cite{hoberg2019fourier,FS20,potukuchi2018discrepancy,macrury2021phase}
or even exponentially close to zero~\cite{costello2009balancing,FS20}
depending on whether the columns of the matrix $A$ are discretely or
continuously distributed.

\paragraph*{\bf Discrepancy via Fourier Analysis.}

To prove our linear discrepancy theorem, we rely on a Fourier analytic
approach, which is completely different from the second-moment counting based
proofs in~\cite{dyer_probabilistic_1989,BDHT22}. The high level approach was
pioneered by Kuperberg, Peled and Lovett~\cite{KLP17}, who applied Fourier
analytic techniques to show the existence of rigid combinatorial objects,
such as orthogonal arrays, Steiner systems and regular hypergraphs. At a
technical level, for a matrix $A \in \{0,1\}^{m \times n}$, they were
interested in understanding the minimum $p \in (0,1]$ such that $pA{\bf 1}_n
= Ax$, where $x \in \{0,1\}^n$ $\|x\|_1 = pn$. They examined this question
for highly symmetric deterministic matrices coming from the above
applications. The Fourier analytic approach was later applied by Hoberg and
Rothvoss~\cite{hoberg2019fourier} and independently by Frank and
Saks~\cite{FS20} to show very strong upper bounds on the discrepancy $\min_{x
\in \{-1,1\}^n} \|Ax\|_\infty$ of a random matrix $A$ when $n \gg m$, where
the columns of $A$ were drawn iid from various distributions.

From a comparative perspective, our \Cref{thm:disc-concrete} sits in between
the work of~\cite{KLP17} and \cite{hoberg2019fourier,FS20}. We work with
random matrices $A$ as in~\cite{hoberg2019fourier,FS20}, though the question
we attack is more in the spirit of~\cite{KLP17}. We note that unlike
\cite{KLP17}, it is not sufficient for us to show that the existence a
rounding $x \in \{0,1\}^n$ such that $A(x-p{\bf 1}_n)$ is small or zero. For
our applications, we require this to be true with $Ap{\bf 1}_n$ replaced by
$pn\mu = \E[Ap{\bf 1}_n]$. We further require the existence of $x$ to hold
for any target $t$ close enough to $pn\mu$. This last requirement however
generally comes for ``free'' with Fourier analytic techniques (the moment you
can hit $pn\mu$ you can hit everything close to it as well). A more
significant difficultly is that the concentration of $Ap{\bf 1}_n$ around its
mean $pn\mu$ is rather weak. That is, the probability that $Ap{\bf 1}_n$ is
close to its mean scales relative to the ambient dimension $m$ (which is
constant) instead of $n$. This makes achieving the $1-e^{-\Omega(pn)}$
success probability more challenging.   

We now explain the high-level approach. Let $A \in \R^{m \times n}$ be our
random matrix with columns having mean $\mu \in \R^m$ and covariance matrix
$\sigma^2 I_m$, and let $p \in (0,1)$ be our parameter. The strategy is to directly analyze the
probability mass function of the random variable $Y = AX$, where
$X_1,\dots,X_n$ are i.i.d.~Bernoulli's with probability $p$. \cite{KLP17}
also use this distribution, whereas~\cite{hoberg2019fourier,FS20} choose
$X_1,\dots,X_n$ to be uniform $\{-1,1\}$ random variables. Restricting
attention to the discrete case, where $Y \in \Z^n$, we show that $\Pr[Y=t]
\gg 0$ for $t \in \Z^n,$ when $\|np\mu-t\| = O_m(\sigma \sqrt{pn})$, where
$O_m(\sigma^2 p n)$ roughly measures the ``available variance'' of $AX$ in
all directions. Obtaining a bound for $\Pr[Y=t]$ is done by applying the
Fourier inversion formula, showing that the Fourier coefficients are close to
those of a Gaussian and integrating (see \Cref{sec:discrepancy} for an
overview). 

For $\theta \in [-1/2,1/2]^m$, the corresponding of Fourier coefficient $AX$
is expressed by $\E_X[e^{2\pi i \langle \theta, AX \rangle}]$. Similar
to~\cite{hoberg2019fourier,FS20}, we control the magnitude of these
coefficients using the anti-concentration properties of the columns of $A$.
Specifically, for our choice of column distributions, we need to show the the
probability that $\langle \theta, A_i \rangle$, $i \in [n]$, is ``close'' to
an integer decays predictably as a function of $\sigma \|\theta\|$. We note
that the exact expression for the Fourier coefficients indeed differs
depending on whether the entries of $X$ are Rademacher or Bernoulli
distributed, where the former is more prone to parity
issues (e.g., if $A \in \{0,1\}^{m \times n}$ and $x \in \{-1,1\}^n$, the parity of $Ax$ is fixed). Such parity issues do not arise in our setting. As mentioned above,
we face a different difficulty, which is being able to pinpoint the exact
targets whose probabilities we can accurate estimate due to the poor
concentration of $\E_X[AX] = A(p{\bf 1}_n)$ around $\E_{X,A}[AX] = np\mu$.     

To deal with this issue, we first carefully subsample a set $S \subseteq [n]$
of columns from $A$, whose sum is close to the mean, and then generate $Y$
from these subsampled columns. To construct $S$, we iterate through the
columns one by one, adding $A_i$ to $S$ if $\langle \sum_{j \in S} (A_j-\mu),
A_i-\mu \rangle \leq 0$ and $\|A_i-\mu\| \leq 2 \E[\|A_i-\mu\|_2^2]^{1/2}$.
This subsampling deterministically ensures that $\|\sum_{i \in S}
(A_i-\mu)\|_2 \leq 2 \sqrt{\sum_{i \in S} \E[\|A_i-\mu\|_2^2]}$, suitably
biasing the sum towards the mean. For the distributions we work with, it is
easy to show that $|S| = \Omega(n)$ with probability $1-e^{-\Omega(n)}$, so
we always have a constant fraction of the columns left over. Note that this
subsampling crucially uses our flexibility to drop columns of $A$, a
distinguishing feature of $\{0,1\}$ combinations versus $\{-1,1\}$
combinations. This subsampling process, however, causes non-trivial
dependencies among the columns of $A$. That is, the submatrix we use to
generate $Y$ no longer has independent columns. Fortunately, we show that
even with the conditioning induced by subsampling, the columns still retain
enough of their anti-concentration properties to allow the Fourier analytic
estimates to go through.

\paragraph*{\bf Related Work.} 
Another success story in analysis of LP based B\&B comes from  fixed
parameter tractability (FPT). In this context, it was shown that important
combinatorial problems such as vertex cover~\cite{lokshtanov2014faster} and
multiway cut~\cite{lokshtanov2014faster} are FPT when parameterized by the
integrality gap of the natural LP relaxation. Since then, multiple
general frameworks for capturing such problems have been
developed~\cite{iwata2016half,wahlstrom2017lp}. These results crucially rely
on the use of half-integral and persistent LP relaxations, which are known
only for very structured problems. The results presented here thus mainly
provide complementary evidence for the effectiveness of B\&B in the
unstructured setting. 
        
While variable based B\&B is most commonly used in practice, there is a
theoretical line of work examining the effectiveness of branching schemes for
solving random integer programs based on general integer disjunctions. These
schemes rely on sophisticated lattice basis reduction
methods~\cite{lenstra1982factoring} and follow the template of Lenstra's
celebrated polynomial time algorithm for integer programming in fixed
dimension~\cite{lenstra1983}. In particular, Furst and Kannan~\cite{FK89}
showed that certain random subset-sum instances can be solved in polynomial
time via basis reduction, and Pataki, Tural and Wong~\cite{pataki_basis_2010}
extended these results to certain IPs with multiple constraints. Apart from
the different type of branching, compared to the present work, the IPs
analyzed in these models are either infeasible or have a unique feasible
solution with high probability. 

\paragraph*{\bf Organization}
The preliminaries and necessary background to our results can be found in \cref{sec:prelim}. 
\cref{sec:discrepancy} contains the proof of the discrepancy result, \cref{thm:disc-concrete}, and the bounds on the integrality gap, given in \cref{thm:ipgap_centered,thm:ipgap_packing}, are proven in \cref{sec:gap-bounds}. \Cref{sec:bnb} is devoted to the proof of Corollary \ref{cor:branch-and-bound}. Finally, in \cref{sec:ac} we prove several anti-concentration results required by our method.

\newcommand{\rad}{\beta}
\section{Discrepancy}
\label{sec:discrepancy}

In this section we describe our main discrepancy results. Let $A \in \mathbb R^{n \times m}$ be a random matrix with independent columns. Throughout the section we assume that the columns of $A$ have the same mean, denoted as $\mu$. We also make the assumption that all columns of $A$ have the same period, as defined in \eqref{eq:dualfunddomain}, and denote it by $\mathrm{period}(A)$. We will deal with two different cases: either $\mathrm{period}(A) = \{0\}$, which means that the columns of $A$ are absolutely continuous, or $\mathrm{period}(A)=\Z^m$, in which case the columns of $A$ are supported in the lattice $\Z^m$.

\paragraph{Notation and assumptions:}
Before stating the result, let us introduce some definitions. We first describe the distributions captured by our result; distributions with a non-negligible mass on every half-space passing through their mean.

\begin{definition}[approximately symmetric distributions] \label{def:addmis}
A probability distribution $\mathcal{D}$ on $\R^m$, with mean $\mu$, is called \emph{approximately symmetric} if, for any $\nu \in \R^m$,
$$\Pr_{X \sim \mathcal{D}}\left(\langle X, \nu\rangle \geq \langle\mu,\nu\rangle \right) \geq \frac{1}{4e^2}.$$
\end{definition}	
\begin{remark} \label{rml:admis}
	The constant $\frac{1}{4e^2}$ in the definition is somewhat arbitrary and could be relaxed to any smaller constant. It is immediately clear that any distribution which is symmetric around its mean is also approximately symmetric. Moreover, Gr\"unbaum's inequality in Lemma \ref{lem:grunbaum} shows that logconcave measures are approximately symmetric.
\end{remark}
Suppose that the columns of $A$ are approximately symmetric. The main idea will be to choose $S\subseteq [n]$ randomly with $\Pr[i\in S]=p$ for $p \leq \frac{1}{\mathrm{poly}(m)}$, independently for all $i\in [n]$. We then show that, with high probability over $A$ and positive probability over $S$, $A\mathbf{1}_S$ is close to a target vector $t$. Thus, let us define the random vector $D:=A\mathbf{1}_S$. 

To understand the distribution of $D$ we will consider its Fourier transform, denoted $\Dh(\theta):=\Exp[\exp(2\pi i \langle D,\theta\rangle)].$ Note, that under the assumptions above, we have $\mathrm{period}(A) = \mathrm{period}(D)$, and so we shall use $V$ to denote the fundamental domain of $D$, as in \eqref{eq:funddomain}.

The next definition quantifies an appropriate notion of anti-concentration properties for the columns of $A$.

\begin{definition}[anti-concentration] \label{def:AC}
	Let $\sigma \geq 0$ and $\kappa \in (0,1)$. We say the measure $\mathcal{D}$ is $(\sigma,\kappa)$-anti-concentrated if $\mathrm{Cov}(\mathcal{D}) \preceq \sigma^2\mathrm{I}_m$, and for any $\nu \in \mathbb{R}^m$ and any $\theta \in V$,
	\begin{align*}
		\Pr_{X\sim \mathcal{D}}\left[d(\theta^\T X, \mathbb{Z}) \geq \kappa\min\left(1,\|\theta\|_\infty\sigma\right) \mid\langle\nu, X\rangle \leq \langle\nu,\mu \rangle \right]\geq \kappa \tag{anti-concentration}, \label{eq:anticoncentration} 
	\end{align*}
	where $d(\theta^\T X, \mathbb{Z}):= \inf\limits_{z \in \mathbb{Z}} |\theta^\T X-z|.$
	When $\sigma$ is clear from the context, we will sometimes omit the dependence on $\sigma$ from the definition.
\end{definition}
Without further details, the definition might seem opaque. Below we explain the rationale for considering this notion of anti-concentration and demonstrate some examples of distributions that satisfy \Cref{def:AC}. For now, it shall suffice to say that the \eqref{eq:anticoncentration} property appears naturally when trying to establish bounds on Fourier transforms.

In the following assumption, which we shall enforce throughout the section, we detail the possible dependencies between the different parameters. These dependencies, in turn, dictate the possible regimes in which our results hold.
\begin{assumption} \label{ass:technical}
	We say that the random matrix $A \in \R^{n\times m}$ satisfies our assumptions if its columns are independent with a common period and a common mean $\mu\in \R^m$, are approximately symmetric, and $(\sigma,\kappa)$-anti-concentrated, with constants $\sigma, \kappa > 0$. Further, let $p \in [0,1]$ and assume that the following technical condition is met:
	\begin{equation}
	\label{eq:p-bound}
	\frac{\left(10m\left(1+ \frac{\|\mu\|}{\sigma}\right)\right)^{20}}{\kappa^{16} n} \leq p \leq \frac{\kappa^8}{10^{15}m^6} .
	\end{equation}
\end{assumption}
\paragraph{Main discrepancy results:}
With the above notation, the main result of this section is:

\begin{theorem}
Let $p \in [0,1]$ and suppose that $A$ satisfies our assumptions, as in Assumption \ref{ass:technical}.
Then, when $\mathrm{period}(A) = \{0\}$, with probability $1-e^{-\Omega(\kappa pn)}$ we have: 

For all $t \in \R^m$ such that
$$\|t-pn\mu \|\leq \frac{\sigma \kappa^{\frac{3}{2}}}{10^4\ln\left(\frac{10^3m}{\kappa^3}\left(1 + \frac{\|\mu\|^2}{\sigma^2}\right)^{\frac{1}{m}}\right)}\frac{\sqrt{pn}}{m},$$ there is a set $S$ of size $|S|\in [\frac{1}{16}pn,\frac{3}{16}pn]$ such that $$\|A\mathbf{1}_S-t\| \leq \exp\left(-\frac{\kappa^3pn}{80m}\right)\sigma n^3.$$
\label{thm:main}
\end{theorem}

Theorem \ref{thm:main}, as stated, only deals with distributions which are absolutely continuous with respect to the Lebesgue measure. However, the argument also applies to measures with singularities. In particular, the result
 also holds for distributions supported on the lattice $\Z^m$, the case which is most relevant to our work. Moreover, in the lattice case, if one finds a subset $S \subset [n]$, such that $A{\bf 1}_S$ is very close to some target vector $t\in \R^m$, then, since $A{\bf 1}_S$ belongs to the lattice as well, one can actually deduce $A{\bf 1}_S = t$. We prove this statement as a part of the proof of Theorem \ref{thm:main}.
\begin{theorem}\label{cor:maindisc}
	Suppose that  $\mathrm{period}(A) = \mathbb{Z}^m$. Then, under the same conditions of Theorem \ref{thm:main} together with $\exp\left(\frac{\kappa^3pn}{80m}\right) \geq 2 \sigma^2 n^2$, with probability $1-e^{-\Omega(\kappa pn)}$, there is a set $S$ of size $|S|\in [\frac{1}{16}pn,\frac{3}{16}pn]$ such that $A\mathbf{1}_S = t$, provided that $t \in \Z^m$.
\end{theorem}
Remark that since, by \eqref{eq:p-bound}, $\kappa^3pn \geq \mathrm{poly}(m)$, the condition $\exp\left(\frac{\kappa^3pn}{80m}\right) \geq 2 \sigma^2 n^2$ almost does not restrict generality.
\paragraph{Road map for the proof:}
Recall Fourier's inversion formula, Theorem \ref{thm:fourier}. In light of the formula, it will be enough to show that the integral in Theorem \ref{thm:fourier} is positive for appropriate $t$, for most choices of $A$. In this case we will get that $\Pr\left[D = t\right] > 0$, which implies the existence of an appropriate subset of columns. At a very high level, we will show that most of the mass lies next to the origin and that the integrand has an exponential decay far from the origin. The following sequence of steps will achieve this:
\begin{enumerate}
	\item Our first step will be to subsample the columns of $A$. This will result in appropriate concentration bounds that improve as $n$ increases, irregardless of the value of $m$ (Lemma \ref{lem:subsample_props}).
	\item We then show that with high probability $\arg (\Dh(\theta))\in [-\frac18\pi ,\frac18\pi ]$ for all $\theta$ with $\|\theta\|=O(\frac{\poly(m)}{\sigma \sqrt{n}})$. This will allow to establish that the real part of $\Dh(\theta)\exp(-2\pi i \langle\theta, t\rangle)$ is large when $\|t-pn\mu\|\|\theta\|$ is small (Lemma \ref{lem:argbound}).
	\item Next, we show that $\int\limits_{\|\theta\|\leq r}|\Dh(\theta)|d\theta$ is large for some $r = O(\frac{\poly(m)}{\sigma\sqrt{ n}})$. In combination with \cref{lem:argbound}, we shall conclude that $\int\limits_{\|\theta\|\leq r}\Dh(\theta)\exp(-2\pi i \langle\theta, t\rangle)d\theta$ is large (Lemma \ref{lem:integral_lower_bound}).
	\item Finally, we will show that for $\|\theta\|\geq r$, $|\Dh(\theta)|$ is rapidly decreasing. So the integral over these $\theta$ can only have a small negative contribution (Lemma \ref{lem:normbound}).
\end{enumerate}

The first bottleneck is Step 1, which will change the distribution and effective size of the random matrix $A$. Below we show that approximately symmetric distributions maintain many desirable properties after our subsampling procedure.
The second bottleneck is Step 4. To establish a rapid enough decay of the Fourier spectrum, we require that the columns of $A$ satisfy the \eqref{eq:anticoncentration} property from \Cref{def:AC}.

To gain a bit of intuition about \Cref{def:AC}, recall that we are working in the Fourier domain. If $X$ is a column of $A$, it is natural to require that $\langle \theta, X\rangle$ be bounded away from integer points. Otherwise, $\langle\theta, D\rangle$ could be close to an integer point with high probability, making $|\Dh(\theta)|$ large. An extra component in the definition says that the anti-concentration continues to hold after conditioning on an arbitrary half-space, passing through the mean. As will become apparent, this is a consequence of the subsampling scheme from Step 1.

Let us just note that the \eqref{eq:anticoncentration} property is not vacuous. In fact Theorem \ref{thm:disc-concrete} is a direct consequence of the following lemma (see the proof in \cref{sec:ac}) and Theorems \ref{thm:main} and \ref{cor:maindisc} (note that by Remark \ref{rml:admis} both cases are approximately symmetric).
\begin{lemma}
	\label{lem:anti-concentration}
	Suppose that for $X = (X_1,\dots,X_m) \sim \mathcal{D}$, one of the following holds,
	\begin{enumerate}
		\item $X$ is logconcave and isotropic.
		\item $X_i$ are \emph{i.i.d.} uniformly on an integer interval $\{a,a+1,\dots, a+k\}$, with $k > 1$.
\end{enumerate}
	Then, $\mathcal{D}$ satisfies \eqref{eq:anticoncentration} with constant $\kappa >\frac{1}{50}$.
\end{lemma}
One may wonder about the necessity of the condition $k > 1$ in Case 2 of Lemma \ref{lem:anti-concentration}. A moment of reflection reveals that, if $X$ is uniform on $\{0,1\}$, then $X$ is not anti-concentrated, for any $\kappa > 0$, and thus our framework does not directly apply to this case. However, by taking account of the possible bad cases, our analysis can be refined to also handle such distributions. We do not pursue this direction here.
\paragraph*{\bf Step 1 - subsampling:}
We will generate a sub-matrix of $A$ by selecting a subset of the columns. This will ensure that the norm of the columns is bounded, as well as that the norm of their sum is small. Suppose that the columns of $A$ satisfy \eqref{eq:anticoncentration} with $\kappa,\sigma > 0$, for $i=1,\ldots, n$ we define random variables $Y_i\in \{0,1\}$:
\begin{align*}
	\Pr[Y_{k+1}=1|A_1,Y_1,\dots A_k,Y_k]=\begin{cases}
		1&\text{if }\langle\sum_{j=1}^kY_{j}\left(A_{j}-\mu\right), A_{k+1}-\mu\rangle < 0 \text{ and } \|A_{k+1} -\mu\| \leq 10 \sigma\sqrt{\frac{m}{\kappa}}\\
		\frac12&\text{if }\langle \sum_{j=1}^{k}Y_{j}\left(A_{j}-\mu\right), A_{k+1}-\mu\rangle = 0 \text{ and } \|A_{k+1} -\mu\| \leq 10 \sigma\sqrt{\frac{m}{\kappa}}\\
		0&\text{else} 
	\end{cases}
\end{align*}
We then select all columns of $A$ for which $Y_i = 1$. For now, let $A'_i$ have the law of column $A_i$, conditional on being selected, and denote the selected set $S_A =\{i \in [n]| Y_i = 1\}$.
\begin{lemma} \label{lem:subsample_props}
	Suppose that the columns of $A$ satisfy \eqref{eq:anticoncentration} with $\kappa,\sigma > 0$ and that it is an approximately symmetric distribution, in the sense of Definition \ref{def:addmis}. Then, if $n \gg \frac{m^2}{\kappa}$:
	\begin{enumerate}
		\item 
		\begin{align}
		\|A'_i - \mu\| \leq 10\sigma\sqrt{\frac{m}{\kappa}}.	\tag{norm concentration}\label{eq:norm bound}
		\end{align}
		\item With probability $1-e^{-\Omega(n)}$, $|S_A| \geq \frac{n}{200}$,
		\begin{align*} 
		\sum_{i \in S_A}(A_i-\mu)(A_i-\mu)^\T  &\preccurlyeq 2n  \sigma^2 \mathrm{I}_m\text{, and}\tag{matrix concentration} \label{eq:matrix-concentration}
		\\
		\left\|\sum_{i \in S_A}A_i-\mu\right\|^2  &\leq 2nm\sigma^2. \label{eq:concentration} \tag{concentration}
		\end{align*}
		\item $\Pr\left(d(\theta^\T A'_i, \mathbb{Z}) \geq \kappa\min\left(1,\|\theta\|_\infty\sigma\right)|A_1Y_1,...,A_{i-1}Y_{i-1}\right)\geq \frac{\kappa}{2}$, for every $\theta \in V$, and $i \in [n']$. 
	\end{enumerate}
\end{lemma}
\begin{proof}[\ifflag Proof \else Proof of Lemma \ref{lem:subsample_props}\fi]
	The first claim is immediate since we have conditioned the columns on the event $\{\|A_i - \mu\|\leq 10\sigma\sqrt{\frac{m}{\kappa}}\}$.
	For \eqref{eq:matrix-concentration},
	let $Z_i\stackrel{\text{law}}{=} A_i|\left(\|A_i - \mu\|\leq 10\sigma\sqrt{\frac{m}{\kappa}}\right)$ and note,
	$$\sum_{i \in S_A}(A_i-\mu)(A_i-\mu)^\T = \sum_{i=1}^{n}Y_i(Z_i-\mu)(Z_i-\mu)^\T \leq \sum_{i=1}^{n}(Z_i-\mu)(Z_i-\mu)^\T.$$
	As the random vectors $\{Z_i - \mu\}_{i=1}^n$ are mutually independent and  $(Z_i-\mu)(Z_i-\mu)^\T \preceq 10\sigma\sqrt{\frac{m}{\kappa}}\mathrm{I}_m$ almost surely, \eqref{eq:matrix-concentration} follows from the matrix Bernstein inequality \cite[Theorem 1.6.2]{T15}.

	For \eqref{eq:concentration}, since $\langle \sum_{j=1}^{i-1}Y_j(A_j -\mu), A_{i} -\mu \rangle \leq 0$,
	\begin{align*}
		\left\|\sum_{i\in S_A}A_i-\mu\right\|^2 \leq \sum_{1\in S_A}\|A_i-\mu\|^2 = \mathrm{Tr}\left(\sum\limits_{i\in S_A}^{n}(A_i -\mu)(A_i -\mu)^\T\right) \leq 2n \mathrm{Tr}\left(\sigma^2\mathrm{I}_m\right),
	\end{align*}
	where the last inequality is monotonicity of the trace.
	Also, recalling that $\mathrm{Cov}(A_i) \preceq \sigma^2\mathrm{I}_d,$ the fact that with high probability $|S_A| \geq \frac{n}{8}$ follows from Azuma's inequality. Indeed, for fixed $i \in [n]$, by Chebyshev's inequality,
	\[
	\Pr\left(\|A_i - \mu\|> 10\sigma\sqrt{\frac{m}{\kappa}}\right) \leq \frac{\kappa\EE\left[\|A_i-\mu\|^2\right]}{100\sigma^2m} = \frac{\kappa\Tr\left(\mathrm{Cov}(A_i)\right)}{100\sigma^2m} \leq \frac{\kappa}{100}.
	\]
	Since $A_i$ has an approximately symmetric law, then, since $A_i$ is independent from $\{Y_j, A_j\}_{j=1}^{i-1}$, by definition, 
	$$\Pr\left(\langle \sum_{j=1}^{i-1}Y_j(A_j -\mu), A_i -\mu \rangle \leq 0\right) \geq \frac{1}{4e^2}.$$
Taken together, the above displays imply
	$$\Pr\left(Y_i = 1| A_1,\dots,A_{i-1}\right) \geq \frac{1}{100}.$$
	Applying Azuma's inequality, as in \eqref{eq:azuma}, we get
	$$\Pr\left(|S_A| \geq \frac{n}{200}\right) = 1-e^{-\Omega(n)}.$$
	Finally, we address the \eqref{eq:anticoncentration} property.
	For fixed $i \in [n']$, let us define $\nu = \sum_{j=1}^{i-1}A_{j}Y_{j}$.
	So,
	\begin{align*}
		\Pr&\left(d(\theta^\T A'_i, \mathbb{Z}) \geq \kappa\min\left(1,\|\theta\|_\infty\sigma\right)|A_1Y_1,...,A_{i-1}Y_{i-1}\right) \\
		&=\Pr\left(d(\theta^\T A_i, \mathbb{Z}) \geq \kappa\min\left(1,\|\theta\|_\infty \sigma\right)|\langle \nu, A_i -\mu \rangle \leq 0 \text{ and } \|A_i - \mu\|\leq 10\sigma\sqrt{\frac{m}{\kappa}}\right)\\
		&\geq\Pr\left(d(\theta^\T A_i, \mathbb{Z}) \geq \kappa\min\left(1,\|\theta\|_\infty \sigma\right)|\langle \nu, A_i -\mu \rangle \leq 0 \right) - \frac{\kappa}{100} \geq \frac{\kappa}{2}.
	\end{align*}
	Here, the last inequality follows from Definition \ref{def:AC}, while the first inequality is a union bound on the anti-concentration event and $\{\|A_i - \mu\|\leq 10\sigma\sqrt{\frac{m}{\kappa}}\}$.
\end{proof} 
In light of the lemma, in the sequel, all computations will be made conditioned on the high-probability event defined by Lemma \ref{lem:subsample_props}, and we will only consider the selected columns. Thus, with a slight abuse of notation, from now on, the random variables $D$, $\hat{D}$, $A_i$, etc., will only be considered with respect to the selected columns. In particular, we will write $n$ for $|S_A|$.

\paragraph*{\bf Step II - bounding the argument:}
We will now show that for small $\theta$, the argument of $\Dh(\theta)$ is close to $2\pi np \langle \theta , \mu\rangle$. Before proceeding, we introduce an auxiliary parameter $\rad =\sqrt{\frac{10^3m^2}{\kappa^3}\ln\left(\frac{10^3m}{\kappa^3}\left(1 + \frac{\|\mu\|^2}{\sigma^2}\right)^{\frac{1}{m}}\right)}$, parametrizing the region in which we control $\Dh(\theta)$. We record here some facts which will be useful later on:
\begin{lemma} \label{lem:tproperties}
	Assume that Assumption \ref{ass:technical} is satisfied. If $\rad =\sqrt{\frac{10^3m^2}{\kappa^3}\ln\left(\frac{10^3m}{\kappa^3}\left(1 + \frac{\|\mu\|^2}{\sigma^2}\right)^{\frac{1}{m}}\right)}$, then for large enough $n$,
	\begin{align*} 
	\frac{1}{\sqrt{24\pi^2}} \leq \rad &\leq \frac{1}{80\sqrt{pm}}\nonumber\\
	\left(\frac{m}{\kappa}\right)^\frac{3}{2}\rad^3& \leq \frac{\sqrt{pn}}{50000}\\
	\left(1 + \frac{\|\mu\|}{\sigma}\right)^{3}\rad^3&  \leq \frac{\sqrt{pn}}{50000}
	\end{align*} 
\end{lemma}
\begin{proof}[\ifflag Proof \else  Proof of Lemma \ref{lem:tproperties} \fi]
	The lower bound in the first inequality is trivial. The upper bound follows from the upper bound in \eqref{eq:p-bound}, $p\leq \frac{\kappa^8}{10^{15}m^6}\frac{1}{\ln\left(10\left(1+\frac{\|\mu\|^2}{\sigma^2}\right)\right)}$ , because 
	\begin{align*}
		\beta &= \sqrt{\frac{10^3m^2}{\kappa^3}\left(\ln\left(\frac{100m}{\kappa^3}\right)+\ln\left(10\left(1 + \frac{\|\mu\|^2}{\sigma^2}\right)^{\frac{1}{m}}\right)\right)}\\
		 &\leq \sqrt{\frac{10^3m^2}{\kappa^3}\left(3\left(\frac{100m}{\kappa^3}\right)^{\frac{1}{3}}+\ln\left(10\left(1 + \frac{\|\mu\|^2}{\sigma^2}\right)^{\frac{1}{m}}\right)\right)}\\
		 &\leq \sqrt{\frac{10^5m^{2.5}}{\kappa^4}\ln\left(10\left(1 + \frac{\|\mu\|^2}{\sigma^2}\right)^{\frac{1}{m}}\right)} \leq \frac{1}{80\sqrt{pm}}.
	\end{align*}
The first inequality uses $\ln(x) = 3\ln(x^{\frac{1}{3}})\leq 3 x^{\frac{1}{3}}$ and the second $x+y \leq xy$, valid when $x,y\geq 1$.

For the first inequality concerning $\beta^3$, we use the above bound, which implies $\beta^3 \leq 10^9\frac{m^{\frac{15}{2}}}{\kappa^6}$. We thus compute,
\begin{align*}
	\left(\frac{m}{\kappa}\right)^\frac{3}{2}\rad^3 \leq 10^9\frac{m^9}{\kappa^8}\left(1 + \frac{\|\mu\|}{\sigma}\right)^2 \leq \frac{\sqrt{pn}}{50000}
\end{align*}
where the second inequality follows, again, from \eqref{eq:p-bound}.
Using a similar argument we also get, 
	\begin{align*}
		\left(1 + \frac{\|\mu\|}{\sigma}\right)^{5}\rad^3&\leq 10^9\frac{m^8}{\kappa^6}\left(1 + \frac{\|\mu\|}{\sigma}\right)^2 \leq \frac{\sqrt{pn}}{50000}
	\end{align*}
	which completes the proof.
\end{proof} The main observation in this step, is that, when we fix the columns $\{A_j\}_{j=1}^n$, a Taylor approximation implies the bound,
$$\left|\sum_{j=1}^n\arg\left(\mathbb{E}_S[\exp(\mathbf{1}_{i\in S}2\pi iA_j)]\right)-2\pi np \langle \theta , \mu \rangle \right|	\leq 2\pi p \left|\sum_{j=1}^n\langle \theta, A_j\rangle-n\langle \theta, \mu\rangle \right|+50 p \sum_{j=1}^n\left|\langle\theta, A_j\rangle \right|^3.$$
The term on the LHS controls $\arg(\Dh(\theta))$ and the two terms on the RHS can be bounded with \eqref{eq:concentration} and \eqref{eq:norm bound} respectively. We then prove:

\begin{lemma}
	\label{lem:argbound}
	With probability $1- e^{-\Omega(n)}$ over $A$, for $\|\theta\|\leq \frac{\rad}{\sigma\sqrt{p n}}$,
	$$|\arg(\Dh(\theta))-2\pi np \langle \theta , \mu\rangle|\leq 4\pi \rad\sqrt{pm} +\frac{5000({(m/\kappa)}^{\frac{3}{2}}+\|\mu\|^3/\sigma^3)\rad^3}{\sqrt{ pn}}.$$
\end{lemma}
\begin{proof}[\ifflag Proof \else Proof of Lemma \ref{lem:argbound}\fi]
	Let $f(x)=\arg(p\cdot \exp(2\pi i\cdot x) + (1-p))$ and  observe that $f(x)=\arctan\left(\frac{p\sin(2\pi x)}{p\cos(2\pi x) +(1-p)} \right)$. A calculation shows $f(0) = 0$, $f'(0) = 2\pi p$ and $f''(0) = 0$. Hence, we set $g(x)=f(x)-2\pi p x$ and, with a second-order Taylor approximation of $f(x)$ around $x=0$, we see that for every $x\geq 0$ there is some $x'\in [0,x]$ such that
	\begin{align*}
		g(x)&=\frac{d^3f}{dx^3}(x')x'^3.
	\end{align*}
	Another calculation shows that, as long as $p\leq 0.1$, we have $\frac{d^3f}{dx^3}(x')\leq 50p$ and hence $|g(x)|
	\leq 50|x|^3p$.	Now, note,
	\begin{align*}
		\arg(\Exp[\exp(\mathbf{1}_{j\in S}2\pi ix)])&=\arg(p\cdot \exp(2\pi i\cdot x) + (1-p))=f(x).
	\end{align*}
	Now,
	\begin{align*}
		\left|\arg(\Dh(\theta))-2\pi np \inner{\theta, \mu }\right|&=\left|\arg(\prod_{j=1}^n\Exp[\exp(\mathbf{1}_{j\in S}2\pi i \inner{\theta, A_{j}})])-2\pi np \inner{\theta, \mu}\right|\\&= \left|\sum_{j=1}^n\arg(\exp(\mathbf{1}_{j\in S}2\pi i \inner{\theta, A_{j}}))-2\pi np \inner{\theta, \mu }\right|
		\\&= \left|\sum_{j=1}^nf(\inner{\theta, A_j})-2\pi np \inner{\theta, \mu }\right|,
	\end{align*}
	where we understand $|\cdot|$ as referring to distance on the circle. The Taylor approximation given above shows that, when $\|\theta\| \leq\frac{\rad}{\sigma\sqrt{p n}}$, we can bound this distance (where now the bound will be expressed as a distance between real numbers),
	\begin{align*}
		\left|\sum_{j=1}^nf(\inner{\theta, A_j})-2\pi np \inner{\theta, \mu }\right|	&\leq 2\pi p \left|\sum_{j=1}^n\inner{\theta, A_j}-n\inner{\theta, \mu}\right|+50 p \sum_{j=1}^n\left|\inner{\theta, A_j}\right|^3\\
		&\leq 2\pi p\left|\left\langle\theta, \sum\limits_{j=1}^n(A_j - \mu) \right\rangle \right| +50p\|\theta\|^3\sum\limits_{j=1}^n\|A_j\|^3_2\\
		&\leq 2\pi p \|\theta\|\left\|\sum\limits_{j=1}^n(A_j - \mu)\right\| + 50p\|\theta\|^3n\left(10\sigma\sqrt{m/\kappa} + \|\mu\|\right)^3\\
		&\leq 4\pi p\rad\sqrt{\frac{m\sigma^2}{\sigma^2p}} +\frac{50(10\sigma\sqrt{m/\kappa}+\|\mu\|)^{3}\rad^3}{\sigma^3\sqrt{np}}\\
		&= 4\pi \rad\sqrt{pm} +\frac{50\left(10\sqrt{m/\kappa}+\frac{\|\mu\|}{\sigma}\right)^{3}\rad^3}{\sqrt{pn}},
	\end{align*}
	where the third inequality follows from \eqref{eq:norm bound} and the penultimate inequality being a consequence of \eqref{eq:concentration}.
\end{proof} 
The following corollary is now immediate.
\begin{corollary} \label{cor:paramters}
	With probability $1- e^{-\Omega(n)}$,
	for all $\|\theta\|\leq \frac{\rad}{\sigma\sqrt{p n}}$,
	$$|\arg(\Dh(\theta))-2\pi np \langle\theta, \mu\rangle| \leq \frac{1}{8}\pi.$$
\end{corollary}
\begin{proof}
		Lemma \ref{lem:tproperties} implies $\rad\sqrt{pm} \leq \frac{1}{80}$, $\rad^3\left(\frac{\|\mu\|}{\sigma}\right)^{3}  \leq \frac{\sqrt{pn}}{50000}$ and, $\rad^3(m/\kappa)^\frac{3}{2} \leq \frac{\sqrt{pn}}{50000}.$ Thus,
		$$4\pi \rad\sqrt{pm} +\frac{50\left(10\sqrt{m/\kappa}+\frac{\|\mu\|}{\sigma}\right)^{3}\rad^3}{\sqrt{pn}} \leq \frac{1}{8}\pi,$$
		and the Corollary follows from Lemma \ref{lem:argbound}.
\end{proof}
\paragraph*{\bf Step III - bounding the integral from below, near the origin:}
Next, we prove that the modulus of $\Dh$ is bounded, near the origin.
 \iffalse This is facilitated by an elementary calculation which shows for small $p$,
$|\Dh(\theta)| \geq \exp(-6\pi^2 p\sum_{i=1}^n\langle \theta, A_i\rangle^2).$
Using \eqref{eq:matrix-concentration} and comparing to a Gaussian integral produces the following result.
\fi
\begin{lemma} \label{lem:firstintegrallowerbound}
	The following holds,
	$$\int_{\|\theta\|\leq \frac{\rad}{\sigma\sqrt{p n}}} |\Dh(\theta)|d\theta \geq \left(\frac{1}{200\pi^3nmp}\right)^{\frac{m}{2}}\frac{1}{\sigma^{m-1}\left(\sigma^2 + \|\mu\|^2\right)^{\frac{1}{2}}}.$$
\end{lemma}
\begin{proof}[\ifflag Proof \else Proof of Lemma \ref{lem:firstintegrallowerbound}\fi]
	We have:
	\begin{align*}
		|\Dh(\theta)|&=\prod_{j=1}^n|\Exp[\exp(2\pi i \langle \theta, A_{j}\rangle)]|=\prod_{j=1}^n|(1-p)\cdot 1+p\exp(2\pi i \langle \theta, A_{j}\rangle)|\\&= \prod_{j=1}^n\sqrt{(1-p+p\cos(2\pi \langle \theta, A_{j}\rangle))^2+p^2\sin(2\pi \langle \theta, A_{j}\rangle)^2 }\\&\geq \prod_{j=1}^n(1-p+p\cos(2\pi \langle \theta, A_{j}\rangle))\geq \exp(-6\pi^2 p\sum\limits_{i=1}^n\langle \theta, A_i\rangle^2).
	\end{align*}

	Here the last inequality follows, as long as $p \leq 0.01$, from the elementary inequalities,
	\begin{align*}
		\cos(x)&\geq 1-x^2\\
		\ln(1-x)&\geq -\frac{3}{2}x\ \ \text{ when } |x| \leq \frac{1}{2}.
	\end{align*}
	Indeed, it's enough to consider $\inner{\theta, A_j} \in [-1,1]$, for which,
	$$\ln(1-p +p\cos\left(2\pi\inner{\theta, A_j}\right)) \geq \ln\left(1- p4\pi^2\inner{\theta, A_j}^2\right) \geq -6\pi^2p\inner{\theta, A_j}^2.$$
	Hence, by applying the \eqref{eq:matrix-concentration} property to the obtained bound, we get,
	\begin{align*}
		|\Dh(\theta)| &\geq \exp\left(-12\pi^2p\left(n\langle\theta, \mu\rangle^2 +\sum\limits_{i=1}^n \langle \theta, A_i - \mu\rangle^2   \right)\right) \geq \exp\left(-24\pi^2np\left(\langle\theta, \mu\rangle^2 +\sigma^2\theta\theta^\T  \right)\right)\\
		&\geq\exp\left(-24\pi^2\theta\left(\mu^\T\mu + \sigma^2\mathrm{I}_m\right)\theta^\T np\right).
	\end{align*}
	Let $Y \sim \mathcal{N}\left(0,\frac{1}{48\pi^2np}(\mu^\T\mu + \sigma^2\mathrm{I}_m)^{-1}\right)$, then,
	\begin{align} \label{eq:probLowerBound}
		\int_{\|\theta\|\leq \frac{\rad}{\sigma\sqrt{pn}}} |\Dh(\theta)|d\theta &\geq \int_{\|\theta\|\leq \frac{\rad}{\sigma\sqrt{pn}}} \exp\left(-24\pi^2\theta\left(\mu^\T\mu + \sigma^2\mathrm{I}_m\right)\theta^\T np\right)d\theta\nonumber\\
		&=\frac{1}{\sqrt{\det\left(96\pi^3np\left(\mu^\T\mu + \sigma^2\mathrm{I}_m\right)\right)}}\Pr\left(\|Y\|\leq \frac{\rad}{\sigma \sqrt{p n}}\right)\nonumber\\
		&= \frac{1}{(96\pi^3np)^{\frac{m}{2}}(\|\mu\|^2 + \sigma^2)^{\frac{1}{2}}\sigma^{m-1}}\Pr\left(\|Y\|\leq \frac{\rad}{\sigma \sqrt{pn}}\right).
	\end{align}
	By Chebyshev's inequality,
	\begin{align*}
		\Pr\left(\|Y\|\geq \frac{\rad}{\sigma\sqrt{pn}}\right) \leq \frac{\sigma^2pn}{\rad^2}\Tr\left(\frac{1}{48\pi^2np}(\mu^\T\mu + \sigma^2\mathrm{I}_m)^{-1}\right) \leq \frac{m}{48\pi^2\rad^2} \leq \frac{1}{2},
	\end{align*}
	where in the last inequality, we have used $\rad\geq\sqrt{\frac{10^3m^2}{\kappa^3}} \geq \sqrt{2\frac{m}{48\pi^2}}$. The proof concludes by plugging this estimate into \eqref{eq:probLowerBound}.
\end{proof} To handle the integral in Theorem \ref{thm:fourier}, we note that when $\|t-pn\mu\|$ is small enough, as dictated by Theorem \ref{thm:main}, Corollary \ref{cor:paramters} implies,
$$|\arg(\Dh(\theta)\exp(2\pi i \langle \theta, t\rangle))|\leq|\arg(\Dh(\theta)) - 2\pi n p\langle\theta, \mu\rangle|+|\arg(\exp(2\pi i \langle \theta, t -pn\mu\rangle))| \leq \frac{1}{4} \pi,$$
whenever $\|\theta\| \leq \frac{\rad}{\sigma\sqrt{p n}}$ (recall  $\rad= \sqrt{\frac{10^3m^2}{\kappa^3}\ln\left(\frac{10^3m}{\kappa^3}\left(1 + \frac{\|\mu\|^2}{\sigma^2}\right)^{\frac{1}{m}}\right)}$). For a complex number $z$, we have $\Re(z) = \cos(\arg(z))|z|$. Thus, from Lemma \ref{lem:firstintegrallowerbound}, we deduce the following result.
\begin{lemma}
	Let $t \in \mathbb{R}^m$ such that
	$$\|t-pn\mu \|\leq \frac{\sigma \kappa^{\frac{3}{2}}}{10^4\ln\left(\frac{10^3m}{\kappa^3}\left(1 + \frac{\|\mu\|^2}{\sigma^2}\right)^{\frac{1}{m}}\right)}\frac{\sqrt{pn}}{m}.$$ 
	Then, for $\|\theta\| \leq \frac{\rad}{\sigma\sqrt{p n}}$,
	$$\Re\left[\Dh(\theta)\exp(-2\pi i\langle \theta, t \rangle )\right]\geq 0,$$
	and 
	$$\Re\left[\int\limits_{\|\theta\|\leq \frac{\rad}{\sigma\sqrt{ pn}}}\Dh(\theta)\exp(-2\pi i\langle \theta, t \rangle )d\theta\right] \geq \cos\left(\frac{\pi}{4}\right)\left(\frac{1}{200\pi^3nmp}\right)^{\frac{m}{2}}\frac{1}{\sigma^{m-1}\left(\sigma^2 + \|\mu\|^2\right)^{\frac{1}{2}}}.$$\label{lem:integral_lower_bound}
\end{lemma}
\paragraph*{\bf Step IV - exponential decay of the Fourier spectrum:}
To show that the Fourier spectrum decays rapidly, we employ an $\epsilon$-net argument over a very large box. The main difficulty is that $\langle \theta, D \rangle$ can be close to an integer, irregardless of the value $\|\theta\|$. Our anti-concentration assumption allows us to avoid this. Specifically, Item 3 in Lemma \ref{lem:subsample_props} implies that for any given $\theta$ in a dense enough net, we can expect many columns to satisfy that $\langle \theta, A_i\rangle$ is far from any integer point. Formally, we prove:
\begin{lemma} \label{lem:spectrumdecay}
	Assume that Assumption \ref{ass:technical} is satisfied. Then, with probability $1-e^{-\Omega(\kappa n)}$, we have $$|\Dh(\theta)|\leq \exp\left(-\frac{1}{80}\kappa^3 n(1-p)p\pi^2\min\left(1, \|\theta\|^2_\infty\sigma^2\right)\right),$$
	for 
	$\theta\in \tilde{V},$
	where $V$ is the fundamental domain, as in \eqref{eq:funddomain}, and
	$$ \tilde{V} := \left[-e^{\frac{\kappa^3pn}{80m}}\frac{1}{4\sigma n^2},e^{\frac{\kappa^3pn}{80m}}\frac{1}{4\sigma n^2}\right]^m \cap V.$$  \label{lem:normbound}
\end{lemma}
\begin{proof}[\ifflag Proof \else Proof of Lemma \ref{lem:normbound}\fi]
	We begin with a technical calculation which will help control the size of the net. Observe that \eqref{eq:p-bound} implies
	$$160m \leq \sqrt{\kappa n} \implies \sqrt{\frac{m}{\kappa}} \leq \frac{1}{2}\sqrt{\frac{\kappa n}{80m}} \implies \leq \sqrt{\frac{m}{\kappa}} \leq \frac{1}{2}e^{\frac{\kappa n}{80m}},$$
	and similarly,
	$$ \frac{\|\mu\|}{\sigma} \leq \frac{1}{2}e^{\frac{\kappa n}{80m}}.$$
	
	Now, let $N$ be an $\epsilon$-net of $\tilde{V}$ for $\epsilon=\frac{1}{2n^2\left(\sigma\sqrt{m/\kappa} + \|\mu\|\right)}$.
	Standard arguments show that one can take $|N|\leq \left(\sqrt{\frac{m}{\kappa}} + \frac{\|\mu\|}{\sigma}\right)^me^{\frac{\kappa}{80} n} \leq e^{\frac{\kappa}{40} n}$, when $n$ is large enough.
	Here we have used the bound, $\sqrt{\frac{m}{\kappa}} + \frac{\|\mu\|}{\sigma} \leq e^{\frac{\kappa n}{80m}}$, established above.
	
	For $\theta\in N$, define $$\varphi(\theta)=\frac{\kappa}{4}\min\left(1,\|\theta\|_\infty\sigma\right) \text{ and } E(\theta)=\{j:d(\inner{\theta, A_j},\mathbb{Z}) \geq \varphi(\theta) \}.$$
	If we fix $\theta$ and set $X_i:={\bf{1}}_{i \in E(\theta)}$ then $|E(\theta)| = \sum\limits_{i=1}^nX_i$. The statement of Item 3 in Lemma \ref{lem:subsample_props}  is  $\mathbb{E}\left[X_i|X_1,\dots, X_{i-1}\right] \geq \frac{\kappa}{2}$. Applying Azuma's inequality \eqref{eq:azuma}, we get:
	\begin{align*}
		\Pr\left[|E(\theta)| \leq \frac{\kappa}{4}n\right]\leq \exp\left(-\frac{\kappa}{8}n\right).
	\end{align*}
	In this case, by the union bound,
	$$\mathbb{P}\left(\exists \theta \in N\ :\ |E(\theta)| \leq \frac{\kappa}{4}n\right) \leq e^{-\frac{\kappa}{16}n
	}.$$
	If $j\in E(\theta)$, since $\varphi(\theta) \leq \frac{1}{4}$, we have:
	\begin{align*}
		|\Exp_{S}[\exp(\mathbf{1}_{j\in S}\cdot 2\pi i \langle \theta, A_j\rangle )]|&=\sqrt{(1-p+p\cos(2\pi \inner{\theta, A_j}))^2+p^2\sin(2\pi \inner{\theta, A_j})^2 }\\
		&= \sqrt{1+2p^2-2p+2(1-p)p\cos(2\pi\inner{\theta, A_j})}\\
		&\leq \sqrt{1+2p^2-2p+2(1-p)p\cos(2\pi\varphi(\theta))}\\
		&\leq \sqrt{1-2(1-p)p(2\pi\varphi(\theta))^2/5}\\
		&\leq 1-(1-p)p(2\pi\varphi(\theta))^2/5\\
		&\leq 1-\frac45(1-p)p\pi^2\frac{\kappa^2}{16}\min\left(1,\|\theta\|_\infty^2\sigma^2\right).
	\end{align*}
	Observe that $\left|\Exp_S[\exp(\mathbf{1}_{j\in S}2\pi i x)]\right|=\sqrt{(1-p+p\cos(2\pi x))^2+p^2\sin(2\pi x)^2 }$ is $4\pi p$-Lipschitz in $x$, as long as $p\leq\frac14$.
	Take an arbitrary $\theta\in \tilde{V}$ and let $\theta'$ be the closest point in $N$.
	Recall that $\max\limits_i \|A_i\| \leq10 \sigma\sqrt{\frac{m}{\kappa}} +\|\mu\|$, which follows from \cref{eq:norm bound}. So, by our choice of $\epsilon$ and with the Cauchy-Schwartz inequality,
	$$|\inner{\theta-\theta', A_j}| \leq \frac{1}{2n^2}.$$
	Thus:
	\begin{align*}
		|\Dh(\theta)|&=\prod_{j=1}^n\left|\Exp_{S}[\exp(\mathbf{1}_{j\in S}2\pi i \inner{\theta, A_j} ) ]\right|\\
		&\leq \prod_{j\in E(\theta')}\left(\left|\Exp_{S}[\exp(\mathbf{1}_{j\in S}2\pi i \inner{\theta', A_j} )]\right| +4\pi p|\inner{\theta-\theta', A_j}|\right)\\
		&\leq \prod_{j\in E(\theta')}\left(1-\frac45(1-p)p\pi^2\frac{\kappa^2}{16}\min\left(1,\|\theta'\|_\infty^2\sigma^2
		\right)+4\pi p|\inner{\theta-\theta', A_j}|\right)\\
		&\leq \exp\left(-\frac{4}{5}(1-p)p\pi^2\frac{\kappa^2}{16}|E(\theta' )|\min\left(1,\|\theta'\|_\infty^2\sigma^2\right)+4\pi p\sum_{j=1}^n|\inner{\theta-\theta', A_j}|\right) \\
		&\leq e^{\frac{2\pi }{n}}\exp\left(-\frac{1}{5}\frac{\kappa^3}{16} n(1-p)p\pi^2\min\left(1,\|\theta'\|_\infty^2\sigma^2\right)\right)\\
		&\leq e^{\frac{2\pi}{n}} \exp\left(-\frac{1}{5}\frac{\kappa^3}{16} n(1-p)p\pi^2\min\left(1, (1-\epsilon)^2\|\theta\|^2_\infty\sigma^2\right)\right)\\
		&\leq \exp\left(-\frac{1}{80}\kappa^3 n(1-p)p\pi^2\min\left(1, \|\theta\|^2_\infty\sigma^2\right)\right),
	\end{align*}
	where the last inequality holds, since by $\eqref{eq:p-bound}$, $\kappa^3pn\geq 10^{20}$.
\end{proof} By properly integrating the inequality, we have thus obtained:
\begin{lemma} \label{lem:DExpDecay}
With probability $1-e^{-\Omega(\kappa n)}$, the following inequality holds:
	\begin{align*}
		\int\limits_{B}|\Dh(\theta)|d\theta\leq 2\left(\frac{1}{10^5pmn \sigma^2 \left(1 +\frac{\|\mu\|^2}{\sigma^2}\right)^{\frac{1}{m}}}\right)^{\frac{m}{2}},
	\end{align*}
	where
	$$B = V \cap \left\{\|\theta\|_\infty < e^{\frac{\kappa^3pn}{80m}}\frac{1}{4\sigma n^2}\right\} \cap \left\{\|\theta\|\geq \frac{\rad}{\sigma\sqrt{ pn}}\right\}.$$
\end{lemma}
\begin{proof}[\ifflag Proof \else Proof of Lemma \ref{lem:DExpDecay} \fi]
	From \cref{lem:normbound}, when $p < 0.1$, and $\sigma\|\theta\|_\infty \leq \sigma\|\theta\| \leq \sqrt{m}$, we have
	$$|D(\theta)|\leq \exp\left(-\frac{\kappa^3 pn\sigma^2\|\theta\|^2}{10m}\right).$$ 
Now, if $Y \sim \mathcal{N}(0,\mathrm{I}_m)$ and $n \geq 100m^6$:
	\begin{align*}
		\int\limits_{\sqrt{m}\geq \|\theta\|\geq \frac{\rad}{\sigma\sqrt{p n}}}|D(\theta)|d\theta&\leq  \int\limits_{\|\theta\|\geq \frac{\rad}{\sigma\sqrt{p n}}} \exp\left(-\frac{\kappa^3 pn\sigma^2\|\theta\|^2}{10m}\right)d\theta\\
		&\leq\left(\frac{5m}{\kappa^3 p n\sigma^2}\right)^{\frac{m}{2}}\mathbb{P}\left(\|Y\| \geq \frac{\rad}{\sigma\sqrt{p n}}\sqrt{\frac{\kappa^3 p n\sigma^2}{5 m}}\right)
		\\
		&= \left(\frac{5m}{\kappa^3 p n\sigma^2}\right)^{\frac{m}{2}}\mathbb{P}\left(\|Y\|^2 \geq \rad^2\frac{\kappa^3  }{5 m}\right)\\
		&= \left(\frac{5m}{\kappa^3 p n \sigma^2}\right)^{\frac{m}{2}}\mathbb{P}\left(\|Y\|^2 \geq 200m\ln\left(\frac{10^3m}{\kappa^3}\left(1 + \frac{\|\mu\|^2}{\sigma^2}\right)^{\frac{1}{m}}\right)\right).\\
	\end{align*}
	The last equality holds since $\rad = \sqrt{\frac{10^3m^2}{\kappa^3}\ln\left(\frac{10^3m}{\kappa^3}\left(1 + \frac{\|\mu\|^2}{\sigma^2}\right)^{\frac{1}{m}}\right)}$.
	By Lemma \ref{lem:chisquareconc},
	\begin{align*} 
	\mathbb{P}\left(\|Y\|^2 \geq 200m\ln\left(\frac{10^3m}{\kappa^3}\left(1 + \frac{\|\mu\|^2}{\sigma^2}\right)^{\frac{1}{m}}\right)\right) &\leq \exp\left(-50m\ln\left(\frac{10^3m}{\kappa^3}\left(1 + \frac{\|\mu\|^2}{\sigma^2}\right)^{\frac{1}{m}}\right)\right)\\
	 &\leq \left(\frac{\kappa^3}{10^6m}\right)^{\frac{m}{2}}\frac{1}{\left(1 + \frac{\|\mu\|^2}{\sigma^2}\right)^{\frac{1}{2}}}.
	\end{align*}
	Hence, 
	$$\int\limits_{\sqrt{m}\geq \|\theta\|\geq \frac{\rad}{\sigma\sqrt{ n}}}|D(\theta)|d\theta\leq \left(\frac{1}{10^5pmn \sigma^2 }\right)^{\frac{m}{2}}\frac{1}{\left(1 + \frac{\|\mu\|^2}{\sigma^2}\right)^{\frac{1}{2}}}.$$
	Furthermore, for all 
	$\theta\in \left[-\frac{e^{\frac{\kappa^3pn}{80m}}}{4\sigma n^2},\frac{e^{\frac{\kappa^3pn}{80m}}}{4\sigma n^2}\right]^m$ with $\|\theta\|_\infty\sigma\geq 1,$ Lemma \ref{lem:normbound} also implies, $$|D(\theta)|\leq \exp\left(-\frac{\kappa^3 p n}{40m}\right).$$
	So:
	\begin{align*}
		\int\limits_{ \frac{e^{\frac{\kappa^3pn}{80m}}}{4\sigma n^2} \geq \|\theta\|_\infty \geq \frac{1}{4}} |D(\theta)|d\theta\leq \frac{1}{\sigma^mn^{2m}}\exp\left(\frac{\kappa^3pn}{80m}\right)^m \cdot \exp\left(-\frac{\kappa^3 p n}{40}\right) =\frac{1}{\sigma^mn^{2m}}\exp\left(-\frac{\kappa^3 p n}{80}\right) .
	\end{align*}
	Summing the previous two inequalities yields
	\begin{align*}
		\int\limits_{B}|\Dh(\theta)|d\theta\leq \left(\frac{1}{10^5pmn \sigma^2 \left(1 +\frac{\|\mu\|^2}{\sigma^2}\right)^{\frac{1}{m}}}\right)^{\frac{m}{2}} +\left(\frac{1}{\sigma^2n^{4}}\right)^{\frac{m}{2}}\exp\left(-\frac{\kappa^3 p n}{80}\right).
	\end{align*}
	 The proof concludes by recalling the assumption from \eqref{eq:p-bound}, which implies a lower bound on $\kappa^3pn$, and hence affords the bound,
	 \begin{align*}
	 &\left(\frac{1}{10^5pmn \sigma^2 \left(1 +\frac{\|\mu\|^2}{\sigma^2}\right)^{\frac{1}{m}}}\right)^{\frac{m}{2}} +\left(\frac{1}{\sigma^2n^{4}}\right)^{\frac{m}{2}}\exp\left(-\frac{\kappa^3 p n}{80}\right)\\
	 \ \ \ \ \ &\leq 2\left(\frac{1}{10^5pmn \sigma^2 \left(1 +\frac{\|\mu\|^2}{\sigma^2}\right)^{\frac{1}{m}}}\right)^{\frac{m}{2}}.
	 \end{align*}
\end{proof} \paragraph*{Proving Theorem \ref{thm:main} and \ref{cor:maindisc}:}
If the support of $\mathcal{D}$ is contained in $\Z^m$, then Lemma \ref{lem:integral_lower_bound} and Lemma \ref{lem:DExpDecay} are enough to prove \Cref{cor:maindisc}.
\begin{proof}[Proof of \Cref{cor:maindisc}]
	Let $B$ be defined as in Lemma \ref{lem:DExpDecay} and note that since $2\sigma n^2\leq  e^{\frac{\kappa^3pn}{80m}}$,
	$$\left[-\frac{1}{2}, \frac{1}{2}\right]^m = V \subset \left\{\|\theta\|_\infty < e^{\frac{\kappa^3pn}{80m}}\frac{1}{4\sigma n^2}\right\}.$$
	Thus,
	$$	B = \left[-\frac{1}{2}, \frac{1}{2}\right]^m \cap \left\{\|\theta\|\geq \frac{\rad}{\sigma\sqrt{ pn}}\right\}.$$
	 By Lemma \ref{lem:DExpDecay}, 
	\begin{align} \label{eq:mainintegralupperbound}
		\int\limits_{B}|\Dh(\theta)|d\theta\leq 0.2\left(\frac{1}{200\pi^3nmp\sigma^2 \left(1 +\frac{\|\mu\|^2}{\sigma^2}\right)^{\frac{1}{m}}}\right)^{\frac{m}{2}}.
	\end{align}
	Moreover, by Lemma \ref{lem:integral_lower_bound},
	\begin{align} \label{eq:mainintegrallowerbound}
		\Re\left[\int_{\|\theta\| \leq \frac{\rad}{\sigma\sqrt{p n}}}\Dh(\theta)\exp(-2\pi i\langle \theta, t \rangle )d\theta\right] &\geq \cos\left(\frac{\pi}{4}\right) \left(\frac{1}{200\pi^3nmp}\right)^{\frac{m}{2}}\frac{1}{\sigma^{m-1}\left(\sigma^2 + \|\mu\|^2\right)^{\frac{1}{2}}}\nonumber \\
		&=\cos\left(\frac{\pi}{4}\right) \left(\frac{1}{200\pi^3nmp\sigma^2\left(1 +\frac{\|\mu\|^2}{\sigma^2}\right)^{\frac{1}{m}}}\right)^{\frac{m}{2}}.
	\end{align}
	So,
	\begin{align*}
		\Pr[D=t] &= \Re\left[\int_{\theta\in [-\frac{1}{2},\frac{1}{2}]^m}\Dh(\theta)\exp(-2\pi i\langle \theta, t \rangle )d\theta\right]\nonumber\\
		&\geq \Re\left[\int_{\|\theta\| \leq \frac{\rad}{\sigma\sqrt{p n}}}\Dh(\theta)\exp(-2\pi i\langle \theta, t \rangle )d\theta\right] - \int_{ B}|\Dh(\theta)|d\theta\nonumber\\
		&\geq \frac{1}{2}\left(\frac{1}{200\pi^3nmp\sigma^2\left(1 +\frac{\|\mu\|^2}{\sigma^2}\right)^{\frac{1}{m}}}\right)^{\frac{m}{2}}.
	\end{align*}
	Finally, using the multiplicative Chernoff bound we can see that:
	\begin{align*}
		\Pr[|S|\notin [0.5pn,1.5pn]]&=\Pr\left[\left||S|-\Exp[|S|]\right|\geq \frac12\Exp[|S|]]\right]\leq 2\exp(-pn/12).
	\end{align*}
	By \eqref{eq:p-bound}, $pn \geq \left(10m\left(1 +\frac{\|\mu\|}{\sigma}\right)\right)^{20}.$ So, we have $\Pr[D=t]>2\exp(-pn/12)$, which implies the existence of a suitable $S$ with $|S|\in [0.5pn,1.5pn]$ and $A\mathbf{1}_S=t$.
\end{proof}
The proof of the continuous case requires an extra step to control the Fourier transform outside the domain of Lemma \ref{lem:DExpDecay}. To deal with continuous distributions, we define $R\sim \mathcal{N}(0,\gamma \mathrm{I}_m)$ for $\gamma$ to be determined later and $H=D+R$. If $f_H$ is the density of $H$, we will show that, for appropriate $t$, $f_H(t)$ is positive, from which it will follow that with high probability there exists a  suitable set $S$ such that $\|A\mathbf{1}_S-t\|$ is small.

\begin{proof}[Proof of \cref{thm:main}]
	As in the proof of \cref{cor:maindisc}, let $B$ be defined as in Lemma \ref{lem:DExpDecay}. Since the columns of $A$ are now absolutely continuous we have $V = \R^m$, and so,
	$$	B = \left\{\|\theta\|_\infty < e^{\frac{\kappa^3pn}{80m}}\frac{1}{4\sigma n^2}\right\} \cap \left\{\|\theta\|\geq \frac{\rad}{\sigma\sqrt{ pn}}\right\}.$$ 
	Another difference from the discrete case is that now, by the multiplication-convolution theorem (Theorem \ref{thm:multconv}), we have $\hat{H}(\theta) = \hat{D}(\theta)\cdot \hat{R}(\theta) = e^{-\frac{\|\theta\|^2\gamma}{2}}\hat{D}(\theta).$
	So, if $\theta_0 := \exp\left(\frac{\kappa^3pn}{80m}\right)\frac{1}{\sigma n^2}$, the estimates in \eqref{eq:mainintegralupperbound} and \eqref{eq:mainintegrallowerbound} imply,
	\begin{align*}
		\Re&\left[ \int\limits_{[-\theta_0,\theta_0]^m} \hat{H}(\theta)\exp(-2\pi i\langle \theta, t \rangle )d\theta\right]\\
		&\geq e^{-\frac{\rad^2\gamma}{\sigma^2 p n}}\left(\Re\left[\int_{\|\theta\| \leq \frac{\rad}{\sigma\sqrt{ pn}}}\Dh(\theta)\exp(-2\pi i\langle \theta, t \rangle )d\theta\right] - \int_{B}|\Dh(\theta)|d\theta\right)\nonumber\\
		&\geq \frac{1}{4}\left(\frac{1}{200\pi^3nmp\sigma^2\left(1 +\frac{\|\mu\|^2}{\sigma^2}\right)^{\frac{1}{m}}}\right)^{\frac{m}{2}},
	\end{align*}
	where the last inequality holds as long as $\gamma < \frac{\sigma^2pn}{\beta^2}$.
	Since we wish to invoke Theorem \ref{thm:fourier}, we need to bound $\int\limits_{\R^m} \hat{H}(\theta)\exp(-2\pi i\langle \theta, t \rangle )d\theta$ from below. In light of the above computation, it will be enough to choose $\gamma$ such that,
	\begin{equation} \label{eq:gammadef}
		\int\limits_{\|\theta\|_\infty \geq \theta_0} |\hat{H}(\theta)|d\theta \leq \frac{1}{8}\left(\frac{1}{200\pi^3nmp\sigma^2\left(1 +\frac{\|\mu\|^2}{\sigma^2}\right)^{\frac{1}{m}}}\right)^{\frac{m}{2}}.
	\end{equation}
	Towards finding an appropriate $\gamma$, let $Y$ stand for the standard Gaussian in $\mathbb{R}^m$, and compute,
	\begin{align*}
		\int\limits_{\|\theta\|_\infty \geq \theta_0} |\hat{H}(\theta)|d\theta &\leq \int\limits_{\|\theta\|\geq \theta_0} e^{\frac{-\gamma\|\theta\|^2}{2}}d\theta\\
		&= \left(\frac{2\pi}{\gamma}\right)^\frac{m}{2}\Pr\left(\|Y\|^2\geq \gamma\theta_0^2\right)\\
		&\leq \left(\frac{2\pi}{\gamma}\right)^\frac{m}{2}e^{-\frac{\gamma\theta_0^2}{3}},
	\end{align*}
	where the last inequality is Lemma \ref{lem:chisquareconc}.
	Let us choose now
	$$\gamma = \exp\left(-\kappa^3\frac{pn}{40m}\right)\sigma^2 n^5 = \frac{ n}{\theta_0^2},$$ for which \eqref{eq:gammadef} holds. Also, $\gamma < \frac{ n\sigma^2}{\beta^2}$ as required earlier and $\gamma\theta_0^2 \geq 7m$, as required by Lemma \ref{lem:chisquareconc}. If $f_H(t)$ is the density of $H$ at $t$, Theorem \ref{thm:fourier} along with \eqref{eq:gammadef} give,
	\begin{align*}
		f_H(t) &= \Re\left[\int\limits_{\mathbb{R}^m}\hat{H}(\theta)\exp(-2\pi i \langle \theta, t\rangle)d\theta\right]\\
		&\geq \Re\left[\int\limits_{\|\theta\|_\infty\leq \theta_0 }\hat{H}(\theta)\exp(-2\pi i \langle\theta, t\rangle)d\theta\right] -  \int\limits_{\|\theta\|_\infty> \theta_0}|\hat{H}(\theta)|d\theta\\
		&\geq \frac{1}{8}\left(\frac{1}{200\pi^3nmp\sigma^2\left(1 +\frac{\|\mu\|^2}{\sigma^2}\right)^{\frac{1}{m}}}\right)^{\frac{m}{2}} > 0.
	\end{align*}
	Now, recall that $H = D + R$, where $R\sim \mathcal{N}(0,\gamma \mathrm{I}_m)$. By \eqref{eq:p-bound}, $n > \left(10m\left(1 +\frac{\|\mu\|}{\sigma}\right)\right)^{20}$. Hence, by applying Lemma \ref{lem:chisquareconc} again,
	$$\Pr\left(\|R\| \geq \sqrt{\gamma n}\right) \leq e^{-\frac{n}{3}} \leq \frac{1}{2}f_H(t).$$
	We conclude that with probability $1-e^{-\Omega(\kappa n)}$ over $A$, there exists some $T \subset [n]$ and some $v \in \mathbb{R}^m$, with $\|v\| \leq \exp\left(-\frac{\kappa^3pn}{80m}\right)\sigma n^2$, such that
	$$A\mathbf{1}_T- v = t,$$
	Or, in other words, $\|A\mathbf{1}_T - t\| = \|v\| \leq \exp\left(-\frac{\kappa^3pn}{80m}\right)\sigma n^3.$
	Finally, we finish, as in the proof of the discrete case, with the multiplicative Chernoff bound:
	\begin{align*}
		\Pr[|S|\notin [0.5pn,1.5pn]]&=\Pr\left[\left||S|-\Exp[|S|]\right|\geq \frac12\Exp[|S|]]\right]\leq 2\exp(-pn/12).
	\end{align*}
	Again, since $\left(\left(1 +\frac{\|\mu\|}{\sigma}\right)10m\right)^{20} \leq pn$, we have,
	$$\Pr\left[\|D-t\| < \exp\left(-\frac{\kappa^3pn}{80m}\right)\sigma n^2 \right]>2\exp(-pn/12),$$
	 which implies, by a union bound, the existence of a suitable $T$ with $|T|\in [0.5pn,1.5pn]$ and $$\|A\mathbf{1}_T-t\| \leq \exp\left(-\frac{\kappa^3pn}{80m}\right)\sigma n^3.$$
\end{proof}  \section{Integrality Gap Bounds}
\label{sec:gap-bounds}
\subsection{Linear Programs and their Duals}
\label{sec:lp-prelim}

We begin with the basic linear programs relevant to this work. We will examine the integrality gap with respect to primal LP defined as follows:
\begin{align*}
\val_\mathsf{LP}(A,b,c) := \max_x         & \quad \val_c(x)=c^\T x \\
\text{s.t. } & Ax\leq b, x \in [0,1]^n,      \tag{Primal LP}\label{primal-lp}
\end{align*}
This LP has the corresponding dual linear program, which we can express in
the following convenient form:
\begin{align*}
\val^*_\mathsf{LP}(A,b,c) := \min_u         & \quad \val_b^*(u)=b^\T u + \left\|\left(c-A^\T u\right)^+\right\|_1 \\
\text{s.t. } & u \geq 0.  \tag{Dual LP}\label{dual-lp}          
\end{align*}
By strong duality, assuming \eqref{primal-lp} is bounded and feasible, we have that
$\val_\mathsf{LP}(A,b,c)=\val^*_\mathsf{LP}(A,b,c)$.

For any primal solution $x$ and dual solution $u$ to the above pair of
programs, we will make heavy use of the standard formula for the primal-dual
gap:
\begin{align*} 
\val^*_b(u)-\val_c(x) &= b^\T u + \left\|\left(c-A^\T u\right)^+\right\|_1 - c^\T x\nonumber\\
&=(b-Ax)^\T u + \left(\langle x, (A^\T u - c)^+\rangle + \langle {\bf 1}_n-x, (c-A^\T u)^+\rangle \right).  \tag{Gap Formula}\label{primal-dual gap}
\end{align*}

In the sequel, we will let $x^*$ denote the optimal solution to
\ref{primal-lp} and $u^*$ denote the optimal solution to \ref{dual-lp}.
For all the LP distributions we work with, the objective $c$ is continuously
distributed (either Gaussian or exponentially distributed), from which it can
be verified that conditioned on the feasibility of \ref{primal-lp} (which
depends only on $A$ and $b$) both $x^*$ and $u^*$ are uniquely defined almost surely. Moreover, if $i \in [n]$, we shall use $A_{.,i}$ to refer to the $i^{\mathrm{th}}$ column of $A$ and extend this definition to other matrices as well.

Once the optimal solution is found for \ref{primal-lp}, one can round its fractional coordinates to an integral vector. While the rounded vector may not be a feasible solution, we shall use the fact that, as long as the $A_{.,i}$ are sufficiently bounded, it cannot be very far from a feasible solution.
\begin{lemma}[{\cite[Lemma 7]{BDHT22}}] \label{lem:rounding}
	There exists $x' \in \{0,1\}^n$, such that,
	$$\|A(x^* - x')\| \leq \sqrt{m}\cdot \max\limits_{i\in [n]}\|A_{.,i}\|.$$
\end{lemma} 
For the optimal solution $x^*$, define,
\begin{equation} \label{eq:zero coordinates}
N_0 := \{ i \in [n]| x_i^* = 0\} , \text{and } N_1 := \{ i \in [n]| x_i^* = 1\}.
\end{equation}
Let $W$ be the matrix with columns $W_{\cdot, i}=\begin{bmatrix}
c_i&A_{\cdot, i}
\end{bmatrix}^\T$.
The distribution of the columns of $W$ with indices in $N_0$ plays an important role in the proofs of \cref{thm:ipgap_centered,thm:ipgap_packing}. The following lemma essentially says that conditioning on the set $N_0$ and on the values of the non-$0$-columns preserves the mutual independence of the $0$-columns. The conditional distribution of the the $0$-columns is also identified. The reader is referred to \cite[Lemma 5]{BDHT22} for the proof.
\begin{lemma} \label{lem:conditionaldist}
	Let $N \subset [n]$. Conditional on $N_0 = N$ and on the values of sub-matrix $W_{\cdot, [n] \setminus N}$, $x^*$ and $u^*$ are almost surely well defined. Moreover, if $i \in N$, then $W_{\cdot, i}$ is independent from $W_{\cdot, N \setminus \{i\}}$ and the conditional law $W_{\cdot, i}\mid i\in N$ is the same as $W_{\cdot, i}\mid \ust A_{\cdot,i}-c_i > 0$.
\end{lemma}
\subsection{The Gap Bound for Centered IPs}

In this subsection we will prove \cref{thm:ipgap_centered}. In the setting of \cref{thm:ipgap_centered}, the
objective $c \in \R^m$ has independent standard Gaussian entries, and the $m
\times n$ constraint matrix $A$ has independent columns which are distributed
as either one of the following two possibilities:
\begin{itemize}
	\item (LI) Isotropic logconcave distributions with support bounded
	by $O(\sqrt{\log n}+\sqrt{m})$.
	\item (DSU) Vectors with independent entries, uniform on a discrete symmetric interval of size $k \geq 3$. 
\end{itemize}

To simplify the notation in the discrete case, we
divide the constraint matrix $A$ and the right hand side $b$ by $k$ (which
clearly does not restrict generality). Thus, in the discrete case (DSU), we will assume that
entries of $A$ are uniformly distributed in $\{0,\pm 1/k,\dots,\pm 1\}$ and
that the right hand side $b \in \Z^m/k$ satisfies $\|b^-\|_2 \leq O(n)$. In
this way, the discrete case is usefully viewed as a discrete approximation of
the continuous setting where the entries of $A$ are uniformly distributed in
$[-1,1]$ (note that the covariance matrix of each column here is $\mathrm{I}_m/3$, and
thus essentially isotropic).   

With the above setup, our goal is to show that $\mathsf{IPGAP}(A,b,c) =
O(\frac{\poly(m)(\log n)^2}{n})$ with probability $1-n^{-\poly(m)}$. 

\subsubsection{Properties of the Optimal Solutions}
To obtain the gap bound, we will need to show $|N_0| = \Omega(n)$ and that $u^*$, the optimal dual solution, has small norm. This is
given by the following lemma, which is a technical adaptation of \cite[Lemma
8]{BDHT22}.

\begin{lemma}
\label{lem:solution-props}
For $A \in \R^{m \times n}$, $n \geq 10^5 m$, distributed as (LI) or (DSU),
$c \sim \mathcal{N}(0,\mathrm{I}_m)$, $\|b^-\| \leq \frac{n}{12\sqrt{2}}$ with probability
at least $1 - e^{-\Omega(n)}$, we have $\|u^*\| \leq 32$ and
$\left|N_0\right| \geq \frac{n}{10^5}$.
\end{lemma}

To prove this, we need two key lemmas. The first lemma will provide a good
approximation for the value of any dual solution. 

\begin{lemma} \label{lem:embedding}
Let $W^\T := (c, A^\T)$ where $c \sim \mathcal{N}(0,\mathrm{I}_n)$ and $A \in
\R^{m \times n}$ is distributed as (LI) or (DSU). Then, for $n = \Omega(m)$,
we have that
$$\Prob\left[\exists v \in \mathbb{S}^{m} : \|(v^\T W)^+\|_1 \notin \left[\frac{n}{12},\frac{3n}{4}\right]\right] \leq e^{-\Omega(n)}.$$
\end{lemma}
\begin{proof}[\ifflag Proof \else Proof of \Cref{lem:embedding} \fi]
    Fix $v \in \mathbb{S}^m$. We wish to understand $\Pr\left[\|(v^\T W)^+\|_1 \notin [\frac{n}{8}, \frac{5n}{8}]\right].$
    Let $i \in [n]$, we first claim
    \begin{equation} \label{eq:l1-l2}
    \frac{1}{6}\leq \EE\left[(v^\T W)^+_i\right]\leq \frac{1}{2}.
    \end{equation}
    To see the right inequality, by Proposition \ref{prop:pos-part}, $\EE\left[(v^\T W)^+_i\right] = \frac{1}{2}\EE\left[|v^\T W|^+_i\right]$. Now observe that every entry has variance at most $1$, so with Jensen's inequality $$\EE\left[(v^\T W)^+_i\right] = \frac{1}{2}\EE\left[|v^\T W|^+_i\right] \leq \frac{1}{2}\sqrt{\Var(|v^\T W|^+_i)}\leq \frac{1}{2}.$$

    For the left inequality, if the columns of $W$ are isotropic log-concave (recall that the standard Gaussian is also log-concave) Lemma \ref{lem:lck} to get,
    $$ \frac{1}{6}\leq \frac{1}{2\sqrt{e}}\leq\frac{1}{2}\EE\left[|v^\T W|^+_i\right] =\EE\left[(v^\T W)^+_i\right].$$
    If the columns of $W$ are discrete, by Proposition \ref{prop:dsu} every entry satisfies, $\mathrm{Var}\left(W_{ij}\right)\geq \frac{1}{3}$. Hence,
    $\Var\left((v^\T W)_i\right) \geq \frac{1}{3}$. Moreover, Proposition \ref{prop:dsu} also implies, $\EE[W^4_{ji}]\leq 3\EE[W^2_{ji}]^2$. So, by Khinchine's inequality in Lemma \ref{lem:khinchine},
    $$\frac{1}{6}\leq \frac{\sqrt{\Var\left((v^\T W)_i\right)}}{2\sqrt{3}} \leq\frac{1}{2}\EE\left[|v^\T W|^+_i\right] =\EE\left[(v^\T W)^+_i\right].$$
    Now, having established \eqref{eq:l1-l2}, we can bound $\Pr\left[\|(v^\T W)^+\|_1 \notin [\frac{n}{8}, \frac{5n}{8}]\right].$ In the log-concave case, Lemma \ref{lem:trun-logcon} immediately gives,
    $$ \Pr\left[\|(v^\T W)^+\|_1 \notin [\frac{n}{8}, \frac{5n}{8}]\right] \leq e^{-\Omega(n)}.$$
    In the discrete case, by Lemma \ref{lem:disc-sub}, every entry of $W$ is $1$-sub-Gaussian, and Lemma \ref{lem:subg-pos} shows that $(v^\T W)^+_i - \EE\left[(v^\T W)^+_i\right]$ -is $\sqrt{2}$-sub-Gaussian. After summing the coordinates we get that $\|(v^\T W)^+\|_1 - \EE\left[\|(v^\T W)^+\|_1\right]$ -is $\sqrt{2n}$-sub-Gaussian. Applying \eqref{eq:gauss-tail}, we can thus conclude a corresponding probability bound, as in the previous display.\\

    We now turn to consider the entire sphere. Fix $\eps$ to be a small constant and let $N_\eps \subset \mathbb{S}^{m-1}$ be an $\eps$-net. It is standard to show that one may take $|N_\eps| \leq \left(\frac{3}{\eps}\right)^{m}$.
    Hence, by applying a union bound,
    \begin{align*}
    \PP\left(\exists v \in N_\eps : \|(v^\T W)^+\|_1 \notin \left[\frac{n}{8}, \frac{5n}{8}\right]\right) \leq \left(\frac{3}{\eps}\right)^me^{-\Omega(n)} \leq e^{-\Omega(n)},
    \end{align*}
    where the last inequality holds when $n = \Omega(m).$

    Let us denote by $E$ the event considered above and for $u \in \mathbb{S}^{m-1}$ let $\tilde{u} \in N_\eps$, with $\|u - \tilde{u}\|_2 \leq \eps$. Under $E$, we have,
    \begin{align*}
    \max\limits_{u \in \mathbb{S}^{m-1}}\|(u^\T W)^+\|_1 &\leq \min\limits_{v\in N_\eps}\|(v^\T W)^+\|_1 + \|((u-\tilde{u})^\T W)^+\|_1\\
    &\leq \frac{5}{8}n + \eps	\max\limits_{u \in \mathbb{S}^{m-1}}\|(u^\T W)^+\|_1 ,
    \end{align*}
    which is equivalent to,
    $$\max\limits_{u \in \mathbb{S}^{m-1}}\|(u^\T W)^+\|_1 \leq \frac{5}{8(1-\eps)}n.$$
    On the other hand,
    \begin{align*}
    \min\limits_{u \in \mathbb{S}^{m-1}}\|(u^\T W)^+\|_1 &\geq \min\limits_{v\in N_\eps}\|(v^\T W)^+\|_1 - \|((u-\tilde{u})^\T W)^-\|_1\\
    &\geq \frac{n}{8} - \eps	\max\limits_{u \in \mathbb{S}^{m-1}}\|(u^\T W)^+\|_1\\
    &\geq \frac{n}{8} -  \eps \frac{5}{8(1-\eps)}n.
    \end{align*}
    Choose now $\eps = \frac{5}{212}$ to conclude,
    $$\frac{n}{12}\leq\min\limits_{u \in \mathbb{S}^{m-1}}\|(u^\T W)^+\|_1 \leq \max\limits_{u \in \mathbb{S}^{m-1}}\|(u^\T W)^+\|_1 \leq \frac{3n}{4}.$$
\end{proof}
 
The second lemma will imply that any LP solution with large support must
have small objective value. 

\begin{lemma}\label{lem:cubebound}
Let $c \sim \mathcal{N}(0,\mathrm{I}_n)$. Then, for every $\alpha \in [0, 2\sqrt{\log(2)}]$,

$$\Prob\left[\max\limits_{x \in \{0,1\}^n, \ \|x\|_1\geq \beta n} c^\T x \geq \alpha n\right]\leq e^{\frac{-\alpha^2n}{2}},$$
		where $\beta \in [1/2,1]$ is such that $H(\beta) \leq \frac{\alpha^2}{4}$, where $H(p)=-p\log p-(1-p)\log(1-p)$, $p \in [0,1]$, is base $e$ entropy.
\end{lemma}
\begin{proof}[\ifflag Proof \else Proof of \Cref{lem:cubebound} \fi]
    For any $x \in \{0,1\}^n$, $c^\T x \sim \mathcal{N}(0,\|x\|_2^2)$ and thus, by \eqref{eq:gauss-tail},
    $$\Prob\left(c^\T x \geq \alpha n\right) \leq e^{-\frac{\alpha^2n^2}{2\|x\|_2^2}} \leq e^{-\frac{\alpha^2n}{2}}.$$
    We now apply a union bound,
    \begin{align*}
    \Prob\left(\max\limits_{x \in \{0,1\}^n, \ \|x\|_1\geq \beta n} c^\T x \geq \alpha n\right) &\leq \left|\{x \in \{0,1\}^n, \ \|x\|_1 \geq \beta n\}\right|e^{-\frac{\alpha^2 n}{2}}
    \leq e^{H(\beta)n}e^{-\frac{\alpha^2 n}{2}} \leq e^{-\frac{\alpha^2n}{4}}.
    \end{align*}
\end{proof} 
We now have the ingredients to prove the main lemma.

	\begin{proof}[Proof of \Cref{lem:solution-props}]
	For the proof, we will consider the extended matrix $W^\T := (c, A^\T)$. We begin by showing that, for the optimal solution, $c^\T x^*$ is large. Let $u \geq 0$ be any dual solution. Then, under the complement of the event defined in Lemma \ref{lem:embedding} for $W$, using \eqref{dual-lp},
	\begin{align}
		\mathrm{val}^*_b(u) &= b^\T u + \|(c-A^\T u)^+\|_1 \geq -\|b^-\|\|u\| + \|((1,-u)^\T W)^+\|_1 \nonumber \\
		&\geq -\|b^-\|\|u\| + \sqrt{1+\|u\|^2}\frac{n}{12}
		\geq \frac{n}{12}\left(-\frac{\|u\|}{\sqrt{2}}+ \sqrt{1 +\|u\|^2}\right)
		\geq \frac{n}{12\sqrt{2}} \label{eq:opt-lb}.
	\end{align}
	The second inequality is the lower bound in Lemma \ref{lem:embedding} and the last inequality follows since the function $\sqrt{1+t^2} - \frac{t}{\sqrt{2}}$ is minimized at $t = 1$. A lower bound on $c^\T x^*$ follows by noting,
	$$c^\T x^* = \mathrm{val}_c(x^*) =  \mathrm{val^*}_b(u^*).$$
		
	We now prove that $\|u^*\|$ cannot be too large. Again, under the complement of the event in Lemma \ref{lem:embedding}, but using the upper bound this time,
	\begin{align*}
		\frac{3n}{4} &\geq \|((1,0)^\T W)^+\|_1 = \|c^+\|_1
		=\mathrm{val^*}_b(0) \geq \mathrm{val^*}_b(u^*)\\
		&\geq \frac{n}{12}\left(-\frac{\|u^*\|}{\sqrt{2}} + \sqrt{1 +\|u^*\|^2}\right)
		\geq \frac{n}{12}\left(1-\frac{1}{\sqrt{2}}\right)\|u^*\|,
	\end{align*}
	where in the third inequality we have applied \eqref{eq:opt-lb} to $v^*$.
	Thus, rearranging we get $\|u^*\|_2 \leq \frac{9\sqrt{2}}{\sqrt{2}-1} \leq 32$.
	Finally, we show that the optimal solution has many $0$ coordinates. Since $x^*$ has at most $m$ fractional coordinates,
	$$\left|\{i \in [n]\ |\  x_i^* = 0\}\right| \geq n - m - \left|\{i \in [n]\ |\  x_i^* = 1\}\right|.$$
	Since, by assumption, $n \geq 10^5 m$, to finish the proof it will suffice to show $\left|\{i \in [n]\ |\  x_i^* = 1\}\right| \leq \left(1-\frac{2}{10^5}\right)n$. Define $\bar{x}$ by,
	$$\bar{x}_i := \begin{cases}
		x^*_i& \text{if } x^*_i \in \{0,1\}\\
		1& \text{if } x^*_i \notin \{0,1\}	\text{ and } c_i \geq 0\\
		0& \text{if } x^*_i \notin \{0,1\}	\text{ and } c_i < 0\\
		\end{cases}.
	$$
  	Letting $\alpha = \frac{1}{12 \sqrt{2}}$ and $\beta = 1-\frac{2}{10^5}$, a calculation reveals that $H(\beta) \leq \frac{1}{4 \alpha^2}$. By~\eqref{eq:opt-lb}, we have
	$$c^\T \bar{x} \geq c^\T x^* \geq \alpha n,$$
	and by conditioning on the complement of the event in Lemma \ref{lem:cubebound} with $\beta$ and $\alpha$ as above,
	$$\beta n \geq \left|\{i \in [n]\ |\  \bar{x}_i = 1\}\right| \geq \left|\{i \in [n]\ |\  x_i^* = 1\}\right|.$$
	The proof concludes by applying the union bound to the events in Lemmas \ref{lem:cubebound} and \ref{lem:embedding}.
	\end{proof}
 
		\subsubsection{Conditional Distribution of $0$-columns of IP}

		Let $B$ be a random variable with the same distribution as the columns of $A$.
		By \cref{lem:anti-concentration}, $B$ satisfies  \eqref{eq:anticoncentration} with constant $\kappa\leq 1$.
		Define $C:=\frac{\sqrt{150}\|\us\|}{\sqrt{\kappa}}$.
		We first show that the anti-concentration property is unaffected if we condition $B$ on a strip of width $2C$.
		\begin{lemma} \label{lem:restriction}
			Let $B'$ have the law of $B$, conditioned on $|\ust B|\leq C$. Then,
			\begin{enumerate}
				\item $\Pr[|\ust B|\leq C]\geq 1-\frac{\kappa}{150}$.
				\item We have $\frac1{10}\mathrm{I}_m\preccurlyeq \Cov(B')\preccurlyeq 2\mathrm{I}_m$.
				\item If $B$ is (DSU), then $B'$ is symmetric and anti-concentrated with parameter $\kappa/2$.
				\item If $B$ is (LI), $B'$ is logconcave.
			\end{enumerate}
		\end{lemma}
		\begin{proof}[{\ifflag Proof \else Proof of \Cref{lem:restriction}\fi}]
			Let $E=\{a\in \mathbb{R}^m: |\ust a|\leq C \}$.
			From Chebyshev's inequality, and since distributions we consider satisfy $\Cov(B) \preceq \mathrm{I}_m$,
			$$\Pr\left(B \in E\right) \geq 1- \frac{\mathbb{E}\left[(\ust B)^2\right]}{C^2}\geq 1- \frac{\|\us\|^2}{C^2} \geq 1-\frac{\kappa}{150},$$
			which is the first claim. 
			
If $w\in \mathbb{R}^m$, then
\begin{align*}
\mathbb{E}\left[|w^\T B'|^2\right] &= \frac{\mathbb{E}\left[|w^\T B |^2{\bf{1}}_E\right]}{\Pr\left(B \in E\right)}\leq 2\mathbb{E}\left[|w^\T B'|^2\right]\leq 2\|w\|^2.
\end{align*}
In the (DSU) case, to lower bound $\Cov(B')$, we note that by \cref{prop:dsu}, $\EE[W^4_{ji}]\leq 3\EE[W^2_{ji}]^2$. As a consequence, \cref{lem:khinchine} implies  $\sqrt{\Exp[|w^\T B|^4]}\leq \sqrt{3}\Exp[|w^\T B|^2]$.
By the Cauchy-Schwarz inequality, $\Exp[|w^\T B\cdot \mathbf{1}_{B\notin E}|^2]\leq \sqrt{\Exp[|w^\T B|^4]\cdot \Exp[\mathbf{1}_{B\notin E}^4]}\leq \sqrt{3\Pr[B\notin E]}\Exp[|w^\T B|^2]\leq \sqrt{\frac{\kappa}{50}}\|w\|^2$. By \cref{prop:dsu} we have $\Exp[|w^\T B|^2 ]\geq \frac13 \|w\|^2$. So:
\begin{align*}
\Exp[|w^\T B'|^2] = \frac{\Exp[|w^\T B|^2] - \Exp[|w^\T B|^2 \cdot \mathbf{1}_{B\notin E}]}{\Pr[B\in E]} \geq \frac{\frac13\|w\|^2 - \sqrt{\frac{\kappa}{50}}\|w\|^2 \cdot }{1-\frac\kappa{50}}\geq \frac{1}{10}\|w\|^2,
\end{align*}
proving the second claim for the (DSU) case.

In the (LI) case, $\langle w, B\rangle $ is logconcave. By \cref{lem:logconcave-density-bounded}, $f_{\inner{B,v}}\leq \frac{1}{\sqrt{\Var{\inner{B,v}}}}=1$. So, $f_{\langle v, B'\rangle}\leq \frac{1}{1-1/150}f_{\inner{B,v}}=\frac{150}{149}$. 
Now, \cref{lem:bound_pdf_variance} implies that $\Var(\inner{B'', v}) \geq \frac{1}{13}$.

Now, if $B$ is (DSU), it is symmetric, and because conditioning on a symmetric set preserves symmetry, so is $B'$. Consequently, in this case, $\Exp[B'] =0$.
Now set $\sigma':=\sqrt{\|\Cov(B')\|_\mathrm{op}}\leq \sqrt{2}$.
For the third claim, let $I(\theta) =\{a \in \mathbb{R}^m: d(\theta^\T a,\mathbb{Z}) \geq \frac{\kappa}{2}\min\left(1,\sigma' \|\theta\|_\infty\right) \}.$ Choose an arbitrary $\nu \in \mathbb{R}^n$. 
By the symmetry of $B'$ we have:
\begin{align*}
\Pr\left[B' \in I(\theta) \mid \inner{B,\nu}\leq 0 \right]
&= \frac{\Pr\left[B  \in I(\theta) \cap E \mid \inner{B,\nu} \leq 0\right]}{\Pr[B\in E \mid \inner{B,\nu} \leq 0]} \\&\geq \Pr\left[B \in I(\theta) \mid \inner{B,\nu} \leq 0\right] - 2\Pr\left[B \notin E\right]
\geq \frac{\kappa}{2}.
\end{align*}
The last inequality follows from \eqref{eq:anticoncentration},
$$\Pr\left[B \in I(\theta)\mid \inner{B,\nu} \leq 0 \right]\geq \Pr\left[d(\theta^\T B,\mathbb{Z}) \geq \frac{\kappa}{2}\min\left(1,\sigma' \|\theta\|_\infty\right) \mid \inner{B,\nu} \leq 0\right]\geq \kappa.$$
This shows anti-concentration when $B$ is (DSU).

If $B$ is (LI), then so is $B'$ because it is a restriction to a convex set, proving the last claim.

\end{proof}
 		
	In the proof of \cref{thm:ipgap_centered}, we work with the columns of $A$ that have negative reduced cost. We show that we can convert their distribution into the distribution of $B'$, by using rejection sampling.
	\begin{lemma}
	Let $B=A_{\cdot, i}$ and let $B', c'$ have $\Law(B', c_i') = \Law(B,c_i\mid \ust B-c_i\geq 0)$. When $\delta \leq C$, there exist a rejection sampling procedure $\psi$ such that:
	\begin{itemize}
		\item $\mathrm{Law}(B'|\psi(B',c_i')=\mathsf{accept})=\mathrm{Law}(B \mid |\ust A_{\cdot, i}| \leq C).$
		\item $\Pr[\psi(B',c_i')=\mathsf{accept}] \geq \delta \exp\left(-\frac{10^{4}}{\kappa}\right).$
	\end{itemize}  \label{lem:rejection-sampling-cont}
	\end{lemma}
	\begin{proof}[\ifflag Proof \else Proof of \cref{lem:rejection-sampling-cont}\fi]
	
	Let $c''$ and $B''$ be independent random variables with $c''\sim \operatorname{Unif}(0, \delta)$ and let $\Law(B'')=\Law(B\mid |\ust B| \leq C)$. Now we will apply \cref{lem:rejection-sampling-prelim} to transform $(B', c')$ into $(B'', c'')$ using rejection sampling. Observe that for any $\bar{B}\in \mathbb{R}^m$ with $\ust \bar{B}-\bar{c} \in [0,\delta ]$:
	\begin{align*}
		\frac{f_{(B'', c'')}(\bar{B}, \bar{c})}{f_{(B', c_i')}(\bar{B}, \bar{c})} &
		=\frac{f_{B}(\bar{B})/ \Pr[|\ust B | \leq C] / \delta }{f_{B}(\bar{B}) f_{c_i}(\bar{c})/ \Pr[\ust B -c \geq 0]}
		=\frac{\Pr[\ust B -c \geq 0]}{\delta f_{c_i}(\bar{c})\Pr[|\ust B | \leq C]}\\&
		\leq \frac{1-\frac\kappa{150} }{\delta \exp(-\frac12 (C+ \delta)^2)/\sqrt{2\pi}}\leq \frac{3\exp(2C^2)}{\delta }.
	\end{align*}
	For the last inequality we assumed that  $\delta \leq C$. 	By \cref{lem:rejection-sampling-prelim} we, using rejection sampling we can turn $(B', c')$ into $(B'', c'')$ with success probability $\frac{\delta}{3\exp(\frac12(C+\delta)^2)}$. Since $C=\frac{\sqrt{2}\|\us\|}{\sqrt{\kappa}}$, by Lemma \ref{lem:solution-props} , the success probability is at least
	$\frac{\delta }{3\exp(2C^2)}\leq \delta \exp\left(-\frac{10^{4}}{\kappa}\right)$.

\end{proof}
 
	When the columns of $A$ are continuously distributed, we have to be more careful, because the distribution that we obtain from rejection sampling is not necessarily symmetric. As a result, $B$ will not necessarily be mean-zero. We apply another step of rejection sampling to handle this case.
	\begin{lemma}
		Let $B'$ denote the random variable $B \mid |\ust B| \leq C$. If $B'$ is logconcave, then there exists a rejection sampling procedure $\psi$, such that $\Pr[\psi(B')=\mathsf{accept}]=\Omega(1)$, and such that the random variable $B''$ with $\Law(B'')=\Law(B'|\psi(B')=\mathsf{accept})$ satisfies:
		\begin{enumerate}
			\item $\Exp[B'']=0$.
			\item $\Cov(B'')\succcurlyeq\frac{1}{768}\mathrm{I}_m$.
			\item The law of $B''$ is an admissible distribution, in the sense of Definition \ref{def:addmis}.
			\item $B''$ satisfies (\ref{eq:anticoncentration}) with an $\Omega(1)$ constant.
		\end{enumerate}
	\label{lem:sample-mean-zero}
	\end{lemma}
	\begin{proof}[\ifflag Proof \else {Proof of \cref{lem:sample-mean-zero}} \fi]
	Let $\mu':=\Exp[B']$. By H\"older's inequality, we have any unit vector $v \in \R^m$,
	\begin{align*}
		\E\left[\inner{B', v}\right]&=\frac{\Exp[\inner{B, v} \cdot \mathbf{1}_{|\ust B|\leq C}]}{\Pr[|\ust B|\leq C]}= -\frac{\Exp[\inner{B, v} \cdot \mathbf{1}_{|\ust B|> C}]}{\Pr[|\ust B|\leq C]} \leq \frac{\sqrt{\Exp[\inner{B, v}^2]}\sqrt{\Prob[|\inner{B, \us}|> C]}}{\Pr[|\ust B|\leq C]}
		\\&\leq  \frac{\sqrt{\frac{1}{150}}}{1-1/150}\leq \frac{1}{12}
	\end{align*}
	where the second equality follows since $B$ is isotropic and from
	$$\Exp[\inner{B, v} \cdot \mathbf{1}_{|\ust B|\leq C}] + \Exp[\inner{B, v} \cdot \mathbf{1}_{|\ust B|> C}]=\Exp[\inner{B, v}] = 0.$$
	Recall $C \geq \sqrt{150} \|\us\|$. Hence, the concentration bound in \cite[Lemma 5.7]{LV07}, coupled with the fact that $B$ is isotropic, gives,
	$$\frac{\sqrt{\Exp[\inner{B, v}^2]}\sqrt{\Prob[|\inner{B, \us}|> C]}}{\Pr[|\ust B|\leq C]} \leq \frac{\sqrt{\Prob\left[|\inner{B, \frac{\us}{\|\us\|}}|> \sqrt{150}\right]}}{\Pr[|\ust B|\leq C]} \leq 2e^{-5}\leq \frac{1}{12}.$$
	So,
	$$\|\mu'\| = \sup\limits_{v \in \R^m ,\|v\| = 1}\EE\left[\langle B', b\rangle \right] \leq \frac{1}{12}.$$
	Define a convex subset of $\R^m$ by
	$$M:=\left\{\frac{\Exp[f(B')B']}{1/4}: f\in L_\infty(\R^m, \R),\ 0\leq f(x)\leq 1 \text{ for every } x\in \R^m\ \text{ and } \Exp[f(B')] =\frac{1}{4} \right\},$$ and
	consider an arbitrary halfspace $H$ that contains $-\mu'$.
	By \cref{lem:grunbaum} we have,
	$$\Prob[B'\in H]\geq \Prob[B\in H  ] - \Prob[|\ust B|> C]\geq \frac{1}{e} - \frac{1}{12}-\frac{1}{150} =\frac1{4}.$$
	Therefore, there exists $S\subseteq H$ with $\Prob[B'\in S]=\frac14$.
	Then, $
	\Exp[\mathbf{1}_S(B')]=\frac{1}{4}$ and, $\Exp[B'\mid B' \in S ]=\Exp[\frac{\mathbf{1}_S(B')B'}{1/4}]\in M$, which implies that $M\cap H \neq \emptyset$.
	Because this holds for any $H$ with $-\mu'\in H$, by the convexity of $M$ we have $-\mu'\in M$. Indeed, suppose not, then there is a hyperplane passing at $-\mu'$ which separates it from $M$, which cannot happen.
	We conclude that there exists $f:\R^m \to \R$, with $\|f\|_\infty \leq 1$ and $\Exp[f(B')]=\frac14$, such that $\frac{\Exp[f(B')B']}{1/4}=-\mu'$. Let $g(x)=\frac{f(x)+4}{5}$.

	Let $B''$ be a random variable with $f_{B''}(x)\propto g(x)f_{B'}(x)$. Note that $f_{B''}(x)/f_{B'}(x)=g(x)/\Exp_{B'}[g(B')] \leq 1/\Exp_{B'}[g(B')] \leq 2$.  Now, by \cref{lem:rejection-sampling-prelim} there exists a rejection sampling procedure $\psi$ with  $\Pr[\psi(B')=\mathsf{accept}]=\frac12$ and $\Law(B'\mid \psi(B')=\mathsf{accept})=B''$. Now we will show that $B''$ satisfies the required properties.
	
	Firstly, we have $\Exp[B'']=\frac{\Exp[g(B')B']}{\Exp[g(B')]}=\frac{\Exp[f(B')B'] + 4\Exp[B']}{5\Exp[g(B')]}=\frac{-4\mu'+4\mu'}{5\Exp[g(B')]}=0$, proving the first stated property.
	Secondly, for every halfspace $H$ containing the origin, by the Gr\"unbaum inequality, $\Pr[B\in H]\geq \frac{1}{e}$. So $\Pr[B'\in H]\geq \Pr[B\in H] - \Pr[|\ust B|> C]\geq \frac{1}{e}-\frac{1}{150}\geq \frac{1}{4}$.
	Because $f_{B''}\geq g \cdot f_{B'}\geq \frac{4}{5}f_{B'}$, we have $\Pr[B''\in H]\geq \frac45 \cdot\Pr[B'\in H]\geq \frac{1}{5}\geq\frac{1}{4e^2}$. This proves that $B''$ has an admissible distribution.

	For all unit vectors $v\in \mathbb{R}^m$ we have:
	$$
	f_{\langle v, B''\rangle}(x) \leq \frac{f_{\langle v, B'\rangle}(x)}{\Pr[\psi(B')=\mathsf{accept}]}\leq 2 f_{\inner{v, B'}}(x).
	$$
	This implies $\Var(\inner{v,B''})\preccurlyeq 2\Var(\inner{v,B'})\preccurlyeq2\mathrm{I}_m$.
	Because $\Cov(B')\geq \frac{1}{10}\mathrm{I}_m$ and the fact that $\inner{B',v}$ is logconcave, by \cref{lem:logconcave-density-bounded}, $f_{\inner{B',v}}\leq \frac{1}{\sqrt{\Var{\inner{B',v}}}}\leq 4$. Hence, $f_{\langle v, B''\rangle}\leq 2\cdot f_{\inner{v,B'}} \leq 768$.
	Now, \cref{lem:bound_pdf_variance} implies that $\Var(\inner{B'', v}) \geq \frac{1}{768}$. Hence $\frac{1}{768}\mathrm{I}_m \preccurlyeq\Cov(B') \preccurlyeq 2\mathrm{I}_m$.
	Now by \cref{lem:aclc} anti-concentration holds with a constant parameter.
\end{proof}

\subsubsection{Proof of \Cref{thm:ipgap_centered}}

\begin{proof}[Proof of \cref{thm:ipgap_centered}]
	Consider the optimal solutions $x^*$ and $\us$ to respectively \eqref{primal-lp} and \eqref{dual-lp}. We condition on the event $|N_0| \geq n/10^5$ and $\|u^*\|_2 \leq 32$, where $N_0 := \{i \in [n]: x_i^* = 0\}$. By \cref{lem:solution-props}, this event occurs with probability at least $1-e^{-\Omega(n)}$. Subject to this, we further condition on the exact values of $x^*$, $u^*$. We will show that for every such conditioning, the integrality is gap is small with high probability over the randomness of $A_{N_0}$. 

	Set $\delta:=\frac{\poly(m)\log n}{n}$, where the polynomial factor
is the same one as dictated by Theorem \ref{thm:main}. We now show that we
can construct a large subset $Z \subseteq N_0$, such that the
reduced costs of the variables indexed by $Z$ are small and the columns
$A_{\cdot,i}, i \in Z$, are independent and satisfy the necessary conditions
in order to apply \cref{thm:main} to round $x^*$ to a near optimal solution. By \Cref{lem:conditionaldist}, first note that $(c_i,A_{\dot,i}), i \in N_0$ are independent and distributed according to $\ust A_{\cdot,i} - c_i > 0$.  

	By \cref{lem:rejection-sampling-cont}, using rejection sampling we
can sample a set $Z\subseteq N_0$, such that $\mathrm{Law}(A_{\cdot, i}|i\in
Z)=\mathrm{Law}(A_{\cdot, i}\mid |\ust A_{\cdot, i}| \leq C )$, $\Pr[i\in Z|i\in N_0]=\Omega(\delta)$ and such that
${u^*}^\T A_{\cdot, i}-c_i \in [0,\delta]$ for all $i\in Z$. If columns of
$A$ have a discrete symmetric distribution (DSU), then $\Exp[A_{\cdot, i} | i\in Z]=0$.
Moreover, by \cref{lem:restriction}, the distribution of the columns in $Z$ is admissible
(since it is symmetric), as per Definition \ref{def:addmis}, and satisfies
the \eqref{eq:anticoncentration} property with parameter 
$\kappa = \Omega(1)$. In the logconcave case, (LI), we apply a second round of
rejection sampling to $Z$, as described in \cref{lem:sample-mean-zero}, which
achieves that the law of $A_{\cdot, i}, i \in Z$ is mean-zero,  admissible, and anti-concentrated with parameter $\kappa=\Omega(1)$. Furthermore, this second step of rejection sampling only decreases the probability that $i \in Z$ by at most a constant factor. 

In both cases, we see that $\E[Z] \geq \Omega(\delta |N_0|) =
\Omega\left(\poly(m) \log n \right)$. Thus, by the Chernoff
bound~\eqref{eq:chernoff}, $|Z| \geq \Omega\left(\poly(m) \log n \right)$
with probability at $1-n^{-\poly(m)}$. We now condition on the exact set $Z
\subseteq N_0$ subject to this size lower bound. Note that $A_{\cdot, i}, i
\in Z$, are independent admissible, anti-concentrated with parameter $\kappa = \Omega(1)$ and mean-zero random
vectors.

Set $p=\frac{\epsilon\cdot \kappa^4}{1000m^5}$, where $\epsilon>0$ is chosen small enough to have $\kappa^3 \exp\left(\frac{\kappa^3}{3\cdot80^2pm^3}\right)\geq 50000m^2$. We consider the rounded vector $x'$, from Lemma \ref{lem:rounding}, and define the target, $t:=A(x^*-x')-n^4\exp(-p\kappa^3|Z|/m)\mathbf{1}_m$ in (LI) setting and $t := \floor{kA(x^*-x')}/k$ in the (DSU) setting.
	We will now apply \cref{thm:main} to obtain a set $T\subseteq Z$ such  that $\|\sum_{i\in T}A_{i}-t\|_2\leq n^4\exp(-\kappa^3p|Z|/m)$ in the (LI) setting or $\sum_{i \in T} A_i = t$ in the (DSU) setting. 
	 This will help us both fix the slack introduced by the rounding as well as enforce that the resulting solution to be feasible.

	We now invoke \cref{lem:rounding}, which coupled with $|Z|p = \Omega(\poly(m)\log(n))$ and the fact $\max\limits_{i \in [n]} \|A_{\cdot,i}\| = O(\sqrt{\log(n)}+\sqrt{m})$, shows that, as long as the degree of the polynomial in $\delta$ is large enough,
	\begin{align*}
	\|t\|&\leq \|A(x^*-x')\|+m(n^4\exp(-\kappa^3 p|Z|/m)+1) \\&\leq O(\sqrt{m\log n} + m) = o(\sqrt{p|Z|m}).
	\end{align*} 
	Thus, for large $n$, \cref{thm:main} applies to the matrix $A_{\cdot,Z}$ and $t$ in the (LI) setting and the matrix $kA_{\cdot,Z},kt$ in the (DSU) setting. Thus, with probability $1-e^{-\Omega(p|Z|)} = 1 - n^{-\poly(m)}$ there exists a set $T\subseteq Z$ such that $|T| \leq \frac{3}{2}p|Z|$, and $\|\sum_{i\in T}A_{i} -t\|\leq 32 n^4\exp(-\kappa^3p|Z|/80m)$ in the (LI) setting and $\sum_{i \in T} A_i = t$ in the (DSU) setting.
	
	Now we let $x''=x'+\mathbf{1}_T$. We now show that 
\begin{equation}
Ax'' \leq b \quad \text{ and } \quad \ust (b-Ax'') \leq 1/\poly(n). \label{eq:cent-gap-ingredients} 
\end{equation}

Firstly, in the (DSU) setting, we have
        \[
         Ax'' = Ax' + \floor{k(Ax^*-Ax')}/k \leq Ax^* \leq b,
        \]
        so $x''$ is a feasible integer solution. Take $j \in [m]$ such that $\us_j > 0$. By complementary slackness we have that $(Ax^*)_j = b_j \in \Z/k$. Since $A \in \Z^{m \times n}/k$ and $x' \in \Z^n$, we have that $Ax' \in \Z^m/k$. In particular, 
\[
(Ax'')_j = (Ax')_j + \floor{k(Ax^*-Ax')_j}/k = (Ax')_j + \floor{k(b_j-Ax')_j}/k = (Ax')_j + k(b_j-Ax')_j/k = b_j.
\]   
We conclude that $\ust(b-Ax'') = 0$ as needed.
        
In the (LI) setting, we first note that  
	\begin{align*}
	\|Ax''-(Ax^* - n^4\exp(-\kappa^3 p|Z|/m)\mathbf{1}_m) \|=\|\sum_{i\in T}A_i-t\|\leq n^4\exp(-\kappa^3 p|Z|/m).
	\end{align*}
	So, we must have $Ax''\leq Ax^*\leq b$, and hence $x''$ is a
feasible. Furthermore, by complementary slackness
\begin{align*}
\ust(b-Ax'') &= \ust(Ax^*-Ax'') \leq \|\us\|\|Ax^*-Ax''\| 
\leq 32(\|Ax^*-t\| + \|t-Ax''\|) \\ &\leq 32(m+1)n^4\exp(-\kappa^3 p |Z|/m) \leq 1/\poly(n).
\end{align*}

To conclude, we use $x''$ to bound the integrality gap with the (\ref{primal-dual gap}) applied to $x''$ and $u^*$:
	\begin{align*}
	\mathsf{IPGAP}(A,b,c) &=  \ust (b-Ax'') + \left(\sum_{i=1}^n x''_i(A^\T u^*-c)_i^+ + (1-x''_i)(c-A^\T u^*)_i^+ \right) \\&= \ust (b-Ax'') + \sum_{i\in T} (A^\T u^*-c)_i \quad \text{(by complementary slackness)}\\
	&\leq 1/\poly(n) + |T|\cdot \delta \hspace{5.5em} \text{(by $\eqref{eq:cent-gap-ingredients}$ and $T\subseteq Z$)}\\
	&\leq O\left(\frac{\poly(m)\log(n)^2}{n}\right).
	\end{align*}
\end{proof}

\subsection{The Gap Bound for Packing IPs}
In this section we will prove \cref{thm:ipgap_packing}. Here, the
objective $c \in \R^m$ has independent entries that are exponentially distributed with parameter $\lambda=1$. The $m
\times n$ constraint matrix $A$ has independent columns which are distributed
with (DU) independent entries which are
uniform on the interval $\{1,\ldots,k\}$, $k \geq 3$. 

As in the centered case, we
divide the constraint matrix $A$ and the right hand side $b$ by $k$. So, we will assume that the entries of $A$ are uniformly distributed in $\{\frac 1k,\ldots,1\}$ and
that the right hand side $b$ lies in  $((n \beta,n(1/2-\beta)) \cap \frac{\Z}{k})^m$. This way, we can see this setting as a discrete approximation of the continuous setting where the entries of $A$ are uniformly distributed in $[0,1]$, like in \cite{dyer_probabilistic_1989}. 

We want to show $\mathsf{IPGAP}(A,b,c) \leq
 \frac{\exp(O(1/\beta)) \poly(m)(\log n)^2}{n}$ with probability at least $1-n^{-\poly(m)}$.
 We will do this by first solving a slightly modified version of the LP-relaxation. We choose a $b'<b$. Now we let $x^*$ be the minimizer of \eqref{primal-lp}, where $b$ is replaced by $b'$ and let $u^*$ be the optimal solution to the corresponding \eqref{dual-lp}. We round down the solution, setting $x'_i:=\floor{x^*_i}$. Note that $\|A(x^*-x')\|\leq \sum_{i:x^*_i\in (0,1)} \|A_{\cdot, i}\|\leq m\sqrt{m}$.
 
 Similar to the proof of \cref{thm:ipgap_centered}, our proof proceeds by flipping $x'_i$ to $1$ for a subset of indices for which $x^*_i=0$. By duality, these are columns with $A_{\cdot, i} - c_i \geq 0$. To be able to apply \cref{thm:disc-concrete}, we convert the conditional distribution of the columns of $A$ back into their original distribution using rejection sampling:
\begin{lemma}
	Let $B=A_{\cdot, i}$ and let $B', c'$ have $\Law(B', c'_i) = \Law(B,c_i\mid \ust B-c_i\geq 0)$.
When $\delta \leq 1$, here exists a rejection sampling procedure $\psi$ such that:
\begin{itemize}
\item $\Pr[\psi(B',c'_i)=\mathsf{accept}]\geq \frac{1}{11}\delta \exp(-\|\us\|_1)$.
\item $\Law(B'\mid \psi(B',c'_i)=\mathsf{accept})=\operatorname{unif}((\{\frac1k,\ldots, 1\} \cap [\frac{1}{3m},1])^m )$.
\item $\ust B' -c'_i \in [0,\delta]$ whenever $\psi(B', c'_i)=\mathsf{accept}$.
\end{itemize} 
\label{lem:rejection-sampling-packing}
\end{lemma}
\begin{proof}[\ifflag Proof \else Proof of \cref{lem:rejection-sampling-packing} \fi]
Let $c''$ and $B''$ be independent random variables with $c''\sim \operatorname{Unif}(0, \delta)$ and let $B''\sim \operatorname{Unif}((\{\frac1k, \ldots, 1\} \cap [\frac{1}{3m}, 1])^m )$. Note that $f_{B''}(x)\leq f_{B}(x)/(1-\frac1m)^m$. Now we will apply \cref{lem:rejection-sampling-prelim} to transform $(B', c')$ into $(B'', c'')$ using rejection sampling. Observe that for any $\bar{B}\in \mathbb{R}^m, \bar{c}\in \mathbb{R}$ with $\ust B - c \in [0,\delta]$:
\begin{align*}
	\frac{f_{(B'', c'')}(\bar{B}, \bar{c})}{f_{(B', c')}(\bar{B}, \bar{c})} &
	=\frac{	  f_{B''}(\bar{B}) / \delta }{f_{B}(\bar{B}) f_c(\bar{c})/ \Pr_c[\ust B -\bar{c} \geq 0]}
	\leq \frac{	  f_{B}(\bar{B})/ (1-\frac1m)^m / \delta }{f_{B}(\bar{B}) f_c(\bar{c})}
	\leq \frac{4}{  \delta f_c(\bar{c})}
	\\&\leq \frac{4}{  \delta \exp(- \ust B  -\delta )}\leq \frac{11}{  \delta \exp(- \|\us\|_1 )} .
\end{align*}
The stated result now follows directly by applying \cref{lem:rejection-sampling-prelim} on $(B', c'_i)$ to get a rejection sampling procedure that satisfies  $\Law((B', c'_i)\mid \psi(B', c'_i)=\mathsf{accept} )=\Law(B'',c'')$.
\end{proof}
 In the previous lemma, both the acceptance probability and the maximal size of $\delta$ depend on $\|\us\|_1$. To prevent this from affecting the proof, we will show that with high probability $\Omega(\beta^4)\leq \|u^*\|_1 \leq O(\frac{1}{\beta})$. Because our proof of \cref{thm:ipgap_packing} will rely on flipping the columns for which $x^*_i=0$, we will also show that with high probability the number of these columns is at least proportional to $n$.
\begin{lemma}
	\label{lem:solution-props-packing}
	Consider the packing setting, with $\beta \in (0,1/4)$ and $b' \in
	((n \beta/2,n(1/2-\beta)) \cap \frac1k \Z)^m$. Then, with probability at least $1 - e^{-\Omega(\beta^2n)}$, we have $\Omega(\beta^4)\leq \|u^*\|_1 \leq O(\frac{1}{\beta})$ and $\left|N_0\right| \geq \Omega( \beta^4n) $.
\end{lemma}
\begin{proof}[\ifflag Proof \else Proof of \cref{lem:solution-props-packing} \fi]
	Note that the distribution of the $c_i$'s is exponential and therefore logconcave with $\Exp[c_i]=1$. By \cref{lem:sumlogconcave} we now see that with probability $1-e^{-\Omega(n)}$, we have $3n \geq \sum_{i=1}^n c_i$ and consequently,
	\begin{align*}
	3 n \geq \sum_{i=1}^n c_i = c^\T {\bf 1}_n \geq \val_{\textsf{LP}}(x^*) = \val^*(u^*)\geq \sum_{i=1}^mb'_iu^*_i \geq \frac{n\cdot \beta}{2} \|u^*\|_1.
	\end{align*}
	Hence, we have $\|u^*\|_1\leq\frac{6}{\beta}$, with high probability.

	For the second claim, let $H:(0,\frac12]\to (0,-\log(\frac12)]$ be defined with $H(x)=-x\log x-(1-x)\log(1-x)$.  Set $\alpha:= \min(\frac12\beta,H^{-1}(\frac18\beta^2))$. As $H(x)\leq 2\sqrt{x}$, we have $H^{-1}(x)\geq \frac{x^2}{4}$ and hence $\alpha\geq \frac{1}{256}\beta^4$. Let $x\in \{0,1\}^n$ and suppose that $K:=|\{i:x_i=1\}|\geq (1-\alpha) n$. By first using $b_1 \leq (\frac{1}{2} - \beta)n$, and $\EE[(Ax)_1] = \frac{K}{2}$, and then applying Hoeffding's inequality we see,
	\begin{align*}
	\Pr\left[(A x)_{1}\leq  b'_{1}\right]&\leq\Pr\left[(A x)_{1}\leq  \frac{1-\beta}{2}n\right]=\Pr\left[(A x)_{1}-\frac{1-\alpha}{2}n \leq -\frac{\beta-\alpha}{2}n \right]\\
	&\leq \Pr\left[(A x)_{1}- \frac{K}{2}\leq -\frac{\beta-\alpha}{2}n\right]=
	\Pr\left[(A x)_{1}- \Exp[(A x)_{1}]\leq -\frac{\beta-\alpha}{2}n \right]\\
	&\leq  \exp\left(-(\beta-\alpha)^2n\right)\leq  \exp\left(-\frac14\beta^2n\right)\text{.}
	\end{align*}
	Let $S=\{x \in \{0,1\}^n:  |\{i:x_i=1\}|\geq (1-\alpha)n\}$. Note that by \cref{lem:binomial_bound}, $|S|\leq \sum_{i=0}^{\floor{\alpha n}}\binom{n}{i}\leq \exp(H(\alpha)n)$.
	Taking the union bound over all $x\in S$, we see that
\begin{align*}
	\Pr[\exists x \in S: (A x)_{1}'\leq  b_{1}' ]&\leq |S|\exp(-\frac14\beta^2n)\leq \exp(H(\alpha) n-\frac14\beta^2n)\leq \exp(-\frac18\beta^2n).
	\end{align*} So, with probability at least $1-e^{-\Omega(\beta^2 n)}$ all feasible values $x\in \{0,1\}^n$ have $|\{i:x_i=0\}|\geq \alpha n$ and in particular $|N_0|\geq \alpha n\geq \frac{1}{256}\beta^4n$.
	
	At the same time, observe that when $i\in N_0$, we must have $c_i - \ust A_{\cdot, i}\leq 0$, so in particular $\|\us\|_1\geq c_i$. We have $\Prob[c_i\leq  \log(\frac{1}{1-\alpha/2})]\leq 1-\exp(-\log(\frac{1}{1-\alpha/2})) = \frac{1}{2}\alpha$. By the Chernoff bound \eqref{eq:chernoff} this implies that with probability at least $1-\exp(-\Omega(n))$ we have $|\{i\in [n]:c_i \leq \log(\frac{1}{1-\alpha/2})\}| \leq \frac{3}{4}\alpha n$. If this event holds and at the same time we have $|N_0|\geq \alpha n$, then this implies $\|\us\|_1\geq \log(\frac{1}{1-\alpha/2})$ because otherwise $N_0\subseteq \{i\in [n]:c_i \leq \log(\frac{1}{1-\alpha/2})\}$, contradicting the bounds on their size. So, we can conclude that with high probability we have $\|\us\|_1\geq \log(\frac{1}{1-\alpha/2})\geq -\log(1-2^{-8}\beta^4)\geq 2^{-8}\beta^4$.

\end{proof}
 
\begin{proof}[Proof of \cref{thm:ipgap_packing}]

	Let $r=\ceil{\frac{10^6  m^{12}\log(n)}{s^2 }} $, where $s$ is 	a constant that we will choose later. Let $\mu=\frac{k+ \ceil{\frac{k}{3m}} }{2k} = \E[U]$, where $U \sim \mathrm{Uniform}(\{\frac 1k,\ldots, 1 \} \cap [\frac{1}{3m}, 1])$. Now we define	\begin{equation} \label{eq:reduced_constraint}
	\gamma=\frac{r}{1000m^5} \mu  \text{ and }\ b' = b-\gamma\mathbf{1}.
	\end{equation} 
	Let $x^*$ and $\us$ be the optimal solutions of \eqref{primal-lp} and \eqref{dual-lp} where $b$ is replaced by $b'$. We will assume that $ \frac{\beta^4}{C_1} \leq \|\us\|_1\leq \frac{C_1}{\beta}$ and $|N_0|\geq C_2\cdot \beta^4\cdot n$, for some constants $C_1, C_2$. By  \cref{lem:solution-props-packing} this happens probability $1 - e^{-\Omega(\beta^2n)}$. Subject to this, we condition on the exact values of $x^*$, $\us$.
	
	Let $\delta:=\frac{11\exp\left(C_1/\beta\right)r }{C_2\beta^4n }$. Note that by our assumption that $n \geq \poly(m) \exp(\Omega(1/\beta))$, we may assume that $\delta \leq \beta^4/(C_1 m) \leq \frac{\|\us\|_1}{3m }$. 
	Thus, by \cref{lem:rejection-sampling-packing} we can sample a set $Z\subseteq N_0$ such that for $\mathrm{Law}(A_{\cdot, i}|i\in Z)=\mathrm{Uniform}((\{\frac 1k,\ldots, 1 \} \cap [\frac{1}{3m}, 1])^m)$ and that $\Pr[i\in Z|i\in N_0]=\frac1{11}\delta \exp\left(-\|\us\|_1\right) $. Noting that $\Exp[|Z|]=2r$, by Chernoff's inequality \eqref{eq:chernoff}, with probability at least $1-n^{-\poly(m)}$, we have $|Z|\geq r$. Now we restrict $Z$ to its first $s$ elements, to get $|Z|=r$. Observe that $\Exp[A_{\cdot, i}|i\in Z]=\mu \mathbf{1}$.

	We consider the target vector $t\in \mathbb{R}^m$, defined by:
	\begin{align*}
		t_i:=\begin{cases}
		b_i-(Ax')_i:& \us_i>0\\
	        \floor{\gamma}:& \text{otherwise}
		\end{cases},
	\end{align*}
	which satisfies $(b-Ax'-t)^\T  \us = 0$.
	Our next step will be to apply \cref{thm:disc-concrete} on $kA_{\cdot,Z} \in \Z^{m \times r}$ and $kt \in \Z^m$ with parameter $p=\frac{\gamma}{\mu r}=\frac{1}{1000m^5}$, to get a set $T\subseteq Z$ such that $\sum_{i\in T}A_{i}=t$. Note that we have chosen $\gamma$ and $r$ to have $p^4=\omega(\frac{m^3}{r})$.

	To verify that $t$ is indeed covered by Theorem
\ref{thm:disc-concrete}, we first note that by \cref{lem:conditionaldist} the
columns $kA_{\cdot, i}$ for $i\in Z$ are independent with entries uniformly
distributed in $\{\ceil{k/m},\dots,k\}$. We now show that $t$ is
sufficiently close to the mean $|Z|p\mu$:
	\begin{align*}
	\left\|t-|Z|p\mu | \right\|& =\left\|t- \gamma\mathbf{1} \right\|=\sqrt{\sum_{j:\us_j>0} (b_j-(Ax')_j)^2 + |\{j: \us_j = 0\}|(\gamma-\floor{\gamma})^2} \\ &\leq \sqrt{\sum_{j:\us_j>0} (A(x^*- x')_j)^2 + |\{j: \us_j = 0\}|}
	\leq \sqrt{\sum_{j:\us_j>0} \|x^*-x'\|_1^2 + |\{j: \us_j = 0\}|} \\ &\leq m^{1.5} \leq \frac{s}{1000m^5}\sqrt{r m} = sp\sqrt{|Z|m}\leq s\sqrt{p|Z|m}\text{.}
	\end{align*}
	Now choose the constant $s$ such that the previous inequality implies the condition from \cref{thm:disc-concrete}.
	As a result, with probability  $1-\exp\left(-p|Z|\right) \geq 1-n^{-\poly(m)}$ there exists a set $T\subseteq Z$, such that $\sum_{i\in T}A_{i}=t$. 

	Let $x''=x'+\mathbf{1}_T$. Noting that $x'$ was obtained from $x^*$, for $i$ with $\us_i=0$ we have
	$$(Ax'')_i=(Ax')_i+ t_i \leq b'_i + \gamma = b_i.$$ For, $i$ with $\us_i>0$ we have:
	$$(Ax'')_i=(Ax')_i+ t_i = b_i,$$
	  which means that $x''$ is a feasible solution to the integer program.
	  
	  Using (\ref{primal-dual gap}) for $x''$ and $u^*$, we now get:
	\begin{align*}
	\mathsf{IPGAP}(A,b,c) &= \text{val}_{\mathsf{LP}}(A,b,c) - \text{val}_{\mathsf{IP}}(A,b,c) \leq \text{val}_b^*(\us) - \text{val}_c(x'')\\
	&= b^\T \us + \sum_{i=1}^n (c-A^\T \us)_i^+ - c^\T  x''   \\
	&= (b-Ax'')^\T  \us + \left(\sum_{i=1}^n x''_i(A^\T \us-c)_i^+ + (1-x''_i)(c-A^\T \us)_i^+ \right) \\
	&= (b-Ax'-t)^\T  \us + \sum_{i\in T} (A^\T \us-c)_i  \hspace{1.2cm}  \text{(by complementary slackness)}\\
	&= \sum_{i\in T} (A^\T \us-c)_i  \hspace{3cm} (\text{ since } (b-Ax'-t)^\T  \us=0) \\
	&\leq \delta|T| \leq \frac{(\exp(C_1/\beta)\poly(m)\log n)^2}{n}. \hspace{1cm} \text{(by Lemma \ref{lem:rejection-sampling-packing})}
	\end{align*}
	
\end{proof}

   \section{The Size of the Branch-and-Bound Tree}
\label{sec:bnb}

In this section, we prove \Cref{cor:branch-and-bound} in the
discrete setting. This will follow by adapting the proof of
\Cref{thm:meta-logcon} from~\cite{BDHT22}.

The first ingredient is the key theorem from \cite[Theorem 3]{dey2021branch},
which relates the Branch-and-Bound tree size to the size of a certain knapsack.

\begin{theorem} \label{thm:bnb}
Let $A\in\mathbb{R}^{m\times n}$, $b\in \mathbb{R}^m$, $c\in \mathbb{R}^n$.
Then, the best bound first Branch-and-Bound algorithm applied to
\ref{primal-ip} with data $A,b,c$ produces a tree of size
\begin{equation}
2n \cdot |\{x \in \{0,1\}^n: \sum_{i=1}^n x_i |(A^\T u^* - c)_i| \leq \mathsf{IPGAP}(A,b,c)\}| + 1, \label{eq:bnb-size}
\end{equation}
where $u^*$ is an optimal solution to \eqref{dual-lp}.
\end{theorem}

Given the above, we must show how to upper bound the size of the knapsack
in~\eqref{eq:bnb-size} in the setting of \Cref{cor:branch-and-bound}. Namely,
when the entries of $c \in \R^n$ are either standard Gaussian or exponential
and the entries of $A$ are discrete.  

For this purpose, we will require the following bound on the size on random
knapsack polytopes.

\begin{lemma}[{\cite[Lemma 11]{BDHT22}}]
\label{lem:one-knapsack}
Let $\omega_1,\dots,\omega_n \in \R$ be independent continuous random
variables with maximum density at most $1$. Then, for any $G \geq 0$, we have
\[
\E[|\{x \in \{0,1\}^n: \sum_{i=1}^n x_i |\omega_i| \leq G\}|] \leq e^{2\sqrt{2 n G}}.
\]
\end{lemma}

To get a bound on~\eqref{eq:bnb-size}, we will use the above lemma for $G$
being an upper bound on $\mathsf{IPGAP}(A,b,c)$, using a union bound over all
dual solutions of suitably bounded norm. In particular, we will require an
upper bound on the norm of $u^*$. We note that the proof of
\Cref{thm:meta-logcon} in~\cite{BDHT22} did not require a bound on the norm
of the dual solution. In the logconcave setting, one could get around needing
such a bound by using the anti-concentration properties of the columns of the
objective extended constraint matrix $W^\T := (c,A^\T)$. In the present
setting, we will only be able to rely on the anti-concentration properties of
the objective $c$.

In both the centered and packing case, the coefficients of $c \in \R^n$ are
independent and have maximum density at most $1$. In particular, for any
$u \in \R^m$, the entries of $(A^\T u - c)$ are also independent and have
maximum density $1$. By applying an appropriate union bound (e.g., see the
proof of \cite[Lemma 16]{BDHT22}), one can derive the following bound over a
family of knapsacks:

\begin{lemma}
\label{lem:bnd-norm-bd}
Let $c \in \R^n$ have independent coordinates with maximum density at most
$1$ and let $A \in [-1,1]^{m \times n}$ have independent entries. Then, for
any $n\geq m, R \geq 2, G \geq 1/n$, we have that
\[
\Pr[\max_{\|u\|_2 \leq R} |\{x \in \{0,1\}^n: \sum_{i=1}^n x_i|(A^\T u-c)_i| \leq G\}| \geq (nR)^{\Theta(m)} e^{2\sqrt{2nG}}] \leq n^{-\Omega(m)}.
\]
\end{lemma}
\begin{proof}
Let $K(u,G):=\{x \in \{0,1\}^n: \sum_{i=1}^n x_i|(A^Tu-c)_i| \leq G\}$.
For $i\in [n]$, let $N_i$ be an $\varepsilon$-net of $\{u\in\mathbb{R}^m:\|u\|_2\leq R \}$ for $\epsilon=\frac{1}{n^3}$. Note that we can choose $N$ with $|N|=O(R^{m}n^{3m})$. 

\Cref{lem:one-knapsack} implies that for any $u\in N'$, $\E[|K(u,G)|] \leq e^{2\sqrt{2 n G}}$.
So, by the Markov bound, we see:
\begin{align*}
\Pr[|K(u,G)| \geq n^m|N| e^{2\sqrt{2 n G}}]\leq \frac{1}{n^m|N|}.
\end{align*}
Taking the union bound over all $u'\in N$, we see:
\begin{align*}
\Pr[\exists u'\in N: |K(u,G)| \geq n^m|N| e^{2\sqrt{2 n G}}]\leq |N|\frac{1}{n^m|N|}=n^{-m}
\end{align*}
Now suppose that the event is true. Then for each $u$ with $\|u\|\leq R$ there will be a $u'\in N$ with $\|u-u'\|\leq \epsilon$. This implies that for all $x\in K(u,G)$ we have:
\begin{align*}
\sum_{i=1}^nx_i|(A^\T u' -c)_i| &\leq \sum_{i=1}^n\left(x_i|(A^\T u -c)_i| + x_i|(A^\T(u'-u))_i|\right)\leq \sum_{i=1}^nx_i|(A^\T u -c)_i| + \|A^\T(u'-u)\|_1\\
&\leq \sum_{i=1}^nx_i|(A^\T u -c)_i| + \sqrt{m}n\epsilon\leq G+ \frac{1}{n}.
\end{align*}
So:
\begin{align*}
\Pr[\exists u'\in N: |K(u,G+\frac1{n})| \geq n^m|N| e^{2\sqrt{2 n G}}]\leq n^{-m}.
\end{align*}
Setting $G':= G-\frac1n$, we see:
\begin{align*}
\Pr[\exists u'\in N: |K(u,G')| \geq n^m|N| e^{2\sqrt{2 n G'+2}}]\leq n^{-m}.
\end{align*}
This proves the lemma because $n^m|N| e^{2\sqrt{2 n G'+2}}\leq (nR)^{\Theta(m)}e^{2\sqrt{2 n G'}}$.
\end{proof}
	
We now have all the required ingredients to prove the bound on
Branch-and-Bound trees in the discrete case.

\begin{proof}[Proof of \Cref{cor:branch-and-bound}]
As in section~\ref{sec:gap-bounds}, in both the discrete centered and packing
case, we divide the constraint matrix $A$ by $k$. Thus, the entries of $A$
are either uniform in $\{0,\pm 1/k,\pm 2/k, \dots,1\}$ in the centered case
or in $\{1/k,2/k,\dots,1\}$ in the packing case, and hence all contained in
$[-1,1]$.

Let $K = \{x \in \{0,1\}^n: \sum_{i=1}^n x_i|(A^\T u^*-c)_i| \leq
\mathsf{IPGAP}(A,b,c)\}$. By \Cref{thm:bnb}, to bound the size of the
Branch-and-Bound tree, it suffices to prove a high probability upper bound on
$|K|$.

In the centered case, by \Cref{thm:ipgap_centered} we have that
$\mathsf{IPGAP}(A,b,c)$ is at most $G := \poly(m)(\log n)^2/n$ with
probability $1-n^{-\poly(m)}$, by \Cref{lem:solution-props} that $\|u^*\| \leq
R := 32$ with probability $1-e^{-\Omega(n)}$. Applying \Cref{lem:bnd-norm-bd}
with $G,R$ together with the union bound, we conclude that $|K| \leq
n^{\poly(m)}$ with probability $1-n^{-\Omega(m)}$, as needed.

In the packing case, by \Cref{thm:ipgap_packing} the integrality gap
$\mathsf{IPGAP}(A,b,c)$ is upper bound by $G := \exp(O(1/\beta))\poly(m)(\log
n)^2/n$ with probability $1-n^{-\poly(m)}$, and by
\Cref{lem:solution-props-packing} we have that $\|u^*\|_2 \leq \|u^*\|_1 \leq
R := O(1/\beta)$ with probability $1-e^{-\Omega(n)}$. Applying
\Cref{lem:bnd-norm-bd} with $G,R$ together with the union bound, we conclude
that $|K| \leq n^{\exp(O(1/\beta))\poly(m)}$ with probability
$1-n^{-\Omega(m)}$, as needed.
\end{proof}
 \section{Anti-Concentration Results} \label{sec:ac}
Throughout this section we use the notation,
$$d(x, \Z) := \min\limits_{z \in \Z} |x-z|.$$
Our goal in this section is to prove that Definition \ref{def:AC} is valid for a large family of distributions and prove Lemma \ref{lem:anti-concentration}.
\paragraph*{\bf Distributions with bounded densities: }
We begin with the following simple $1$-dimensional lemma.
\begin{lemma} \label{lem:1dlc}
	Let $X$ be a random variable with $\EE\left[X\right] = \mu$ and $\Var(X) = \sigma^2$. Suppose that $X$ has a density $\rho$, which satisfies,
	$$\rho(x) \leq \frac{C}{\sigma},$$
	for some $C > 0$. Then, for every $\epsilon > 0$, then for $\delta = \frac{\eps^2}{12C}$,
	$$\Pr\left(d(X, \Z) \geq \delta\min \left(1, \sigma\right)\right) \geq 1-\eps.$$
\end{lemma}
\begin{proof}
By Chebyshev's inequality $\Pr\left(|X - \mu| \geq \frac{2}{\eps}\sigma\right) \leq \frac{\eps^2}{4}.$
	Define $\sigma' = \min(1, \sigma)$ and note that if $\Z + [-\delta \sigma', \delta \sigma']:= \bigcap\limits_{z \in \Z} [z -\delta \sigma', z+\delta \sigma']$, then, for any $\delta  > 0$,
	$$\Pr\left(X \in \left[\mu -\frac{2}{\eps}\sigma, \mu + \frac{2}{\eps}\sigma\right] \bigcap \left(\Z + \left[-\delta\sigma', \delta\sigma'\right]\right)\right) \leq \sum\limits_{z\in \Z, |z-\mu|<\frac{2}{\epsilon}\sigma}\int\limits_{z - \delta\sigma'}^{z  + \delta\sigma'}\rho(x)dx.$$
	If $\sigma \geq 1$, since $\rho(x) \leq \frac{C}{\sigma}$,
	$$\sum\limits_{z\in \Z, |z-\mu|<\frac{2}{\eps}\sigma}\int\limits_{z - \delta}^{z  +\delta}\rho(x)dx \leq \frac{6}{\eps}\sigma\cdot 2\delta\cdot \frac{C}{\sigma}.$$
	We now choose $\delta = \frac{\eps^2}{12C}$, so the right hand side becomes smaller than $\frac{\eps}{2}$, and
	\begin{align*}
		\Pr&\left(d(X, \Z) \geq \delta\min \left(1, \sigma\right)\right)\\
		&\geq \Pr\left(|X - \mu| < \frac{\eps}{2}\sigma\right) - \Pr\left(X \in \left[\mu - \delta, \mu + \delta\right] \bigcap \left(\Z + \left[-\delta, \delta\right]\right)\right) \geq 1-\frac{\eps^2}{4} - \frac{\eps}{2} > 1 - \eps.
	\end{align*}
	If $\sigma <1$, then $\sigma' = \sigma$ and 
	$$\sum\limits_{z\in \Z, |z-\mu|<\frac{2}{\eps}\sigma}\int\limits_{z - \delta\sigma}^{z  +\delta\sigma}\rho(x)dx \leq \frac{6}{\eps}\cdot2\delta\sigma\cdot\frac{C}{\sigma}.$$
	We then arrive at the same conclusion.
\end{proof}
We now prove our anti-concentration result measures with an appropriate density bound.
\begin{lemma} \label{lem:aclc}
	Let $X$ be a random vector in $\R^m$ with $\Sigma := \mathrm{Cov}\left(X\right)$. For $\theta \in \R^m$ let $\rho_\theta$ stand for the density of $\langle X,\theta\rangle$. Assume that there is a constant $C > 1$, satisfying the following three conditions:
	\begin{itemize}
		\item For every $\theta \in \R^m$, $x \in \R$, $\rho_\theta(x) \leq \frac{C}{\sqrt{\mathrm{Var}(\langle X, \theta\rangle)}}.$
		\item For every $\nu \in \R^m$, $\Pr\left(\langle\nu, X\rangle \leq \langle\nu,\mu \rangle\right) \geq \frac{1}{C}.$
		\item $\|\Sigma\|_{\mathrm{op}}\|\Sigma^{-1}\|_{\mathrm{op}} \leq C$.
	\end{itemize} Then, for any $\theta,\nu \in \R^m$,
	\begin{align*}
		\Pr\left[d(\theta^\T X, \mathbb{Z}) \geq \frac{1}{48C^3}\min\left(1,\|\theta\|_\infty\sigma\right) \mid\langle\nu, X\rangle \leq \langle\nu,\mu \rangle \right]\geq \frac{1}{2C}, 
	\end{align*}
	where $\sigma:= \|\Sigma\|_{\mathrm{op}}.$
	In other words, $X$ satisfies \eqref{eq:anticoncentration} with $\sigma$ and  $\kappa = \frac{1}{48C^{3}}$.
\end{lemma}
Before proving the result, we just note that by Lemmas \ref{lem:grunbaum} and \ref{lem:logconcave-density-bounded}, Lemma \ref{lem:aclc} applies to isotropic logconcave distributions and hence proves the first half of Lemma \ref{lem:anti-concentration}.
\begin{proof}
Let us denote $\eta^2 := \Var(\theta^\T X)$ and observe
	$$\frac{\|\theta\|^2}{\|\Sigma^{-1}\|_\mathrm{op}}\leq \eta^2 \leq \|\theta\|^2\|\Sigma\|_{\mathrm{op}}.$$
	With this in mind, we will actually show the seemingly stronger result,
	$$\Pr\left[d(\theta^\T X, \mathbb{Z}) \geq \frac{\kappa}{C}\min\left(1,\sigma\right)\mid\langle\nu, X\rangle \leq \langle\nu,\mu \rangle\right]\geq \frac{\kappa}{C}.$$
	However, since the class of measures considered by the lemma is preserved under rotations this is an actual equivalent statement.
	
	By assumption, $\frac{\eta}{\sigma\|\theta\|} \geq \frac{1}{\sqrt{C}}$, and if we choose $\eps = \frac{1}{2C}$ in Lemma \ref{lem:1dlc}, then
	$$\Pr\left[d(\theta^\T X, \mathbb{Z}) \geq \frac{\delta}{\sqrt{C}}\min\left(1,\|\theta\|\sigma\right)\right]\geq 1 - \frac{1}{2C},$$
	where $\delta = \frac{1}{48C^3}$
	To complete the proof, by assumption $\nu$,
	$$\Pr\left(\langle\nu, X\rangle \leq \langle\nu,\mu \rangle\right) \geq \frac{1}{C}.$$
	Thus, with a union bound
	\begin{align*}
		\Pr&\left[d(\theta^\T X, \mathbb{Z}) \geq \frac{\delta}{\sqrt{C}}\min\left(1,\|\theta\|_\infty\sigma\right) \mid\langle\nu, X\rangle \leq \langle\nu,\mu \rangle \right]\\
		&= \frac{\Pr\left[d(\theta^\T X, \mathbb{Z}) \geq \frac{\delta}{\sqrt{C}}\min\left(1,\|\theta\|_\infty\sigma\right) \text{ and }\langle\nu, X\rangle \leq \langle\nu,\mu \rangle \right]}{\Pr\left(\langle\nu, X\rangle \leq \langle\nu,\mu \rangle\right)}\\
		&\geq \Pr\left[d(\theta^\T X, \mathbb{Z}) \geq \frac{\delta}{\sqrt{C}}\min\left(1,\|\theta\|_\infty\sigma\right)\right] + \Pr\left(\langle\nu, X\rangle \leq \langle\nu,\mu \rangle\right) - 1\\
		&\geq \left(\frac{1}{2C} + 1 - \frac{1}{C}\right) + \frac{1}{C} - 1 \geq \frac{1}{2C}.
	\end{align*}
\end{proof}
\paragraph*{\bf Discrete distributions:}
We now prove anti-concentration results for discrete distributions supported on $\Z^m$.
Our result pertains to random variables which are uniform on intervals of length at least $3$.
\begin{lemma} \label{lem:acdiscinterval}
	Let $X = (X_1,\dots,X_m)$ be a random vector in $\mathbb{R}^m$, such that $\{X_i\}_{i=1}^m$ are \emph{i.i.d.} uniformly on $\{a, a+1,\dots,a+k\}$, for some $a, k \in \mathbb{N}$, with $k > 1.$ Set $\mu = \mathbb{E}[X_1]$ and $\sigma = \sqrt{\mathrm{Var}(X_1)}$. Then, for every $\theta \in [-\frac{1}{2}, \frac{1}{2}]^m$, and every $\nu \in \mathbb{R}^m$,
	$$\Pr\left(d(\theta^\T X, \mathbb{Z}) \geq \frac{1}{10}\min\left(\|\theta\|_\infty\sigma,1\right)\mid \langle\nu,X \rangle \leq \langle\nu, \mu\mathbf{1} \rangle\right) \geq \frac{1}{40}.$$
	In other words, $X$ satisfies \eqref{eq:anticoncentration} with $\sigma$ and $\kappa = \frac{1}{40}$.
\end{lemma}
\begin{proof}
	Observe that, as $X$ is  symmetric around its mean,
	\begin{align*}
		\Pr&\left(d(\theta^\T X, \mathbb{Z}) \geq \frac{1}{10}\min\left(\|\theta\|_\infty\sigma,1\right)\mid \langle\nu,X \rangle \leq \langle\nu, \mu\mathbf{1} \rangle \right) \\
		&= \frac{\Pr\left(d(\theta^\T X, \mathbb{Z}) \geq \frac{1}{10}\min\left(\|\theta\|_\infty\sigma,1\right)\text{ and } \langle\nu,X \rangle \leq \langle\nu, \mu\mathbf{1} \rangle \right)}{\Pr\left(\langle\nu,X \rangle \leq \langle\nu, \mu\mathbf{1} \rangle\right)}\\
		&\geq \Pr\left(d(\theta^\T X, \mathbb{Z}) \geq \frac{1}{10}\min\left(\|\theta\|_\infty\sigma,1\right)\text{ and } \langle\nu,X \rangle \leq \langle\nu, \mu\mathbf{1} \rangle \right),
	\end{align*}
	With no loss of generality, let us assume $|\theta_m| = \|\theta\|_\infty$ and consider the event,
	$$E = \left\{\sum\limits_{i=1}^{m-1} \nu_iX_i \leq \mu\sum\limits_{i=1}^{m-1}\nu_i \text{ and } \nu_mX_m\leq \mu\nu_m\right\}.$$
	Clearly, $E\subset \{\langle\nu,X \rangle \leq \langle\nu, \mu\mathbf{1} \rangle\}$, and by symmetry and independence, $\Pr\left(E\right)\geq\frac{1}{4}$.
	With the previous display,
	\begin{align*}
		\Pr&\left(d(\theta^\T X, \Z) \geq \frac{1}{10}\min\left(\|\theta\|_\infty\sigma,1\right)\mid \langle\nu,X\rangle \leq \langle\nu, \mu \rangle \right)\\
		&\geq\Pr\left(d(\theta^\T X, \Z) \geq \frac{1}{10}\min\left(\|\theta\|_\infty\sigma,1\right),\ X \in E \right)\\
		&=\frac{1}{4}\Pr\left(d(\theta^\T X,\Z) \geq \frac{1}{10}\min\left(\|\theta\|_\infty\sigma,1\right)\mid X \in E \right).
	\end{align*}
	Now, denote $r:= \sum\limits_{i=1}^{m-1}\theta_iX_i,$ and rewrite,
	$$d(\theta^\T X, \Z) = d(\theta_mX_m, (\mathbb{Z}-r)).$$
	We observe that under the conditioning on $E$, depending on $\sign(\nu_m)$, $\theta_mX_m$ is either uniform on $\{\theta_ma,\theta_m(a+1)\dots \theta_m\lfloor \mu \rfloor\}$, or on $\{\theta_m\lceil \mu \rceil \dots \theta_m(a + k)\}$.
	We are then interested in the size of the set 
	$$F = \{\theta_m\cdot x : d(\theta_mx, (\mathbb{Z}-r)) \geq \frac{1}{10}\min\left(\theta_m\sigma,1\right) \text{ and } x \in \mathrm{support}(X_m|E)\}.$$
	Since $|\theta_m| \leq \frac{1}{2}$, $\sigma = \sqrt{\frac{k^2 + 2k}{12}}$ and $|\mathrm{Support}(X_m|E)| \geq \lceil\frac{k + 1}{2} \rceil$, it is not hard to see that, as long as $k \geq 2$,
	\begin{equation} \label{eq:support}
		\frac{|F|}{|\mathrm{Support}(X_m|E)|} \geq \frac{1}{10}.
	\end{equation}
	This can be done by inspecting the trajectory $a\theta_m + \{0, \theta_m, 2\theta_m,... , \lceil\frac{k + 1}{2} \rceil\theta_m\} \mod 1$ and noting that at most a $\frac{9}{10}$ fraction of the set $\{0, \theta_m, 2\theta_m,... , \lceil\frac{k + 1}{2} \rceil\theta_m\} \mod 1$ can occupy \emph{any} interval of length $\frac{1}{5}\min\left(\theta_m\sigma,1\right)$.
	
	Indeed, let $I$ be such an interval. If $\sigma_m > \frac{1}{5}\min\left(\theta_m\sigma,1\right)$, then it cannot be the case that  for some $j$, both $j\theta_m, (j+1)\theta_m \in I +\mathbb{Z}$. On the other hand, if $\theta_m \leq \frac{1}{5}\min\left(\theta_m\sigma,1\right)$, then if, for some $j$, $j\theta_m \in I$, necessarily, $(j + \min(\frac{\sigma}{5}, \frac{1}{5\theta_m})\theta_m \notin I$ and \eqref{eq:support} follows since $\min(\frac{\sigma}{5}, \frac{1}{5\theta_m}) \leq \frac{4}{5}k$.
Thus, by invoking the law of total probability on all possible values of $r$,
	$$\Pr\left(d(\theta^\T X, \Z) \geq \frac{1}{10}\min\left(\theta_m\sigma,1\right)\mid X \in E \right) \geq \frac{|F|}{|\mathrm{Support}(X_m|E)|} \geq \frac{1}{10}.$$
\end{proof}
\printbibliography
\appendix
\section{Preliminaries}
\label{sec:prelim}
We begin this section by introducing some notation, to be used throughout the paper.
If $(x_1,\dots,x_m) \in \R^m$ and $p \geq 1$, the $p$-norm is defined by,
$$\|x\|_p:= \left(\sum\limits_{i=1}^m|x_i|^p\right)^{\frac{1}{p}}.$$
When $p=2$, i.e. the Euclidean norm, we will sometimes omit the subscript, so $\|x\| = \|x\|_2$. We interpret $p = \infty$ in the limiting sense,
$$\|x\|_\infty := \max\limits_{1\leq i \leq n} |x_i|.$$
We denote the positive part of $x$ as, $x^+:= (\max(x_1,0),\dots,\max(x_m,0))$ and the negative part $x^-:= (\max(-x_1,0),\dots,\max(-x_m,0))$. In this way, $x = x^+-x^{-}$. The all-ones vector is denoted ${\bf 1}_m:=(1,\dots,1)$. We will sometimes omit the subscript, when the dimension is clear from the context. If $S$ is a subset of indices, we will write ${\bf 1}_S$ for a vector such that $({\bf 1}_S))_i = 1$ if $i \in S$ and $0$ otherwise.
If $a$ and $b$ are quantities that depend on the problem's parameters we will write $a = O(b)$ (resp. $a = \Omega(b)$) to mean $a \leq Cb$, (resp. $a \geq CB$) for some numerical constant $C > 0$. We also write $a \ll b$ or $a = o(b)$ (resp. $a \gg b$ or $a = \omega(b)$) to mean $\lim\limits_{a\to \infty} \frac{a}{b} = 0$ (resp. $\lim\limits_{a\to \infty} \frac{b}{a} = 0$ ).
The identity matrix in $\R^m$ is denoted by $\mathrm{I}_m$.

\subsection{Fourier analysis}
Our main tool for proving the discrepancy result is Fourier analysis and we review here the necessary details. 
Fix $X\sim \mathcal{D}$, a random vector in $\R^m$. The Fourier transform of $X$ (sometimes also called the characteristic function) is the complex-valued function defined by $\hat{X}(\theta):=\Exp[\exp(2\pi i \langle X,\theta\rangle)].$

To understand the natural domain for $\theta$, we first define the (dual) period of $\mathcal{D}$. For this, choose an arbitrary $a\in \mathrm{support}(\mathcal{D})$ and denote,
\begin{equation} \label{eq:dualfunddomain}
\mathrm{period}(\mathcal{D}) := \{v\in \mathbb{R}^m: \langle v, w - a\rangle\in \mathbb{Z},\ \forall w\in \mathrm{support}(\mathcal{D})\}.
\end{equation}
The definition of $\mathrm{period}(\mathcal{D})$ does not depend on the choice of $a$. It is readily seen that when $\mathcal{D}$ is absolutely continuous with respect to the Lebesgue measure $\mathrm{period}(\mathcal{D}) =\{0\}$, while $\mathrm{period}(\mathcal{D}) =\Z^m$, when $\mathrm{support}(\mathcal{D}) \subset \Z^m$. These are the cases on which we focus.
Using the period, we define the fundamental domain of $\mathcal{D}$, in Fourier space (where we suppress the dependence on $\mathcal{D}$),
\begin{equation} \label{eq:funddomain}
V:=\{\theta\in \mathbb{R}^m: \|\theta\|\leq \inf_{0\neq w\in \mathrm{period}(\mathcal{D})}\|\theta-w\|\}.
\end{equation}
Observe that if $\mathcal{D}$ is absolutely continuous with respect to the Lebesgue measure, then $V = \R^m$ and when $\mathrm{support}(\mathcal{D})\subset \Z^m$,  $V=[-\frac12,\frac12]^m$. 

The connection between $X$ and its Fourier transform comes from the Fourier inversion formula \cite[Theorem 1.20]{SW71}:
\begin{theorem}[Fourier inversion formula] \label{thm:fourier}
	Suppose that either $\mathcal{D}$ is absolutely continuous, or $\mathrm{support}(\mathcal{D})\subset \Z^m$. Then, for $t\in \mathrm{support}(\mathcal{D})$:
	\begin{align*}
	\Pr[X=t] = \int_{\theta\in V}\hat{X}(\theta)\exp(-2\pi i\langle \theta, t \rangle )d\theta
	\end{align*}
	If $X$ is absolutely continuous, we interpret $\Pr[X=t]$ as the density of $X$ at $t$. 
\end{theorem}
Another desirable property of the Fourier transform is that it is particularly amenable to convolutions (this is clear from the exponential representation but see \cite[Theorem 3.18]{SW71}).
\begin{theorem}[Multiplication-convolution theorem] \label{thm:multconv}
	Let $X$ and $Y$ be two independent random vectors. Then,
	$$\widehat{\left(X + Y\right)}(\theta) = \hat{X}(\theta)\hat{Y}(\theta).$$
\end{theorem}

\subsection{Probability Distributions}

Let $X \in \R^m$ be a random vector distributed according to a probability
measure $\nu$ on $\R^m$. We will use $f_X$ to refer to the probability density function of $X$. We define the mean $\mathrm{Mean}(\nu):=\E[X] \in
\R^m$ and covariance matrix by $\Cov(\nu) := \Cov(X) := \E[X X^\T] -
\E[X]\E[X]^\T \succeq 0$. If $X \in \R$ is a real random variable, we use the
notation $\Var[X]$ to write the variance instead of $\Cov(X)$. We say that
$X$, or its law $\nu$, is isotropic if and $\E[X] = 0$ and $\Cov(X)=\mathrm{I}_m$. 

\begin{proposition} 
\label{prop:pos-part}
Let $X \in \R^n$ satisfy $\E[X] = 0$. Then $\E[X^+]=\E[|X|]/2$.
\end{proposition}
\begin{proof}
Note that $0=\E[X]=\E[X^+-X^-] \Rightarrow \E[X^+]=\E[X^-]$. Thus,
$\E[|X|] = \E[X^+] + \E[X^-] = 2\E[X^+]$, as needed. 
\end{proof}

Let $X_1,\dots,X_n$ be independent $\{0,1\}$ random variables with $\mu = \E[\sum_{i=1}^n X_i]$. Then, the Chernoff bound gives \cite[Corollary 1.10]{doerr},
\begin{align}
\Prob[
\sum_{i=1}^n X_i \leq \mu(1-\epsilon)] &\leq e^{-\frac{\epsilon^2 \mu}{2}}, \epsilon \in [0,1]. \label{eq:chernoff} \\
\Prob[\sum_{i=1}^n X_i \geq \mu(1+\epsilon)] &\leq e^{-\frac{\epsilon^2 \mu}{3}}, \epsilon \in [0,1]. \nonumber
\end{align}
A more refined version is given by Azuma's inequality which allows the random variables to admit some mild dependencies. Let $X_1,\dots,X_n$ be $\{0,1\}$ random variables with $\mu = \sum_{i=1}^n \E[X_i|X_1,\dots,X_{i-1}]$. Then,
\begin{align}
	\Prob[
	\sum_{i=1}^n X_i \geq (1+\epsilon)\mu] &\leq e^{-\frac{\epsilon^2 \mu^2}{2n}}, \epsilon \in [0,1]. \label{eq:azuma}
\end{align}
To see this bound, apply \cite[Theorem 1.10.30]{doerr} to the martingale $S_i := \sum\limits_{j=1}^i X_j - \E[X_j|X_1,\dots,X_{j-1}]$.
\begin{lemma}
	If the density function $f_X$ of $X$ is bounded from above by $M$, then:
$$
\Var(X)\geq \frac{1}{12M^2}.
$$
\label{lem:bound_pdf_variance}
\end{lemma}
\begin{proof}
If we want to minimize $\Var(X)=\int_{-\infty}^\infty t^2f_X(t)dt$ under the conditions $f_X\leq M$ and $\Exp[X]=0$, then the the unique minimizer is $f_x=M\cdot \mathbf{1}_{[-\frac{1}{2M},\frac{1}{2M}]}$. We then see,
\begin{align*}
\Var(X)=\int_{-\frac{1}{2M}}^{\frac{1}{2M}} t^2 \cdot Mdt=\left[\frac13 M\cdot  t^3\right]^{\frac{1}{2M}}_{\frac{1}{-2M}}=\frac{1}{12M^2}
\end{align*}
\end{proof}

\subsection{Gaussian and Sub-Gaussian Random Variables}
\label{sec:subg}

If $\mu \in \R^m$ and $\Sigma \succ 0$ is an $m \times m$ positive-definite
matrix, we denote by $\mathcal{N}(\mu, \Sigma)$, the law of the Gaussian with
mean $\mu$ and covariance $\Sigma$. The probability density function of
$\mathcal{N}(\mu,\Sigma)$ is given by $\frac{1}{\sqrt{2\pi}^m
\det(\Sigma)^{1/2}} e^{-\frac{1}{2} (x-\mu)^\T \Sigma^{-1} (x-\mu)}$, $\forall
x \in \R^n$. 

The following is a basic concentration
fact for the norm of the standard Gaussian (see \cite[Lemma
1]{laurent2000adapative}, for example).
\begin{lemma}\label{lem:chisquareconc}
Let $G \sim \mathcal{N}(0,\mathrm{I}_m)$ and let $x \geq 7m$. Then,
$$\Pr\left(\|G\|^2 \geq x\right) \leq e^{-\frac{x}{3}}.$$
\end{lemma}
A random variable $Y \in \R$ is $\sigma$-sub-Gaussian if for all $\lambda \in \R$, we have
\begin{equation}
\label{eq:gauss-mgf}
\E[e^{\lambda Y}] \leq e^{\sigma^2\lambda^2/2}.
\end{equation}
A standard normal random variable $X \sim \mathcal{N}(0,1)$ is
$1$-sub-Gaussian. If variables $Y_1,\dots,Y_k \in \R$ are independent and
respectively $\sigma_i$-sub-Gaussian, $i \in [k]$, then $\sum_{i=1}^k Y_i$ is
$\sqrt{\sum_{i=1}^k \sigma_i^2}$-sub-Gaussian. 

\noindent For a $\sigma$-sub-Gaussian random variable $Y \in \R$ we have the
following standard tail bound \cite[Proposition 2.5.2]{vershynin2018high}:
\begin{equation}
\label{eq:gauss-tail}
\max \{\Prob[Y \leq -\sigma s], \Prob[Y \geq \sigma s]\} \leq e^{-\frac{s^2}{2}}, s \geq 0.
\end{equation}

The following standard lemma shows that bounded random variables are
sub-Gaussian.

\begin{lemma}
\label{lem:disc-sub} Let $X \in [-1,1]$ be a mean-zero random variable. Then $X$ is $1$-sub-Gaussian.
\end{lemma}
\begin{proof}
Let $\phi(x) := e^{\lambda x}$ for $\lambda \in \R$. By convexity of $\phi$,
note that for $x \in [-1,1]$, $\phi(x) \leq \frac{1-x}{2} \phi(-1) +
\frac{1+x}{2} \phi(1)$. Therefore, \begin{align*} \E[\phi(X)] &\leq
\E[\frac{1-X}{2}\phi(-1) + \frac{1+X}{2}\phi(1)] =
\frac{1}{2}(\phi(-1)+\phi(1)) = \frac{1}{2}(e^{-\lambda} + e^{\lambda}) \\ &=
\sum_{i=0}^\infty \frac{\lambda^{2i}}{(2i)!} \leq \sum_{i=0}^\infty
\frac{(\lambda^2/2)^i}{i!}  = e^{\lambda^2/2}, \text{ as needed.}
\end{align*}
\end{proof}
We also need the following fact about truncated sub-Gaussian random variables, which is a
slight generalization of \cite[Lemma 7]{BDHT22}:
\begin{lemma}
\label{lem:subg-pos}
Let $X \in \R$ be $1$-sub-Gaussian. Then $\E[X^+] \leq 1/2$ and $X^+-\E[X^+]$ is $\sqrt{2}$-sub-Gaussian.  
\end{lemma} 
\begin{proof}
Since $X$ is $1$-sub-Gaussian, note that $\E[X] = 0$ and that $\E[X^2] \leq
1$. Therefore, by \Cref{prop:pos-part}, we have $\mu := \E[X^+] = \E[|X|]/2
\leq \E[X^2]^{1/2}/2 \leq 1/2$ by H{\"o}lder. $\sqrt{2}$-sub-Gaussianity of
$X^+-\mu$ now follows verbatim from the proof of \cite[Lemma 5]{BDHT22}
using that $\mu^2 \leq 1/3$ and replacing Gaussian by sub-Gaussian. 
\end{proof}

\subsubsection{Logconcave Measures}

If a measure $\nu$ has a density that is a logconcave function, we call $\nu$
logconcave. Logconcave distributions have many useful analytical properties.
In particular, the marginals of logconcave random vectors are also
logconcave. 

\begin{theorem}[{\cite{Prekopa71}}]
\label{thm:log-marg}
Let $X \in \R^d$ be a logconcave random vector. Then, for any surjective
linear transformation $T: \R^d \rightarrow \R^k$, $TX$ is a logconcave random vector. 
\end{theorem}

The following gives a (essentially tight) bound on the maximum density of any
one dimensional logconcave $\mathbb{R}$ in terms of the variance.
\begin{lemma}[{\cite[Lemma 5.5]{LV07}}]
\label{lem:logconcave-density-bounded}
Let $X\in\mathbb{R}$ be a logconcave variable. Then its density function is upper bounded by $\frac{1}{\sqrt{\Var[X]}}$. 
\end{lemma}

The above has an important consequence. If $X \in \R^n$ is logconcave and
isotropic, then for any vector $v \in \R^n \setminus \{0\}$, the random
variable $v^\T X$ has maximum density at most $1/{\sqrt{\Var[v^\T X]}} =
1/\|v\|_2$, where we have used $v^\T X$ is logconcave. 

By a result of Gr{\"u}nbaum, the mean of logconcave
measure is also an approximate median. In particular, any halfspace
containing the mean has measure at least $1/e$. We will use the following generalization of this result.

\begin{lemma}[{\cite{bertsimas_solving_2004}}] Let $X \in \R^n$ be a logconcave measure with mean $\E[X] = \mu$. Then for any $\theta \in \mathbb{S}^{n-1}$ and $t\in \mathbb{R}$, $\Pr[\theta^\T X \geq \theta^\T\mu -t ] \geq 1/e -|t|$.
\label{lem:grunbaum}
\end{lemma}

We shall require the fact that logconcave random variables satisfy the following comparison inequality.

\begin{lemma}[{\cite{Fradelizi99}}]\label{lem:pos-comp}
Let $X \in \R_+$ logconcave with $\E[X]=\mu$ and let $Z$ have density
$e^{-x}$ (exponential distribution), $x \geq 0$. Then, for any convex
function $\phi: \R_+ \rightarrow \R$, $\E[\phi(X)] \leq \E[\phi(\mu Z)]$. In
particular, 
\begin{enumerate}
\item $\E[X^2] \leq 2\mu^2$.
\item $\E[e^{\lambda X}] \leq \frac{1}{1-\lambda\mu}$, $\lambda < 1/\mu$.  
\end{enumerate}
\end{lemma}

\begin{lemma}\label{lem:lck}
For $X \in \R$ mean-zero and logconcave, we have $\E[X^2] \leq e\E[|X|]^2$.
\end{lemma}
\begin{proof}
Let $X_l := -X \mid X \leq 0$, $p_l = \Pr[X \leq 0]$ and $X_r := X_r \mid X
\geq 0$, $p_r = \Pr[X \geq 0]$. Note that $X_l,X_r$ are both non-negative
logconcave random variables and that $p_l,p_r \geq 1/e$ by
\Cref{lem:grunbaum}. By \Cref{prop:pos-part}, $p_l\E[X_l] =
\E[X^-] = \E[|X|]/2 = \E[X^+] = p_r\E[X_r]$. We now see that 
\begin{align*}
\E[X^2] = p_l \E[X_l^2] + p_r \E[X_r^2] \underbrace{\leq}_{\Cref{lem:pos-comp}} 2(p_l \E[X_l]^2 + p_r \E[X_r]^2) = \E[|X|](\E[X_l]+\E[X_r]) \leq e \E[|X|]^2.
\end{align*}
\end{proof} 
We will also require the following concentration inequality for sums of
non-negative logconcave random variables.
\begin{lemma}
\label{lem:sumlogconcave}
Let $X_1,\dots,X_n \in \R_+$ be independent non-negative logconcave random variables with mean $\mu$. Then, the following holds:
\begin{enumerate}
\item $\Pr[\sum_{i=1}^n X_i \geq (1+\eps) \mu n] \leq e^{-n(\eps-\ln(1+\eps))} \leq e^{-n(\frac{\eps^2}{2(1+\eps)^2})}$, $\eps > 0$.
\item $\Pr[\sum_{i=1}^n X_i \leq (1-\eps) \mu n] \leq e^{-n(-\ln(1-\eps)+\eps)} = e^{-n(\sum_{j=2}^\infty \eps^j/j)}$, $\eps \in [0,1]$. 
\end{enumerate}
\end{lemma}
\begin{proof}
By homogeneity, we assume wlog that $\mu = 1$.

\noindent \textbf{Proof of 1.} Let $\lambda := \frac{\eps}{1+\eps}$. Then, 
\[
\Pr[\sum_{i=1}^n X_i \geq (1+\eps) n] \underbrace{\leq}_{\text{ Markov } } \E[e^{\lambda \sum_{i=1}^n X_i}] e^{-\lambda(1+\eps)n}
\underbrace{\leq}_{\Cref{lem:pos-comp}} (\frac{1}{1-\lambda})^n e^{-\lambda(1+\eps)n} = e^{-n(\eps-\ln(1+\eps))}.
\]

\noindent \textbf{Proof of 2.} Let $\lambda := \frac{\eps}{1-\eps}$. Then, 
\[
\Pr[\sum_{i=1}^n X_i \leq (1-\eps) n] \underbrace{\leq}_{\text{ Markov } } \E[e^{-\lambda \sum_{i=1}^n X_i}] e^{\lambda(1-\eps)n}
\underbrace{\leq}_{\Cref{lem:pos-comp}} (\frac{1}{1+\lambda})^n e^{-\lambda(1+\eps)n} = e^{-n(-\ln(1-\eps)-\eps)}.
\]
\end{proof}
Finally, we will require concentration of truncated sums.
\begin{lemma}
\label{lem:trun-logcon} 
Let $X_1,\dots,X_n \in \R$ be i.i.d. mean zero logconcave random variables with $\E[X_1^+]=\alpha$. Then, for $\eps \in [0,1/2]$, we have that
\begin{enumerate}
\item $\Pr[\sum_{i=1}^n X_i^+ \geq (1+\eps)^2 n\alpha] \leq e^{-\frac{n\eps^2}{3e}} + e^{-\frac{n\eps^2}{2e(1+\eps)}}$.
\item $\Pr[\sum_{i=1}^n X_i^+ \leq (1-\eps)^2 n\alpha] \leq e^{-\frac{n\eps^2}{2e}} + e^{-\frac{n(1-\eps)\eps^2}{2e}}$.
\end{enumerate}
\end{lemma}
\begin{proof}
Define $X_1',\dots,X_n' \in \R_+$ to be i.i.d. copies of $X_1 \mid X_1 \geq
0$ and let $p = \Pr[X_1 \geq 0]$, where $1-1/e \geq p \geq 1/e$
(\Cref{lem:grunbaum}). Let $C = \sum_{i=1}^n 1[X_i \geq 0]$, which a binomial
distribution with parameters $n$ and $p$. Since $X_1,\dots,X_n$ are i.i.d.,
$\sum_{i=1}^n X_i^+$ has the same law as $\sum_{i=1}^C X_i'$. Therefore,
by~\eqref{eq:chernoff} and \Cref{lem:sumlogconcave}, we have that
\begin{align*}
\Pr[\sum_{i=1}^C X_i' \geq (1+\eps)^2n\alpha] &\leq
\Pr[C \geq \ceil{(1+\eps)pn}] + \Pr[\sum_{i=1}^{\floor{(1+\eps)pn}} X_i' \geq (1+\eps)^2n\alpha] \\ 
&\leq e^{-\frac{np \eps^2}{3}} + e^{-\frac{(1+\eps)np\eps^2}{2(1+\eps)^2}} \leq e^{-\frac{n \eps^2}{3e}} + e^{-\frac{n \eps^2}{e(1+\eps)}},
\end{align*}
and
\begin{align*}
\Pr[\sum_{i=1}^C X_i' \leq (1-\eps)^2n\alpha] &\leq
\Pr[C \leq \floor{(1-\eps)pn}] + \Pr[\sum_{i=1}^{\ceil{(1-\eps)pn}} X_i' \leq (1-\eps)^2n\alpha]  \\
&\leq e^{-\frac{np \eps^2}{2}} + e^{-\frac{np(1-\eps)\eps^2}{2}} \leq e^{-\frac{n \eps^2}{2e}} + e^{-\frac{n(1-\eps)\eps^2}{2e}}.
\end{align*}
\end{proof}
\subsection{Khinchine Inequality}

\begin{lemma} 
\label{lem:khinchine}
Let $X_1,\dots,X_n \in \R$ be independent mean zero random variables
satisfying $\E[X_i^4] \leq 3 \E[X_i^2]^2 < \infty$, $\forall i$. Then, for
any scalars $a_1,\dots,a_n \in \R$, we have that $\sqrt{\frac13\E[|\sum_i a_i X_i|^4]}\leq \E[|\sum_i a_i X_i|^2] \leq
3 \E[|\sum_i a_i X_i|]^2$. 
\end{lemma}
\begin{proof}
Letting $Z := |\sum_i a_i X_i| \geq 0$, we wish to show $\sqrt{\Exp[Z^4]}\leq \E[Z^2] \leq 3 \E[Z]^2$. We first show that $\E[Z^4] \leq 3 \E[Z^2]^2$, proving the first part of our claim:
\begin{align*}
	\E[Z^4] &= \sum_{i,j,k,l} a_ia_ja_ka_l \E[X_iX_jX_kX_l] 
	= \sum_i a_i^4 \E[X_i^4] + \sum_{i \neq j} 3 a_i^2a_j^2 \E[X_i^2]\E[X_j^2] \\
	&= \sum_{i} a_i^4(\underbrace{\E[X_i^4]-3\E[X_i^2]^2}_{\leq 0}) + \sum_{i,j} 3a_i^2 a_j^2 \E[X_i^2]\E[X_j^2]
	\leq 3(\sum_i a_i^2\E[X_i^2])^2 = 3 \E[Z^2]^2.
\end{align*}
As a consequence, we have
\[
\E[Z^2] = \E[Z^{2/3} (Z^4)^{1/3}] \underbrace{\leq}_{\text{H{\"o}lder}} \E[Z]^{2/3} \E[Z^4]^{1/3}
\underbrace{\leq}_{\E[Z^4] \leq 3\E[Z^2]^2} 3^{1/3} \E[Z]^{2/3} \E[Z^2]^{2/3},
\]
which, after rearranging, gives $\E[|\sum_i a_i X_i|^2] \leq
3 \E[|\sum_i a_i X_i|]^2$.
\end{proof}

\subsubsection{Discrete Random Variables}

Here we list some of the moments of (DSU) for reference. 

\begin{proposition}[Discrete Symmetric Moments]
\label{prop:dsu}
For $k \geq 1$, let $U$ be uniformly distributed on $\{0,\pm 1/k,\dots,\pm
1\}$. Then, $\E[U^2] = \frac{k+1}{3k} \geq 1/3$, $\E[U^4]
= \frac{(k+1)(3k^2+3k-1)}{15 k^3}$ and $\E[U^4]/\E[U^2]^2 =
\frac{9(3k^2+3k-1)}{15 (k+1)k} \leq 2$.  
\end{proposition}

\subsection{Rejection sampling}
In the proofs of \cref{thm:ipgap_centered,thm:ipgap_packing} we make use of a tool called rejection sampling.
This tool allows us to change the distribution of a random variable $X\sim \mathcal{D}$. It works by `accepting' an evaluation of the variable with a probability that depends on its value. Let us formally define a rejection sampling procedure first.

\begin{definition}
A rejection sampling procedure on a random variable $X$ is a randomized algorithm $\psi$ that, given a realization of $X$, outputs either $\mathsf{reject}$ or $\mathsf{accept}$.
\end{definition}
Now the variable $X$ conditioned on $\psi(X)=\mathsf{accept}$, will be distributed according to a probability $\mathcal{D}'$. We can choose the distribution $\mathcal{D}'$ by appropriately setting $\Pr[\psi(X)=\mathsf{accept}|X]$. This is formalized in the following lemma.

\begin{lemma}
	\label{lem:rejection-sampling-prelim}
	Let $X, Y$ be a random variable on some set $S$, with probability density functions $f_X$ and $f_Y$, and $\Support(Y)\subseteq \Support(X)$. Suppose that $\frac{f_Y(s)}{f_X(s)} \leq K$ for all $s\in S$. Then there exists a rejection sampling procedure $\psi$ with $\Law(X|\psi(X)=\mathsf{accept})=\Law(Y)$ and $\Pr[\psi(X)=\mathsf{accept}]= \frac1K$.
\end{lemma}
\begin{proof}
  Define $\psi$ such that:
  \begin{align*}
\Pr[\psi(X)=\mathsf{accept}|X=x]=\begin{cases}
\frac{f_Y(x)}{K f_X(x)}&x\in \Support(X) \\
0&\text{else}
\end{cases}.
  \end{align*}
This probability is well defined as $\frac{f_Y(x)}{K f_X(x)} \leq \frac{K}{K}=1$ for all $x\in \Support(X)$.  Now $\Pr[\psi(X)=\mathsf{accept}]=\int_{x\in \Support(X)}\frac{f_Y(x)}{Kf_X(x)}f_X(x) dx=\frac{1}{K}$. Call the variable $X$ conditioned $(\psi(X)=\mathsf{accept})$, $Z$.  Now:
	\begin{align*}
		f_Z(z)=\frac{ \frac{f_Y(x)}{K f_X(x)}f_X(x)}{\Pr[\psi(X)=\mathsf{accept}]}=\frac{ f_Y(x)}{K \cdot 1/K}=f_Y(x).
	\end{align*}
This proves the statement.
\end{proof}

\subsection{Combinatorics}
\Cref{thm:ipgap_centered,thm:ipgap_packing} rely on some properties of the optimal dual solution, that hold with high probability. We prove these by taking the union bound over all vectors in $\{0,1\}^n$ that contain at most $\alpha n$ zeroes. By applying the following lemma, we are able to upper bound the number of these vectors by $\exp(H(\alpha)n)$ where $H$ is the entropy function defined as $H(x)=-x\log(x)-(1-x)\log(1-x)$.

\begin{lemma}[{\cite[Theorem 3.1]{galvin_three_2014}}]
For all $\alpha \leq \frac12$ and all $n$,
\begin{align*}
\sum_{i=0}^{\floor{\alpha n}}\binom{n}{i}\leq \exp(H(\alpha)n).
\end{align*}
\label{lem:binomial_bound}
\end{lemma}
 \end{document}